\documentclass[a4paper,11pt,leqno, twoside]{article}
\usepackage{amssymb}
\usepackage{amsthm}
\usepackage[mathscr]{eucal}

\def\NM{{\mathbb{N}}}

\def\GM{{\mathbb{G}}}

\def\QM{{\mathbb{Q}}}
\def\FM{{\mathbb{F}}}
\def\ZM{{\mathbb{Z}}}

\def\CM{{\mathbb{C}}}

\def\SG{{\mathfrak S}}

\def\AC{{\mathcal A}}

\def\CC{{\mathcal C}}
\def\HC{{\mathcal H}}
\def\NC{{\mathcal N}}

\def\RC{{\mathcal R}}

\def\OC{{\mathcal O}}
\def\MC{{\mathcal M}}
\def\KC{{\mathcal K}}
\def\PC{{\mathcal P}}

\def\oQl{\o\QM_{\ell}}
\def\oZl{\o\ZM_{\ell}}
\def\oFl{\o\FM_{\ell}}

\def\ssi{si et seulement si}
\def\para{sous-groupe parabolique }
\def\paras{sous-groupes paraboliques }
\def\levi{sous-groupe de Levi }

\def\tr{{\rm tr}}

\def\simto{\buildrel\hbox{$\sim$}\over\longrightarrow}

\def\leq{\leqslant}
\def\geq{\geqslant}
\def\injo{\hookrightarrow}
\def\id{\mathop{\mathrm{Id}}\nolimits}

\def\ch{\check}

\def\wt{\widetilde}
\def\wh{\widehat}
\def\o#1{\overline{#1}}

\def\l{\lambda}

\def\application#1#2#3#4#5{\begin{array}{rcl}
                            #1 \;\;\; #2 & \to &  #3 \\
                              #4 & \mapsto & #5 
                            \end{array}} 
\def\cas#1#2#3#4#5#6{\begin{array}{rcl} #1 \; & = &
    \left\{\begin{array}{rll} #2 & \hbox{ si } & #3 \\
                             #4 & \hbox{ #5 } & #6 \end{array}
                         \right. \end{array}}

\def\To#1{\buildrel\hbox{\tiny{$#1$}}\over\longrightarrow}
\def\to{\rightarrow}

\def\mod#1{\mathop{#1-\hbox {\sl Mod}}}     %modules {\`a} gauche 
\def\Mod#1{\mathop{\hbox {\sl #1-Mod}}}     %cat des R-modules {\`a} gauche
\def\ker{\mathop{\hbox{\sl ker}\,}}

\def\Hom{\mathop{\hbox{\rm Hom}}\nolimits}
\def\Ext{\mathop{\hbox{\rm Ext}}\nolimits}

\def\Mo#1#2{\mathop{{\rm Mod}_{#1}(#2)}}%cat des RG-modules lisses
\def\Adm#1#2{\mathop{\hbox {\sl Adm}_{#1}(#2)}}%cat des RG-modules adm.
\def\Irr#1#2{\mathop{{\rm Irr}_{#1}\left(#2\right)}}%RG-modules irreductibles
\def\Ind#1#2{\hbox {\sl Ind}_{#1}^{#2}}
\def\ind#1#2#3{\hbox {\sl Ind}_{#1}^{#2}\>\!\left(#3\right)}  %induction
\def\cInd#1#2{\hbox {\rm ind}_{#1}^{#2}}
\def\cind#1#2#3{\hbox {\rm ind}_{#1}^{#2}\>\!\left(#3\right)} %ind a supports compacts
\def\res#1#2#3{\hbox {\sl Res}_{#1}^{#2}\>\!\left(#3\right)}
\def\Res#1#2{\hbox {\sl Res}_{#1}^{#2}}              %restriction   
\def\dim{\mathop{\mbox{\rm dim}}\nolimits}
\def\vol{\mathop{\mbox{\sl vol}}\nolimits}
\def\val{\mathop{\mbox{\sl val}}\nolimits}
\def\rk{\mathop{\mbox{\rm rk}}\nolimits}
\def\ad{\mathop{\mbox{\sl Ad}}\nolimits}

\def \limproj{{\lim\limits_{\longleftarrow}}}

\def\exp{\mathop{\hbox{\sl Exp}\,}}

\makeatletter
\renewcommand{\subsubsection}{\@startsection{subsubsection}{3}{\parindent}{-\baselineskip}{-0.01\baselineskip}{\bf}}
\renewcommand*{\@seccntformat}[1]{%
  \csname the#1\endcsname\
}
\makeatother  

\def\ali{\subsubsection{}\setcounter{equation}{0}}
\def\alin#1{\setcounter{equation}{0}\subsubsection{\it  #1}. }

\newtheoremstyle{th}
  {\baselineskip}{.5\baselineskip}{\itshape}
  {\parindent}{\bf}
  {--}{.5em}{}

\newtheoremstyle{def}
  {\baselineskip}{\baselineskip}{}
  {\parindent}{\bf}
  {--}{.5em}{}

\newtheoremstyle{th*}
  {.5\baselineskip}{.5\baselineskip}{\itshape}
  {\parindent}{\bf}
  {--}{.5em}{}

\newtheoremstyle{remark*}
  {.5\baselineskip}{.5\baselineskip}{}
  {\parindent}{\bf}
  {--}{.5em}{}

\newtheoremstyle{remark}
  {.5\baselineskip}{.5\baselineskip}{}
  {\parindent}{\bf}
  {--}{.5em}{}

\swapnumbers
\theoremstyle{th}
\newtheorem{theo}[subsubsection]{\sc Th{\'e}or{\`e}me.\bf}
\newtheorem{lemme}[subsubsection]{\sc Lemme.\bf}
\newtheorem{prop}[subsubsection]{\sc Proposition.\bf}
\newtheorem{coro}[subsubsection]{\sc Corollaire.\bf}

\theoremstyle{def}

\theoremstyle{remark}

\newtheorem{ques}[subsubsection]{\sc Question.\bf} %\renewcommand{\theques}{}

\theoremstyle{th*}
\newtheorem*{thm}{\sc Th{\'e}or{\`e}me.}
\newtheorem*{lem}{\sc Lemme.}
\newtheorem*{pro}{\sc Proposition.}
\newtheorem*{cor}{\sc Corollaire.}

\newtheorem*{defn}{\sc D\'efinition.}
\newtheorem*{lj}{\sc Langlands-Jacquet.}
\newtheorem*{jl}{\sc Jacquet-Langlands.}

\theoremstyle{remark*}
\newtheorem*{rem}{\sc Remarque.}
\newtheorem*{fait}{\sc Fait.}
\newtheorem*{que}{\sc Question.}
\newtheorem*{exe}{\sc Exemple.}

\newcommand{\findem}{\hfill$\Box$\par\medskip}
\newcommand{\dem}{\indent {\it Preuve :} \rm }

\usepackage[french]{babel}
\usepackage[arrow,matrix]{xy} 
\usepackage[latin1]{inputenc}

\title{Un cas simple de correspondance de Jacquet-Langlands modulo $\ell$}
\author{J.-F. Dat, with an appendix by M.-F. Vignéras}
\date{}

\begin{document}
\maketitle
\bibliographystyle{plain}
\renewcommand{\proofname}{\indent Preuve}

\def\la{\langle}
\def\ra{\rangle}
\def\knr{{\wh{K^{nr}}}}
\def\ka{\wh{K^{ca}}}
\abstract{ Let $G$ be a general linear group over a $p$-adic field and
  let $D^{\times}$ be an anisotropic inner form of $G$.  The
  Jacquet-Langlands correspondence between irreducible complex representations
  of $D^{\times}$ and discrete series of $G$ does not behave well with
  respect to reduction modulo $\ell\neq p$. However we show that the
  Langlands-Jacquet transfer, from the Grothendieck group of
  admissible $\o\QM_{\ell}$-representations of $G$ to that of $D^{\times}$ is
  compatible with congruences and reduces modulo $\ell$ to a similar
  transfer for $\o\FM_{\ell}$-representations, which moreover can be
  characterized by some Brauer 
  characters identities. Studying more carefully this transfer, we 
  deduce a bijection between irreducible $\o\FM_{\ell}$-representations of
  $D^{\times}$ and  ``super-Speh'' $\o\FM_{\ell}$-representations of
  $G$. Via reduction mod $\ell$, this latter bijection is compatible 
with the classical Jacquet-Langlands correspondence composed with
the Zelevinsky involution. Finally we discuss the question whether our
Langlands-Jacquet transfer sends irreducibles to effective
virtual representations up to a sign. This is related to a possible
cohomological realization of this transfer in the Lubin-Tate tower,
 and  presumably boils down
to some unknown properties of parabolic affine Kazhdan-Lusztig polynomials.

In the appendix, we reproduce Vignéras' first construction of Brauer
characters for $p$-adic groups. It follows  Harish Chandra's classical
approach while our construction uses resolutions on the building. 
}

\def\dd{D_d^\times}
\def\mdro{\MC_{Dr,0}}
\def\mdrn{\MC_{Dr,n}}
\def\mdr{\MC_{Dr}}
\def\mlto{\MC_{LT,0}}
\def\mltn{\MC_{LT,n}}
\def\mlt{\MC_{LT}}
\def\mltK{\MC_{LT,K}}
\def\LJ{{\rm LJ}}
\def\JL{{\rm JL}}
\def\SL{{\rm SL}}
\def\GL{{\rm GL}}

\def\Ql{\QM_{\ell}}
\def\Zl{\ZM_{\ell}}
\def\Fl{\FM_{\ell}}

\tableofcontents

\section{Expos\'e des r\'esultats}
Soit $K$ un corps local non-archim\'edien de caract\'eristique r\'esiduelle
$p$, $d$ un entier et $D$ une alg\`ebre \`a division de centre $K$ et de
dimension $d^{2}$ sur $K$. Notons aussi $W_{K}$ le groupe de Weil
d'une cl\^oture alg\'ebrique $\o{K}$ de $K$.

\subsection{Rappels sur Jacquet-Langlands et Langlands-Jacquet classiques}

\alin{Le point de vue de la {fonctorialit\'e de Langlands}} Les groupes
$D^{\times}$ et $G$ ont le m\^eme $L$-groupe $\GL_{d}(\CM)$. Pour $G$,
tous les param\`etres de Langlands  $W_{K}\times \SL_{2}(\CM) \To{}
\GL_{d}(\CM)$ sont \emph{relevants} au sens de \cite[3.3]{BorelCorv}, tandis que pour
$D^{\times}$, seuls ceux qui ne se factorisent pas \`a travers un \para
propre de $\GL_{d}(\CM)$ le sont. Le principe de fonctorialit\'e
de Langlands pr\'edit donc dans ce cas l'existence d'une injection 
$$\JL_{\CM}:\,\Irr{\CM}{D^{\times}}\injo \Irr{\CM}{G}$$
de l'ensemble
$\Irr{\CM}{D^{\times}}$ des classes de repr\'esentations complexes lisses
irr\'eductibles de $D^{\times}$ dans l'ensemble
correspondant pour $G$. Selon le desideratum \cite[10.3(3)]{BorelCorv}, l'image de cette
injection doit co\"{\i}ncider avec l'ensemble des s\'eries discr\`etes de $G$.

\alin{Le point de vue de l'{endoscopie}} Le groupe $G$ est un groupe endoscopique
elliptique pour $D^{\times}$. Cela signifie entre autres que l'on a
une application des classes de conjugaison elliptiques r\'eguli\`eres de
$G$ vers celles de $D^{\times}$, qui dans ce cas est m\^eme une bijection puisque
$G$ est forme int\'erieure de $D^{\times}$. Cette bijection est facile \`a
expliciter puisque des deux c\^ot\'es les classes de conjugaisons
elliptiques r\'eguli\`eres sont param\'etr\'ees par les polyn\^omes unitaires
irr\'eductibles de degr\'e $d$, via l'application ``polyn\^ome caract\'eristique''. 
Ce transfert des classes de conjugaison elliptiques implique un
transfert des caract\`eres virtuels 
% un transfert 
% %$\CC^{\infty}(G^{\rm ell},\CM)^{G}\To{} \CC^{\infty}(D^{\rm ell},\CM)^{D^{\times}}$
% des
% fonctions lisses et invariantes par conjugaison \`a support elliptique
% r\'egulier, puis via la notion de caract\`ere de Harish-Chandra, implique un transfert des
% repr\'esentations virtuelles
que l'on peut r\'esumer par le diagramme
commutatif suivant
%$$ \LJ_\CM:\, \RC(G,\CM)\otimes\CM \To{} \RC(D^{\times},\CM)\otimes \CM $$
$$\xymatrix{\CC^{\infty}(G^{\rm ell},\CM)^{G} \ar[r]^{\sim} &
  \CC^{\infty}(D^{\rm ell},\CM)^{D^{\times}} \\
\RC(G,\CM)\otimes\CM \ar[u]^{\theta^{G}} \ar[r]_{\LJ_{\CM}} &  \RC(D^{\times},\CM)\otimes\CM
\ar[u]_{\theta^{D}} ^{\simeq}
}$$
o\`u $\RC(G,\CM)$ d\'esigne le groupe de Grothendieck des repr\'esentations
lisses complexes de longueur finie de $G$,
$\CC^{\infty}(G^{\rm ell},\CM)^{G}$ l'espace des fonctions lisses et
invariantes par conjugaison sur l'ouvert elliptique de $G$,  $\theta^{G}$ la fonction
caract\`ere de Harish Chandra restreinte aux elliptiques, et idem pour $D^{\times}$.
Le fait que $\theta^{D}$ soit un isomorphisme vient de ce que
$D^{\times}$ est compact modulo son centre, et que l'ouvert elliptique
$ D^{\rm ell}$ y est dense.
Notons enfin que le transfert des fonctions $\CC^{\infty}(G^{\rm ell},\CM)^{G} \To{}
  \CC^{\infty}(D^{\rm ell},\CM)^{D^{\times}}$ est normalis\'e par un
  signe $(-1)^{d+1}$.

% \alin{Les \'enonc\'es des ``correspondances''} Le th\'eor\`eme que
% Jacquet et Langlands ont prouv\'e pour $d=2$ et qui est g\'en\'eralis\'e \`a $d$
% quelconque dans  \cite{DKV} et \cite{Badu1} peut s'\'enoncer comme suit
% :

\ali Venons-en maintenant aux \'enonc\'es des ``correspondances''. \label{corclassic}

\begin{jl}[\cite{DKV},\cite{Badu1}]
  Il existe une injection $\JL_\CM:\,\Irr{\CM}{D^{\times}}\injo
  \Irr{\CM}{G}$ caract\'eris\'ee par les deux propri\'et\'es suivantes :
  \begin{itemize}
  \item Son image est l'ensemble des s\'eries
discr\`etes de $G$.
\item Pour tout $\rho\in\Irr{\CM}{D^{\times}}$, on a $\LJ_\CM [\JL_\CM(\rho)]=[\rho]$.
  \end{itemize}
\end{jl}

% \emph{Il existe une injection $\JL_\CM:\,\Irr{\CM}{D^{\times}}\injo
%   \Irr{\CM}{G}$ caract\'eris\'ee par les deux propri\'et\'es suivantes :
%   \begin{itemize}
%   \item Son image est l'ensemble des s\'eries
% discr\`etes de $G$.
% \item Pour tout $\rho\in\Irr{\CM}{D^{\times}}$, on a $\LJ_\CM [\JL_\CM(\rho)]=[\rho]$.
%   \end{itemize}
% }

Rappelons que la preuve de ce th\'eor\`eme repose sur la formule des
traces ``simple'' et est de nature essentiellement globale. Par
contre, une fois
connu ce r\'esultat, c'est par des arguments purement locaux bas\'es sur
la classification de Zelevinsky que l'on prouve la propri\'et\'e suivante
de $\LJ_\CM$.

\begin{lj}[\cite{Badu2},\cite{lt}, Cor. 2.1.5] {L'application $\LJ_\CM$ envoie $\RC(G,\CM)$ dans
    $\RC(D^{\times},\CM)$. De plus, pour toute $\pi\in\Irr{\CM}{G}$,
    on a $\LJ_\CM([\pi])=0$ ou $\LJ_\CM([\pi])=\pm [\rho]$ pour une
    repr\'esentation $\rho\in\Irr{\CM}{D^{\times}}$.}
\end{lj}

% \emph{L'application $\LJ_\CM$ envoie $\RC(G,\CM)$ dans
%   $\RC(D^{\times},\CM)$. De plus, pour toute $\pi\in\Irr{\CM}{G}$, on
%   a $\LJ_\CM([\pi])=0$ ou $\LJ_\CM([\pi])=\pm [\rho]$ pour une repr\'esentation
%   $\rho\in\Irr{\CM}{D^{\times}}$.}

\subsection{Correspondances  $\ell$-modulaires}

Soit maintenant $\ell\neq p$ un nombre premier. Notre but est de
g\'en\'eraliser autant que possible les \'enonc\'es pr\'ec\'edents aux
$\oFl$-repr\'esentations. Il y a deux approches possibles pour cela.

\alin{Caract\`eres de Brauer}
L'approche directe  se heurte \`a deux difficult\'es. 
D'une part il n'y a pas de
notion \'evidente de s\'erie discr\`ete sur $\oFl$, ce qui laisse perplexe
quant \`a l'image d'une \'eventuelle injection $\JL_{\oFl}$. D'autre part,
bien que Vign\'eras ait adapt\'e la th\'eorie des caract\`eres de
Harish-Chandra 
aux $\oFl$-repr\'esentations, ces caract\`eres ne conduisent qu'\`a un
transfert $\RC(G,\oFl)\otimes\oFl\To{}\RC(D^{\times},\oFl)\otimes\oFl$
et ne peuvent donc pas induire un \emph{unique} transfert \'eventuel 
$\RC(G,\oFl)\To{}\RC(D^{\times},\oFl)$. De m\^eme une formule des traces
modulo $\ell$ ne donne a priori que des informations ``modulo $\ell$''.

C'est pourquoi nous d\'eveloppons une notion de \emph{caract\`ere de
  Brauer}, calqu\'ee sur celle de Brauer pour les groupes finis. Il
s'agit \emph{grosso-modo} d'un homomorphisme
$$  \wt\theta^{G} :\ \RC(G,\oFl)\To{} \CC^{\infty}(G^{\rm ell}_{\ell'},
\oZl)^{G} $$
qui rel\`eve ``canoniquement''  le caract\`ere ordinaire, mais que l'on ne
d\'efinit que sur les \'el\'ements elliptiques ``d'ordre premier \`a $\ell$''.
Nous renvoyons au th\'eor\`eme \ref{theoBrauer} et \`a la proposition
\ref{propBrauer} pour un \'enonc\'e
pr\'ecis de ses propri\'et\'es. Notons que M.-F. Vign\'eras avait d\'ej\`a d\'efini un caract\`ere de
Brauer en \'etendant l'approche de Harish-Chandra. Son manuscrit est
rest\'e non publi\'e depuis 1998 et nous le reproduisons en appendice. Notre
approche est diff\'erente et repose sur \cite{MS1}. En l'\'etat actuel,
elle est aussi plus g\'en\'erale puisque nous d\'efinissons le caract\`ere de
Brauer sur tout \'el\'ement semi-simple r\'egulier compact, et pas seulement
sur les elliptiques r\'eguliers. De plus, notre construction fonctionne
encore lorsque le corps local de
d\'efinition est de caract\'eristique positive.

\alin{Congruences} L'autre approche est de  transporter les correspondances  de $\CM$ \`a
$\oQl$, puis 
d'\'etudier les compatibilit\'es aux
congruences et \`a la r\'eduction ``modulo $\ell$''. 
Un bon point est que le transfert de $\CM$ \`a $\oQl$ se passe bien.
\begin{fait}
  Si l'on admet l'existence d'un isomorphisme\footnote{Heureusement, on n'a
    pas besoin de supposer l'existence de tels isomorphismes pour
    d\'efinir $\LJ_{\oQl}$ et $\JL_{\oQl}$, car les s\'eries discr\`etes de
    $G$, comme les irr\'eductibles de $D^{\times}$, sont ``essentiellement''
    d\'efinies sur la cl\^oture de $\QM$ dans $\CM$.} $\CM\simto
  \o\QM_{\ell}$, les applications $\LJ_{\oQl}$ et $\JL_{\oQl}$ que
  l'on en d\'eduit ne d\'ependent pas du choix de cet isomorphisme, \emph{cf.} \cite[2.1.1]{lt}.
  De plus, elles envoient repr\'esentations
  $\ell$-enti\`eres sur repr\'esentations $\ell$-enti\`eres, \emph{cf.} \ref{lentier}.
\end{fait}

Mais on voit rapidement que $\JL_{\oQl}$ ne peut pas \^etre compatible, via la
r\'eduction modulo $\ell$, avec une application analogue $\JL_{\oFl}$.
% la formulation m\^eme des propri\'et\'es de
% $\JL$ n'est gu\`ere compatible avec les congruences
% En effet, pour que
% $\JL_{\oQl}$ fournisse par r\'eduction modulo $\ell$ une application $\JL_{\oFl}$, 
% il faudrait que les irr\'eductibles $\ell$-enti\`eres de $D^{\times}$ et les
% s\'eries discr\`etes $\ell$-enti\`eres de $G$ soient de r\'eduction
% irr\'eductible, ou tout au moins que l'on puisse faire se correspondre
% les diff\'erents constituents irr\'eductibles. Or, cela est sans
% espoir.
Par exemple, pour $d=2$, $q\equiv -1[\ell]$, la triviale
$1_{\oQl}$ de
$D^{\times}$ correspond \`a la Steinberg ${\rm St}_{\oQl}$, dont la
r\'eduction modulo $\ell$ est de longueur $2$.
C'est pourquoi, \`a rebours du cas complexe, on commence par \'etudier $\LJ$.

\alin{Langlands-Jacquet modulo $\ell$}
La preuve du r\'esultat suivant est donn\'ee au paragraphe
\ref{preuvetheoLJ}. Elle utilise l'existence et les propri\'et\'es du
caract\`ere de Brauer, ainsi que la classification ``\`a la Zelevinsky''
par Vign\'eras des $\oFl$-repr\'esentations irr\'eductibles. 

\begin{thm}%[Thm \ref{theoLJ}]
\label{theoLJ} Il existe un morphisme de groupes ab\'eliens  $\RC(G,\o\FM_{\ell})
\To{\LJ_{\o\FM_{\ell}}} \RC(D^{\times},\o\FM_{\ell})$ rendant
commutatifs les deux diagrammes suivants :
$$
\xymatrix{\CC^{\infty}(G^{\rm ell}_{\ell'},\o\ZM_{\ell})^{G} \ar[r]^{\sim} &
  \CC^{\infty}(D^{\rm ell}_{\ell'},\o\ZM_{\ell})^{D^{\times}} \\
  \RC(G,\o\FM_{\ell}) \ar[u]^{\tilde\theta^{G}}
  \ar@{..>}[r]_{\LJ_{\o\FM_{\ell}}} &  \RC(D^{\times},\o\FM_{\ell})  \ar[u]_{\tilde\theta^{D^{\times}}}
}
\;\,\hbox{ et }\,\; 
\xymatrix{
\RC^{\rm ent}(G,\o\QM_{\ell}) \ar[d]_{r_{\ell}^{G}} \ar[r]^{\LJ_{\oQl}} &
\RC^{\rm ent}(D^{\times},\o\QM_{\ell}) \ar[d]^{r_{\ell}^{D^{\times}}} \\
\RC(G,\o\FM_{\ell})  \ar@{..>}[r]_{\LJ_{\o\FM_{\ell}}} &
\RC(D^{\times},\o\FM_{\ell})  \\
}.
$$
De plus, ce morphisme est uniquement d\'etermin\'e par l'un ou l'autre de
ces diagrammes.
\end{thm}
Ici, l'application $r_{\ell}^{G}$ d\'esigne l'application de r\'eduction
modulo $\ell$, ou ``d\'ecomposition'', dont l'existence est prouv\'ee dans \cite[II.5.11b]{Vig}.
Notons que, contrairement au cas complexe, il est facile de voir que l'image d'une irr\'eductible par $\LJ_{\oFl}$
n'est pas n\'ec\'essairement nulle ou ``irr\'eductible au signe pr\`es''. Par
contre, on peut esp\'erer que la variante suivante est vraie.

\begin{que}\label{questionLJ}
  Soit $\pi\in\Irr{\oFl}{G}$. La repr\'esentation virtuelle
  $\LJ_{\oFl}(\pi)$ est-elle effective au signe pr\`es ?
\end{que}

Nous discutons quelques cas o\`u cette propri\'et\'e est non-trivialement
vraie dans la section \ref{seceff}. Pour y r\'epondre en g\'en\'eral, en admettant
l'analogue modulaire de l'''analogue $p$-adique des conjectures de
Kazhdan-Lusztig'' (fortement plausible), on est ramen\'e \`a examiner
certains signes dans la matrice inverse de certaine matrice de
polyn\^omes de Kazhdan-Lusztig. Mais l'auteur n'a pas su tirer parti de
cette traduction ``concr\`ete''.

\alin{Jacquet-Langlands modulo $\ell$}
Puisque nous disposons maintenant de $\LJ_{\oFl}$, il est naturel de
rechercher $\JL_{\oFl}$ sous la forme d'une section remarquable de
$\LJ_{\oFl}$. Comme on l'a d\'ej\`a mentionn\'e, la section $\JL_{\oQl}$ de $\LJ_{\oQl}$
n'est en g\'en\'eral pas compatible \`a la r\'eduction modulo $\ell$, \emph{cf} l'exemple
\ref{exempleLJ}. En fait, cet exemple montre aussi qu'il
  n'existe pas n\'ecessairement de section de $\LJ_{\oFl}$ envoyant les
  irr\'eductibles sur des irr\'eductibles.
N\'eanmoins, le th\'eor\`eme ci-dessous exhibe
 une section de $\LJ_{\oFl}$ qui envoie les irr\'eductibles sur
des irr\'eductibles au signe pr\`es.
% Ici, ce que nous notons $Z_{\oQl}$ est l'involution de $\Irr{\oQl}{G}$
% que nous prolongeons par lin\'earit\'e \`a $\RC(G,\oQl)$. 

% qui envoie irr\'eductible sur irr\'eductible au signe pr\`es. 

Pour caract\'eriser l'image de cette section, rappelons que
dans la classification de Zelevinsky des repr\'esentations complexes en
termes de multisegments, un r\^ole important est jou\'e par
les repr\'esentations ``de Speh'', qui sont celles associ\'ees aux
segments. Ces repr\'esentations sont \'echang\'ees avec les ``s\'eries
discr\`etes'' par l'involution de Zelevinsky $Z_{\CM}$ sur $\Irr{\CM}{G}$.

De m\^eme, sur $\oFl$, on peut associer une repr\'esentation ``de Speh'' \`a
tout segment cuspidal, et on dit que c'est une repr\'esentation
``\emph{super}Speh'' si le segment est \emph{super}cuspidal,
\emph{cf.} paragraphe \ref{defsuperSpeh}. 
Le r\'esultat suivant est prouv\'e au paragraphe \ref{preuvetheoJL}.

\begin{thm} \label{theoJL}
  Il existe une injection $^{z}\JL_{\oFl}: \, \Irr{\oFl}{D^{\times}}\injo
  \Irr{\oFl}{G}$, caract\'eris\'ee par les propri\'et\'es suivantes :
  \begin{itemize}
  \item son image est l'ensemble des repr\'esentations ``superSpeh''.
  \item pour toute $\rho\in\Irr{\oFl}{D^{\times}}$, on a $\LJ_{\oFl} [{^{z}\JL}_{\oFl}(\rho)]=\pm[\rho]$.
  \end{itemize}
Elle est aussi uniquement d\'etermin\'ee par la relation suivante \`a
$\JL_{\oQl}$. Pour toute $\rho\in\Irr{\oFl}{D^{\times}}$ et tout
rel\`evement $\tilde\rho\in\Irr{\oQl}{D^{\times}}$, la repr\'esentation
$Z_{\oQl}(\JL_{\oQl}(\tilde\rho))$ est enti\`ere,  de r\'eduction irr\'eductible, et on a 
$$ ^{z}\JL_{\oFl}(\rho) = r_{\ell}(Z_{\oQl}(\JL_{\oQl}(\tilde\rho))).$$ 
% son prolongement par lin\'earit\'e s'inscrit dans un
% diagramme commutatif caract\'eristique
% $$
% \xymatrix{
% \RC^{\rm ent}(D^{\times},\o\QM_{\ell}) \ar[d]_{r_{\ell}^{D^{\times}}}
% \ar[r]^{Z_{\oQl}\circ \JL_{\oQl}} &
% \RC^{\rm ent}(G,\o\QM_{\ell}) \ar[d]^{r_{\ell}^{G}} \\
% \RC(D^{\times},\o\FM_{\ell})  \ar[r]_{^{z}\JL_{\o\FM_{\ell}}} &
% \RC(G,\o\FM_{\ell})  \\
% }.
%$$
\end{thm}
Notons que comme $D^{\times}$ est r\'esoluble, le th\'eor\`eme de Fong-Swan
assure que toute $\oFl$-irr\'eductible $\rho$ se rel\`eve. Pr\'ecisons aussi
que la seconde partie du th\'eor\`eme ne signifie pas que la compos\'ee
$Z_{\oQl}\circ\JL_{\oQl}$ est compatible \`a la r\'eduction modulo
$\ell$. Elle ne l'est g\'en\'eralement pas, toujours par l'exemple \ref{exempleLJ}.

N\'eanmoins, la situation rappelle celle de la correspondance de Langlands modulo
$\ell$ puisque, d'apr\`es \cite[1.8]{VigLanglands}, seule la correspondance de
Langlands \emph{compos\'ee avec l'involution de Zelevinsky} a des
propri\'et\'es de  compatibilit\'e partielle \`a la
r\'eduction modulo $\ell$. Comme dans \emph{loc. cit.}, il est donc tentant
d'utiliser l'involution de Zelevinsky-Vign\'eras $Z_{\oFl}$ sur
$\Irr{\oFl}{G}$, voir \cite{VigSheaves}, pour
\emph{d\'efinir} la correspondance de Jacquet-Langlands par la formule :
$$  \JL_{\oFl}:= Z_{\oFl}\circ {^{z}\JL_{\oFl}}. $$

\alin{Correspondance de Langlands modulo $\ell$ pour $D^{\times}$}
Notons 
$$\sigma^{G}_{\oFl}:\, \Irr{\o\FM_{\ell}}{G} \simto {\rm
  Rep}^{d}_{\o\FM_{\ell}}(WD_{K}) $$
 la correspondance de Langlands-Vign\'eras entre $\oFl$-repr\'esentations
 irr\'eductibles de $G$ et 
 $\oFl$-repr\'esentations de Weil-Deligne de dimension $d$, voir
 \cite[1.8]{VigLanglands}. Rappelons qu'une repr\'esentation de
 Weil-Deligne est un couple $\sigma=(\sigma^{\rm ss},N)$ form\'e d'une repr\'esentation
 \emph{semi-simple} $\sigma^{\rm ss}$ ``du'' groupe de Weil $W_{K}$ de $K$ et d'un
 homomorphisme \emph{nilpotent} $N:\sigma^{\rm ss}\To{}\sigma^{\rm ss}(-1)$.

\begin{cor}\label{coroLD}
  L'application $ \sigma^{D^{\times}}_{\oFl}:= \sigma^{G}_{\oFl}\circ
  \JL_{\oFl}$ est une bijection 
$ \Irr{\o\FM_{\ell}}{D^{\times}} \simto {\rm
  Rep}^{d}_{\o\FM_{\ell}}(WD_{K})^{\rm indec} $
entre $\oFl$-repr\'esentations irr\'eductibles de $D^{\times}$ et
 $\oFl$-repr\'esentations de Weil-Deligne ind\'ecomposables de dimension $d$, 
caract\'eris\'ee par la propri\'et\'e suivante : si
$\rho\in\Irr{\o\FM_{\ell}}{D^{\times}}$ et $\tilde\rho$ est un
rel\`evement de $\rho$ \`a $\o\QM_{\ell}$, alors $\sigma^{D^{\times}}_{\o\FM_{\ell}}(\rho) = R_{\ell}(\sigma_{\oQl}^{D^{\times}}(\tilde\rho))$.
\end{cor}

Ici, $\sigma_{\oQl}^{D^{\times}}=\sigma^{G}_{\oQl}\circ \JL_{\oQl}$
est la correspondance de Langlands classique pour $D^{\times}$, et
$R_{\ell}$ est l'application de r\'eduction modulo $\ell$ des (classes
de) $\oQl$-repr\'esentaitions de Weil-Deligne enti\`eres, implicite dans
\cite[1.8]{VigLanglands}, et expliqu\'ee dans \cite[4.1.8]{lefschetz}.

La morale de ce corollaire est qu'une certaine forme du principe de
fonctorialit\'e semble \^etre envisageable pour les
$\oFl$-repr\'esentations. Mais ce principe de fonctorialit\'e, s'il
existe, fera plut\^ot intervenir le ${\rm SL}_{2}$ d'Arthur que celui de
Deligne-Langlands.

\section{Caract\`eres de Brauer et groupes de Grothendieck}

Dans cette section, nous construisons le {caract\`ere de
  Brauer} d'une $\oFl$-repr\'esentation lisse de longueur finie d'un
groupe r\'eductif $p$-adique quelconque, et nous sp\'eculons sur les
propri\'et\'es que l'on peut en attendre, par analogie avec les groupes finis.
Puis nous rappelons les classifications de Zelevinsky et Vign\'eras des
repr\'esentations irr\'eductibles de $\GL_{d}(K)$ pour
en d\'eduire la surjectivit\'e de la r\'eduction modulo
$\ell$ entre groupes de Grothendieck, qui est un point clef de la
construction de $\LJ_{\oFl}$. Enfin nous prouvons quelques propri\'et\'es
des repr\'esentations de $D^{\times}$ utilis\'ees dans la construction de $^{z}\JL_{\oFl}$.

\subsection{Caract\`ere de Brauer}

\alin{Rappels terminologiques}
Dans cette section, $G$ d\'esigne un groupe r\'eductif
$p$-adique quelconque, et $Z(G)$ d\'esigne son centre. Un \'el\'ement $\gamma$ de $G$ est dit
\emph{semi-simple r\'egulier} si son centralisateur connexe
$Z_{G}(\gamma)^{\circ}$ est un tore de $G$. Lorsque ce tore est 
  anisotrope modulo $Z(G)$, on dit que $\gamma$ est
  \emph{elliptique}. Par exemple, pour $G=\GL_{d}(K)$ ou $D^{\times}$,
  l'\'el\'ement $\gamma$ est semi-simple r\'egulier \ssi\ son polyn\^ome
  minimal est s\'eparable de degr\'e $d$, et il est de plus elliptique \ssi\ ce
  polyn\^ome minimal est de plus irr\'eductible.

Rappelons aussi qu'un \'el\'ement $\gamma$ de $G$ est \emph{compact} si le
sous-groupe qu'il engendre est relativement compact. Si $Z$ est un
sous-groupe ferm\'e central de $G$ on dira que $\gamma$ est
 \emph{compact modulo $Z$} si son image dans $G/Z$ est un \'el\'ement compact.
 De mani\`ere \'equivalente, $\gamma$ est
compact, resp. compact modulo $Z(G)$, s'il fixe un point de
l'immeuble \'etendu de $G$, resp. de l'immeuble semi-simple (\emph{i.e.}
non \'etendu) $X$ de $G$. En cons\'equence, un \'el\'ement compact modulo le
centre normalise des sous-groupes ouverts compacts arbitrairement petits.

Un \'el\'ement elliptique $\gamma$ est compact modulo le centre et
l'ensemble des points fixes $X^{\gamma}$ de $\gamma$ dans l'immeuble
de $G$ est compact. R\'eciproquement, un \'el\'ement semi-simple r\'egulier
dont l'ensemble des points fixes $X^{\gamma}$ est compact, est elliptique.

L'ensemble $G^{\rm rs}$  des \'el\'ements semi-simples r\'eguliers de $G$
est ouvert dans $G$ et stable par conjugaison. Pour tout anneau
$\Lambda$, on note
$\CC^{\infty}(G^{\rm rs},\Lambda)^{G}$ le $\Lambda$-module des
fonctions \`a valeurs dans $\Lambda$ sur $G^{\rm rs}$ qui sont  localement
constantes et invariantes par conjugaison. On utilisera des notations
similaires $\CC^{\infty}(G^{\rm crs},\Lambda)^{G}$ et
$\CC^{\infty}(G^{\rm ell},\Lambda)^{G}$ o\`u les notations $G^{\rm crs}$
et $G^{\rm ell}$ d\'esignent respectivement les sous-ensembles ouverts
de $G$ form\'es des \'el\'ements semi-simples r\'eguliers et compacts modulo
le centre, resp. elliptiques.

Si $\pi$ est une repr\'esentation admissible de $G$ sur un corps, les
op\'erateurs $\pi(\mu)$, pour $\mu$ dans l'alg\`ebre de Hecke de $G$, sont \`a
image finie. Moyennant le choix d'une mesure de Haar $dg$ sur $G$, on
obtient un caract\`ere-distribution $f\mapsto {\rm tr}(\pi(fdg))$ sur
$G$. Lorsque le corps de coefficients est $\CM$, Harish Chandra a montr\'e que
cette distribution est donn\'ee par int\'egration contre une fonction
invariante lisse d\'efinie sur $G^{\rm rs}$. Pour un corps de
coefficients de caract\'eristique diff\'erente de $p$, Vign\'eras et
Waldspurger ont \'etendu les arguments de Harish Chandra, puis
r\'ecemment
 Meyer et Solleveld ont trouv\'e une approche diff\'erente de ce r\'esultat. 

Rappelons que $\RC(G,\Lambda)$ d\'esigne le groupe de Grothendieck des
repr\'esentations $\Lambda$-admissibles de longueur finie.

\begin{theo}[Caract\`ere ordinaire, \cite{MS2}, Thms 7.2 et 7.4 et
  \cite{VigWalds}, Thm. E.4.4] Supposons $\Lambda=\o\FM_{\ell}$ ou $\o\QM_{\ell}$.
  Il existe un unique homomorphisme 
$$\RC(G,\Lambda)\To{\theta}
\CC^{\infty}(G^{\rm rs},\Lambda)^{G}$$  appel\'e \emph{caract\`ere
  ordinaire}, tel que pour toute $\pi\in\Irr{\Lambda}{G}$ et tout
$\gamma\in G^{\rm rs}$ il existe  un sous-groupe
  ouvert compact $H_{\pi,\gamma}$ tel que pour tout pro-$p$-sous-groupe ouvert
   $H\subset H_{\pi,\gamma}$, on ait l'\'egalit\'e
   \begin{equation}
  \theta_{\pi}(\gamma)={\rm
    tr}(\pi(\varepsilon_{H}*\gamma*\varepsilon_{H}))\label{caractheta}
  \end{equation}
 o\`u $\varepsilon_{H}$ d\'esigne la  mesure de Haar normalis\'ee de $H$ (un
 idempotent de l'alg\`ebre de   Hecke de $G$).
\end{theo}

M.-F. Vign\'eras a propos\'e une d\'efinition, reproduite en appendice, du caract\`ere de Brauer d'une
 $\o\FM_{\ell}$-repr\'esentation  sur certains \'el\'ements elliptiques.
Nous allons ici utiliser les techniques de Meyer et Solleveld pour
d\'efinir un caract\`ere de Brauer sur certains \'el\'ements semi-simples
r\'eguliers et compacts modulo le
centre.

\alin{Rappel sur la notion de caract\`ere de Brauer}
Soit $\mu^{\ell'}(\Lambda)$ le groupe des racines de
l'unit\'e de l'anneau $\Lambda$ qui sont d'ordre premier \`a $\ell$.
L'homomorphisme de r\'eduction $r_{\ell}: \,
\o\ZM_{\ell}\To{}\o\FM_{\ell}$ induit une bijection $\mu^{\ell'}(\o\ZM_{\ell}) \To{}
\mu^{\ell'}(\o\FM_{\ell})$ dont on note $\iota$
la bijection r\'eciproque.
% admet des sections, et nous en
% fixons une que nous notons $\iota$ (l'ensemble de ces sections est principal homog\`ene sous le
% groupe des automorphismes du corps $\o\FM_{\ell}$).
On la prolonge en $0$ en posant $\iota(0):=0$.

 Soit $V$ un $\o\FM_{\ell}$-espace vectoriel et $\rho$ un
 endomorphisme de $V$ tels que
\begin{enumerate}
\item[B1)]  $\rho$ satisfait l'\'egalit\'e $\rho^{n+1}=\rho$ pour un entier $n$ premier \`a
  $\ell$,
\item[B2)] l'image de $\rho$ est de dimension finie. %, $\dim(\im\rho)<\infty$.
\end{enumerate}

Alors pour tout sous-espace $W$ de $V$ contenant
 $\rho(V)$, l'endomorphisme $\rho_{|W}$ de $W$ est diagonalisable \`a
 valeurs propres dans $\mu^{\ell'}(\o\FM_{\ell})\cup\{0\}$. De plus,
 la multiplicit\'e des valeurs propres non nulles est ind\'ependante de $W$.
Ainsi, notant
 ${\rm Sp}(\rho_{|W})$ le multi-ensemble de ces valeurs propres compt\'ees
 avec leurs multiplicit\'es, la somme 
$$ \tilde{\rm tr}(\rho):=\sum_{\zeta\in{\rm Sp}(\rho_{|W})} \iota(\zeta)
\in\o\ZM_{\ell}$$
 est ind\'ependante du choix de $W$ puisque $\iota(0)=0$.
Les propri\'et\'es \'el\'ementaires de cette ``trace de Brauer'' sont r\'esum\'ees
dans le lemme suivant :

\begin{lem} \label{lemmeperm} Avec les notations ci-dessus, 
  \begin{enumerate}
  \item On a l'\'egalit\'e $r_{\ell}(\tilde{\rm tr}(\rho)) = {\rm tr}(\rho)$.
  \item  Si $\tilde{\rho}$ est un endomorphisme %satisfaisant
                                %$\tilde\rho^{n+1}=\tilde\rho$ 
d'un $\o\QM_{\ell}$-espace vectoriel $\tilde{V}$ satisfaisant les
propri\'et\'es B1) et B2) ci-dessus, et dont l'image
contient un $\o\ZM_{\ell}$-r\'eseau $\omega$ stable sous $\tilde{\rho}$
de r\'eduction isomorphe \`a $(V,\rho)$,
%que la r\'eduction de $(\omega,\tilde\rho)$ soit isomorphe \`a $(V,\rho)$,
alors $\tilde{\rm tr}(\rho)={\rm tr}(\tilde\rho)$.
\item Supposons qu'il existe une d\'ecomposition $V=\bigoplus_{i\in
      I}V_{i}$ et une permutation $\sigma$ de $I$ \emph{sans points
      fixes} telle que $\forall i\in I, \, \rho(V_{i})\subset
    V_{\sigma(i)}$. Alors $\tilde\tr(\rho)=0$.
  \item Si $ 0 \To{} (V_{1},\rho_{1}) \To{} (V_{2},\rho_{2})\To{}
    (V_{3},\rho_{3})\To{}0 $ est une suite exacte courte, alors
    $\sum\limits_{i=1}^{3} (-1)^{i}\tilde\tr(\rho_{i})=0$.
  \end{enumerate}
\end{lem}
\begin{proof}
  Les propri\'et\'es i) et iv) sont imm\'ediates. La propri\'et\'e ii) vient du
  fait que $\omega$ est somme directe de sous-$\o\ZM_{\ell}$-modules
  propres pour $\rho$ puisque l'\'equation $X^{n+1}-X$ est s\'eparable sur
  $\o\FM_{\ell}$. Prouvons la propri\'et\'e iii). On a $\rho(V)=\sum_{i\in
    I}\rho(V_{i})\subset \sum_{i\in I}(\rho(V)\cap
  V_{\sigma(i)})
$. Comme $\tilde\tr(\rho)=\tilde\tr(\rho_{|\rho(V)})$,
  on est ainsi ramen\'e au cas o\`u $V=\rho(V)$ est de dimension finie et
  $\rho$ est un automorphisme d'ordre $n$ premier \`a $\ell$. Dans ce
  cas $\rho$ induit un isomorphisme $V_{i}\simto V_{\sigma(i)}$ pour
  tout $i$. Choisissant une base de chaque $V_{i}$ pour $i$ dans un
  syst\`eme de repr\'esentants des orbites de $\sigma$ dans $I$, on se
  ram\`ene au cas o\`u les $V_{i}$ sont de dimension $1$.
 Enfin, puisque $\rho$ stabilise $\sum_{n\in
    \NM}V_{\sigma^{n}(i)}$, on peut supposer que $\sigma$ a une seule orbite dans $I$.
Dans ce cas, $d:={\rm dim}(V)$ divise $n$ et $\rho$ est la  matrice de
permutation circulaire d'une base de $V$. Ses valeurs propres sont les
puissances d'une racine $d$-\`eme de l'unit\'e et ont multiplicit\'e
$1$. Leur somme est donc nulle et il
s'ensuit que $\tilde\tr(\rho)=0$.
\end{proof}

% On remarque que $r_{\ell}(\tilde{\rm tr}(\rho)) = {\rm tr}(\rho)$. De
% plus, si $\tilde{\rho}$ est un endomorphisme satisfaisant $\tilde\rho^{n+1}=\tilde\rho$ d'un
% $\o\QM_{\ell}$-espace vectoriel contenant un
% $\o\ZM_{\ell}$-sous-module $\omega$, sans $\o\QM_{\ell}$-droite, stable sous $\tilde{\rho}$, et tel
% que la r\'eduction de $(\omega,\tilde\rho)$ soit isomorphe \`a $(V,\rho)$,
% alors $\tilde{\rm tr}(\rho)={\rm tr}(\tilde\rho)$.
% Le lemme suivant est \'el\'ementaire mais nous sera tr\`es utile.

% \begin{lem} \label{lemmeperm}
%   Supposons qu'il existe une d\'ecomposition $V=\sum_{i\in I}V_{i}$ et
%   une permutation $\sigma$ de $I$ \emph{sans points fixes} telle que
%   $\forall i\in I, \, \rho(V_{i})\subset V_{\sigma(i)}$. Alors $\tilde\tr(\rho)=0$.
% \end{lem}

\begin{rem}
  L'hypoth\`ese $n$ premier \`a $\ell$ dans B1) n'est pas n\'ecessaire pour
  d\'efinir $\tilde\tr(\rho)$. Sans cette hypoth\`ese les propri\'et\'es i) et
  iv) ci-dessus sont encore vraies mais les propri\'et\'es ii) et iii) ne
  le sont g\'en\'eralement plus.
\end{rem}

On s'int\'eresse au cas o\`u $V$ est l'espace d'une $\o\FM_{\ell}$-repr\'esentation
 $\pi$ de $G$ et $\rho = \pi(\gamma*\varepsilon_{H})$ o\`u $H$
est un pro-$p$-sous-groupe ouvert normalis\'e par un \'el\'ement 
$\gamma$ de $G$. Ceci implique en particulier que $\gamma$ est compact
modulo le centre. La condition B2) est v\'erifi\'ee si $\pi$ est suppos\'ee
\emph{admissible}. Pour s'assurer que la condition B1) le soit on
introduit les notions suivantes :
\begin{defn} Soit $Z$ un sous-groupe ferm\'e  central de $G$.
Un \'el\'ement $\gamma$ de $G$ est dit \emph{d'ordre premier \`a $\ell$
  modulo $Z$}, s'il satisfait \`a 
l'une des deux conditions \'equivalentes suivantes :
\begin{enumerate}
\item il existe un  pro-$p$-sous-groupe ouvert $H$ et un entier $n$ premier \`a
$\ell$ tel que $\gamma^{n}\in H.Z$.
\item l'adh\'erence du sous-groupe engendr\'e par $g$ dans
  $G/Z$ est un groupe profini de pro-ordre premier \`a $\ell$.
\end{enumerate}
\end{defn}
% \begin{defn} Soit $Z$ un sous-groupe central de $G$.
% Un \'el\'ement $\gamma$ de $G$ est dit \emph{d'ordre premier \`a $\ell$},
% resp. \emph{d'ordre premier \`a $\ell$ modulo $Z$}, s'il v\'erifie
% l'une des deux conditions \'equivalentes suivantes :
% \begin{enumerate}
% \item il existe un  pro-$p$-sous-groupe ouvert $H$ et un entier $n$ premier \`a
% $\ell$ tel que $\gamma^{n}\in H$, resp. $\gamma^{n}\in H.Z$.
% \item l'adh\'erence du sous-groupe engendr\'e par $g$ dans $G$, resp. dans
%   $G/Z$ est un groupe profini de pro-ordre premier \`a $\ell$.
% \end{enumerate}
% \end{defn}
L'\'equivalence entre les deux conditions d\'ecoule de la pro-nilpotence des
pro-$p$-groupes. Reprenant la discussion pr\'ec\'edente, on remarque que
la condition B2) est  v\'erifi\'ee par $\rho=\pi(\gamma*\varepsilon_{H})$
lorsqu'il existe un sous-groupe ferm\'e central $Z$ de $G$ tel que 
\begin{itemize}
\item $\gamma$ est d'ordre premier \`a $\ell$ modulo $Z$
\item  $\pi$ est \emph{$Z$-semisimple} au sens o\`u
  $\pi(\o\FM_{\ell}[Z])\subset {\rm End}_{\o\FM_{\ell}}(V)$ est
  une alg\`ebre semi-simple.
%$\pi(Z)\subset \o\FM_{\ell}^{\times}$, \emph{i.e.} $Z$ agit par
%  homoth\'eties sur $(\pi,V)$. 
\end{itemize}
Par exemple, ces conditions sont remplies si $\gamma$ est d'ordre
premier \`a $\ell$ et $\pi$ admissible, ou si $\gamma$ est d'ordre
premier \`a $\ell$ modulo le centre et $\pi$ admet un caract\`ere central.
 Dans chacun de ces cas on a alors l'\'egalit\'e
$$ \tilde\tr(\pi(\gamma*\varepsilon_{H})) = \tilde\tr(\pi(\gamma)_{|V^{H}}).$$

Nous noterons $G^{\rm crs}_{\ell'/Z}$, resp. $G^{\rm ell}_{\ell'/Z}$
l'ensemble des \'el\'ements compacts modulo $Z$, resp. elliptiques,
de pro-ordre premier \`a $\ell$ modulo $Z$ au sens ci-dessus. Ces ensembles sont
clairement stables par conjugaison et ouverts.

\begin{theo}[Caract\`ere de Brauer] \label{theoBrauer}
Soit $(\pi,V)$ une $\o\FM_{\ell}$-repr\'esentation de longueur finie de $G$ et $Z$ un
sous-groupe central ferm\'e de $G$ tel que $\pi$ soit $Z$-semisimple. Il
existe une fonction $$\tilde\theta_{\pi}\in \CC^{\infty}(G^{\rm
  crs}_{\ell'/Z},\o\ZM_{\ell})^{G},$$ appel\'ee \emph{caract\`ere de Brauer
de $\pi$}, et
caract\'eris\'ee par la propri\'et\'e suivante :
pour tout $\gamma\in G^{\rm crs}_{\ell'/Z}$ il existe  un sous-groupe
  ouvert compact $H_{\pi,\gamma}$ tel que pour tout pro-$p$-sous-groupe ouvert
   $H\subset H_{\pi,\gamma}$ normalis\'e par $\gamma$, on ait l'\'egalit\'e
   \begin{equation}
  \tilde\theta_{\pi}(\gamma)=\tilde{\rm
    tr}(\pi(\gamma)_{|V^{H}}).\label{caracthetatilde}
  \end{equation}
\end{theo}

\begin{proof}
% Pour prouver l'existence de $\tilde\theta_{\pi}$,
Nous allons utiliser les r\'esultats d'acyclicit\'e de \cite{MS1}
g\'en\'eralisant ceux de \cite{SS2}. Rappelons de quoi il s'agit.
%Nous suivons la preuve du th\'eor\`eme 7.2 de \cite{MS2}. 
Nous avons besoin de la structure polysimpliciale de l'immeuble semi-simple $X$ de
$G$ et  noterons pour cela $X_{q}$ l'ensemble des facettes de
dimension $q$ de $X$. Le stabilisateur d'une facette $\sigma$ est not\'e
$P_{\sigma}^{\dag}$. Nous utiliserons aussi
les pro-$p$-sous-groupes ouverts distingu\'es $U_{\sigma}^{(e)}$ de
$P_{\sigma}^{\dag}$, pour $e\geq 1$ entier, introduits
par Schneider et Stuhler dans \cite[I.2.7]{SS2}.
 Le th\'eor\`eme 2.4 de \cite{MS1} affirme que pour toute  repr\'esentation
 lisse $(\pi, V)$ de $G$ sur laquelle $p$ est inversible, et 
tout
sous-complexe polysimplicial \emph{convexe} $\Sigma$ de $X$, 
le complexe cellulaire $C_{*}(\Sigma,
\sigma\mapsto V^{U_{\sigma}^{(e)}})$ est une r\'esolution du sous-espace
$V_{\Sigma}:= \sum_{x\in \Sigma_{0}} V^{U_{x}^{(e)}}$. Le cas
$\Sigma=X$ remonte \`a \cite[II.3.1]{SS2}.

% Enfin nous noterons $X^{\gamma}$ le sous-ensemble polysimplicial
% convexe de $X$ form\'e des facettes  stables par $\gamma\in G$.

%Fixons une $\o\FM_{\ell}$-repr\'esentation de longueur finie $(\pi,V)$ et

Choisissons $e$ assez grand pour que $V=\sum_{x\in X_{0}}
V^{U_{x}^{(e)}}$. 
Fixons aussi un \'el\'ement $\gamma\in G$, semi-simple
r\'egulier et compact modulo le centre.
Soit $H$ un pro-$p$-sous-groupe ouvert normalis\'e par
$\gamma$. On peut alors
 trouver un sous-complexe polysimplicial convexe \emph{fini}
$\Sigma$ de $X$, stable par $H$ et $\gamma$, et  tel que $V^{H}\subset V_{\Sigma}:=\sum_{x\in \Sigma_{0}}
V^{U_{x}^{(e)}}$. Pour un tel $\Sigma$, le sous-espace $V_{\Sigma}$
est de dimension finie, muni d'une action du groupe
$\langle H,\gamma\rangle$ engendr\'e par $H$ et $\gamma$, et le
complexe cellulaire $C_{*}(\Sigma,
\sigma\mapsto V^{U_{\sigma}^{(e)}})$ est une r\'esolution $\langle
H,\gamma\rangle$- \'equivariante de $V_{\Sigma}$.
% D'apr\`es le th\'eor\`eme 2.4 de \cite{MS1}, pour tout
% sous-complexe polysimplicial convexe fini $\Sigma$ de $X$, contenant
% $\Sigma_{0}$ et stable par $H$ et $\gamma$, le complexe cellulaire $C_{*}(\Sigma,
% \sigma\mapsto V^{U_{\sigma}^{(e)}})$ est une r\'esolution de
% $V_{\Sigma}$, \'equivariante sous le groupe $\langle H,\gamma\rangle$
% engendr\'e par $H$ et $\gamma$.
Supposons de plus $\gamma$ d'ordre premier \`a $\ell$ modulo $Z$. C'est
encore le cas de tout \'el\'ement  de $\langle H,\gamma\rangle$, et tout
tel \'el\'ement agit donc par un automorphisme de degr\'e premier \`a $\ell$
sur $V_{\Sigma}$.
En utilisant successivement les
propri\'et\'es iv) et iii) du lemme
\ref{lemmeperm}, on obtient pour tout $  g\in \langle
H,\gamma\rangle$,  
\begin{equation}\label{ep}
\tilde\tr\left(\pi(g)_{|V_{\Sigma}}\right)  =  \sum_{\sigma\in \Sigma^{g}} (-1)^{\rm
  dim \sigma}\epsilon_{\sigma}(g) \tilde\tr\left(\pi(g)_{|V^{U_{\sigma}^{(e)}}}\right). 
\end{equation}
Ici $\Sigma^{g}$ d\'esigne l'ensemble des facettes $\sigma$ de $\Sigma$
qui sont stables par $g$, \emph{i.e.}
telles que $g\in P_{\sigma}^{\dag}$. De plus, 
$\epsilon_{\sigma}:\, P_{\sigma}^{\dag}\To{}\{\pm 1\}$ est le
caract\`ere d\'ecrivant l'action de $P_{\sigma}^{\dag}$ sur les orientations de
$\sigma$. 

Par ailleurs, soit $U_{\Sigma}^{(e)}:=\bigcap_{\sigma\in \Sigma}
U_{\sigma}^{(e)}$. C'est un pro-$p$-sous-groupe ouvert de $G$ qui agit
trivialement sur $V_{\Sigma}$ et dont l'intersection avec $\langle
H,\gamma\rangle$ est distingu\'ee dans celui-ci.
D'apr\`es le lemme \ref{pgroupe} appliqu\'e \`a $\Gamma:=\langle
H,\gamma\rangle / (\langle H,\gamma\rangle\cap U_{\Sigma}^{(e)}), P:= H/(H\cap U_{\Sigma}^{(e)})$ et $\gamma$, on a
\begin{equation}  \label{moy}
 \tilde{\rm tr}\left(\pi(\gamma)_{|V^{H}}\right) 
 = \frac{1}{[H:(H\cap U_{\Sigma}^{(e)})]} \sum_{h\in H/(H\cap
   U_{\Sigma}^{(e)})} \tilde\tr(\pi(h\gamma)_{|V_{\Sigma}}).
\end{equation}

% Soit alors $\Sigma_{q}/H$ l'ensemble des orbites
% de $H$ dans $\Sigma_{q}$. On a 
%  $$C_{q}(\Sigma,\sigma\mapsto
% V^{U_{\sigma}^{(e)}})^{H}\simeq \bigoplus_{[\sigma]\in \Sigma_{q}/H}
% \cind{H\cap P_{\sigma}^{\dag}}{H}{\epsilon_{\sigma}\otimes V^{U_{\sigma}^{(e)}}}^{H}\simeq 
% \bigoplus_{[\sigma]\in \Sigma_{q}/H}
% \epsilon_{\sigma}\otimes V^{U_{\sigma}^{(e)}.(H\cap
%   P_{\sigma}^{\dag})},$$
% o\`u $\epsilon_{\sigma}:\, P_{\sigma}^{\dag}\To{}\{\pm 1\}$ est le
% caract\`ere d\'ecrivant l'action de $P_{\sigma}$ sur les orientations de
% $\sigma$.
% Pour une facette $\sigma$ dont la $H$-orbite est $\gamma$-stable,
% choisissons un \'el\'ement $h_{\sigma}\in H$ tel que $h_{\sigma}\gamma(\sigma)=\sigma$.
% % Comme dans la preuve du th\'eor\`eme
% % 4.1 de \cite{MS1}, o
% Utilisant les points iv) et  iii) du lemme \ref{lemmeperm}, on obtient alors
% la formule suivante :
% \begin{equation} \label{forcar}
%  \tilde{\rm tr}\left(\pi(\gamma)_{|V^{H}}\right) 
% % & = & \sum_{\sigma\in
% %   \Sigma/H} (-1)^{\rm dim \sigma}\epsilon_{\sigma}(\gamma)
% % \tilde{\rm tr}\left(\pi(\gamma)_{|V^{U_{\sigma}^{(e)}(H\cap
% %       P_{\sigma}^{\dag}})}\right) \\
%  =  \sum_{[\sigma]\in
%   (\Sigma/H)^{\gamma}} (-1)^{\rm dim \sigma}\epsilon_{\sigma}(h_{\sigma}\gamma)
% \tilde{\rm tr}\left(\pi(h_{\sigma}\gamma)_{|V^{U_{\sigma}^{(e)}(H\cap
%       P_{\sigma}^{\dag}})}\right).
% \end{equation}

Supposons momentan\'ement que $\gamma$ est \emph{elliptique}, de sorte
que $X^{\gamma}$ est un complexe polysimplicial fini. Posons
$H_{\pi,\gamma} := \bigcap_{\sigma \in X^{\gamma}}
U_{\sigma}^{(e)}$. C'est un pro-$p$-sous-groupe ouvert de $G$
normalis\'e par $\gamma$, et qui fixe un voisinage de
$X^{\gamma}$ dans $X$ (car $e\geq 1$). Par cons\'equent on a
$X^{h\gamma}=X^{\gamma}$ pour tout $h\in H_{\pi,\gamma}$.
De plus, 
% Alors pour tout $H \subset H_{\pi,\gamma}$
% normalis\'e par $\gamma$, on a 
% $$ H\gamma \subset \{\delta\in G,\, X^{\delta}=X^{\gamma}\}.$$
% En particulier, $H$ agit trivialement sur $X^{\gamma}$, et les seules
% $H$-orbites $\gamma$-stables de $X$ sont les singletons inclus dans
% $X^{\gamma}$. En d'autres termes, $(X/H)^{\gamma}=X^{\gamma}$.
% Comme de plus on a $H\cap P_{\sigma}^{\dag}\subset U_{\sigma}^{(e)}$,
% on obtient la formule
% $$
%  \tilde{\rm tr}\left(\pi(\gamma)_{|V^{H}}\right) =
%  \tilde\theta_{\pi}(\gamma): =  \sum_{\sigma\in
%   X^{\gamma}} (-1)^{\rm dim \sigma}\epsilon_{\sigma}(\gamma)
% \tilde{\rm tr}\left(\pi(\gamma)_{|V^{U_{\sigma}^{(e)}}}\right)
% $$
% dont le terme de droite est visiblement ind\'ependant de $H$.
comme $H_{\pi,\gamma}$ agit trivialement sur $V^{U_{\sigma}^{(e)}}$ pour
toute facette $\sigma\in X^{\gamma}$, la formule (\ref{ep}) montre que
pour tout $\Sigma$ tel que $V_{\Sigma}\supset V^{H_{\pi,\gamma}}$, on a
$$ \forall h\in H_{\pi,\gamma},\,\, \tilde\tr(\pi(h\gamma)_{|V_{\Sigma}})=
\tilde\tr(\pi(\gamma)_{|V_{\Sigma}}).$$ 
Soit alors $H\subset H_{\pi,\gamma}$ un sous-groupe ouvert normalis\'e
par $\gamma$.
La formule
(\ref{moy}), appliqu\'ee \`a un $\Sigma$ qui contient $X^{\gamma}$ et pour
lequel on a $V_{\Sigma}\supset V^{H}$, donne 
$$
 \tilde{\rm tr}\left(\pi(\gamma)_{|V^{H}}\right) =
 \tilde\theta_{\pi}(\gamma): =  \sum_{\sigma\in
  X^{\gamma}} (-1)^{\rm dim \sigma}\epsilon_{\sigma}(\gamma)
\tilde{\rm tr}\left(\pi(\gamma)_{|V^{U_{\sigma}^{(e)}}}\right),
$$
ce qui montre que le terme de gauche est ind\'ependant de $H$, pourvu
que $H\subset H_{\pi,\gamma}$.
La fonction $\gamma\mapsto \tilde\theta_{\pi}(\gamma)$ ainsi d\'efinie
est constante sur $H_{\pi,\gamma}\gamma$ et manifestement invariante
par conjugaison ; elle appartient bien \`a $\CC^{\infty}(G^{\rm
  ell}_{\ell'/Z},\o\ZM_{\ell})^{G}$.
% La commutativit\'e du premier diagramme restreint \`a $G^{\rm ell}_{\ell
%   '}$ d\'ecoule alors du point i) du lemme \ref{lemmeperm} et de la
% formule $\theta_{\pi}(\gamma)= \sum_{\sigma\in
%   X^{\gamma}} (-1)^{\rm dim \sigma}\epsilon_{\sigma}(\gamma)
% \tilde{\rm tr}\left(\pi(\gamma)_{|V^{U_{\sigma}^{(e)}}}\right)$ qui se
% prouve de la m\^eme mani\`ere que ci-dessus, et qui est aussi un cas particulier de \cite[Prop
% 4.1]{MS1}.
% La commutativit\'e du second diagrammme restreint \`a $G^{\rm ell}_{\ell
%   '}$ d\'ecoule alors du point ii) du lemme \ref{lemmeperm} et de la
% formule $\theta_{\tilde\pi}(\gamma)= \sum_{\sigma\in
%   X^{\gamma}} (-1)^{\rm dim \sigma}\epsilon_{\sigma}(\gamma)
% \tilde{\rm tr}\left(\tilde\pi(\gamma)_{|V^{U_{\sigma}^{(e)}}}\right)$
% pour tout \'el\'ement $\tilde\pi \in\RC^{\rm ent}_{\ell'}(G,\o\QM_{\ell})$
% relevant $\pi$, que l'on trouve aussi dans \cite[Thm III.4.16]{SS2}.

Revenons au cas g\'en\'eral d'un \'el\'ement $\gamma$ semi-simple r\'egulier et
d'ordre premier \`a $\ell$ modulo  $Z$. Dans ce cas, $X^{\gamma}$ n'est plus n\'ecessairement  fini et
l'intersection $\bigcap_{\sigma\in X^{\gamma}} U_{\sigma}^{(e)}$ n'est
plus n\'ecessairement ouverte dans $G$. Cependant, fixons $\sigma_{0}\in X^{\gamma}$ :
nous allons suivre les arguments de Meyer et Solleveld dans \cite{MS2} pour prouver l'assertion suivante :
\begin{center}
  \emph{(*) : il existe un pro-$p$-sous-groupe ouvert distingu\'e
    $H_{\pi,\gamma}$ de $P_{\sigma_{0}}^{\dag}$ tel que \\ pour tout
    $\Sigma$ convexe fini stable sous $P_{\sigma_{0}}^{\dag}$  et tout
$h\in H_{\pi,\gamma},$ on a
 $   \tilde\tr(\pi(h\gamma)_{|V_{\Sigma}})=
    \tilde\tr(\pi(\gamma)_{|V_{\Sigma}}).$}
\end{center}
Admettons momentan\'ement (*) et supposons $H$ normalis\'e par $\gamma$
et inclus dans $H_{\pi,\gamma}$.
Choisissons $\Sigma$ convexe fini stable sous $P_{\sigma_{0}}^{\dag}$
tel que $V_{\Sigma}\supset V^{H}$. On a donc aussi $V_{\Sigma}\supset
V^{H_{\pi,\gamma}}$. La formule (\ref{moy}) montre alors que
$\tilde\tr(\pi(\gamma)_{|V^{H}})$ est \'egal \`a
$\tilde\tr(\pi(\gamma)_{|V^{H_{\pi,\gamma}}})$ et est donc ind\'ependant
de $H$. La fonction $\gamma\mapsto
\tilde\theta_{\pi}(\gamma):=\tilde\tr(\pi(\gamma)_{|V^{H_{\pi,\gamma}}})$
est constante sur $H_{\pi,\gamma}\gamma$ et donc lisse sur
$G^{crs}_{\ell'}$. 
 Notons qu'on a encore des formules du type 
$$  \tilde\theta_{\pi}(\gamma) =  \sum_{\sigma\in
  X_{\gamma}} (-1)^{\rm dim \sigma}\epsilon_{\sigma}(\gamma)
\tilde{\rm tr}\left(\pi(\gamma)_{|V^{U_{\sigma}^{(e)}}}\right),
$$
pour $X_{\gamma} \subset X^{\gamma}$ convenable, par exemple
$X_{\gamma}=\Sigma^{\gamma}$ pour $\Sigma$ comme ci-dessus. Si $g\in
G$, on voit facilement que $  \tilde\theta_{\pi}(g\gamma g^{-1}) =  \sum_{\sigma\in
  gX_{\gamma}g^{-1}} (-1)^{\rm dim \sigma}\epsilon_{\sigma}(g\gamma g^{-1})
\tilde{\rm tr}\left(\pi(g\gamma g^{-1})_{|V^{U_{\sigma}^{(e)}}}\right),
$ d'o\`u l'on d\'eduit que $\tilde\theta_{\pi}$ est invariante par
conjugaison, et  appartient donc \`a
$\CC^{\infty}(G^{crs}_{\ell'},\o\ZM_{\ell})^{G}$.

Reste donc \`a prouver (*). Soit $T$ le tore centralisateur de $\gamma$.
D'apr\`es \cite[9.1]{Korman} l'ensemble $X^{\gamma}$ est compact modulo
l'action de $T$. Il s'ensuit que
 l'intersection $$T_{\pi,\gamma}:= T\cap \left(\bigcap_{\sigma\in X^{\gamma}}
U_{\sigma}^{(e)}\right) $$
est un sous-groupe ouvert de
$T$. 
%En effet, d'apr\`es \cite{MS2} (40) et (34), et en utilisant les
%notations $\tilde H_{r}$ et $s(\gamma)$ de
%\cite[Lemma 4.4]{MS2}, le sous-groupe ouvert
% $T\cap \tilde H_{e+s(\gamma)}$ de $T$ est inclus dans $T_{\pi,\gamma}$.
Comme plus haut dans le cas elliptique, on a 
$X^{t\gamma}=X^{\gamma}$ pour tout $t\in T_{\pi,\gamma}$, et puisque
$T_{\pi,\gamma}$ agit trivialement sur chaque 
$V^{U_{\sigma}^{(e)}}$ pour $\sigma\in X^{\gamma}$, on en tire pour
tout $\Sigma$ convexe l'\'egalit\'e
$$ \forall t\in T_{\pi,\gamma},\,\, \tilde\tr(\pi(t\gamma)_{|V_{\Sigma}})=
\tilde\tr(\pi(\gamma)_{|V_{\Sigma}}).$$
Supposons de plus $\Sigma$ stable sous $P_{\sigma_{0}}^{\dag}$. Le
sous-espace $V_{\Sigma}$ est alors aussi stable sous
$P_{\sigma_{0}}^{\dag}$ et la
fonction $g\in P_{\sigma_{0}}^{\dag}\mapsto
\tilde\tr(\pi(g)_{|V_{\Sigma}})$ est clairement invariante par
conjugaison sous $P_{\sigma}^{\dag}$. En d'autres termes on a
$$ \forall g\in P_{\sigma_{0}}^{\dag },\, \forall t\in
T_{\pi,\gamma},\,\, \tilde\tr(\pi(gt\gamma g^{-1})_{|V_{\Sigma}})=
\tilde\tr(\pi(\gamma)_{|V_{\Sigma}}).$$ 
Or il est bien connu que l'application 
$$\application{\psi_{\gamma}:\,}{G\times T}{G}{(g,t)}{gt\gamma
  g^{-1}}$$
est ouverte. Il suffit donc de prendre pour $H_{\pi,\gamma}$ n'importe
quel sous-groupe ouvert, normalis\'e par $\gamma$, et tel que
$H_{\pi,\gamma}.\gamma\subset \psi_{\gamma}(P_{\sigma_{0}}^{\dag}\times T_{\pi,\gamma})$.
On trouvera dans \cite[Lemma 6.5]{MS2} un exemple de $H_{\pi,\gamma}$ explicite.

\end{proof}

\begin{lemme}\label{pgroupe}
  Soit $\Gamma$ un groupe fini, $P$ un $p$-sous-groupe de $\Gamma$, et
  $\gamma\in \Gamma$ un \'el\'ement d'ordre premier \`a $\ell$ qui normalise $P$.
 Alors pour toute repr\'esentation $(\pi,V)$ de $\Gamma$, on a l'\'egalit\'e
$$ \tilde\tr(\gamma_{|V^{P}})=\frac{1}{|P|} \sum_{h\in
  P}\tilde\tr(\gamma h).$$
\end{lemme}
\begin{proof}
  Rappelons que
  $\tilde\tr(\gamma_{|V^{P}})=\tilde\tr(\frac{1}{|P|}\sum_{h\in
    P}\gamma h)$. Ainsi la propri\'et\'e que l'on veut montrer est une
  propri\'et\'e de lin\'earit\'e de $\tilde\tr$ dans un cas bien pr\'ecis. 
Notons qu'il suffit de prouver l'\'enonc\'e lorsque $\Gamma$ est engendr\'e par $P$ et $\gamma$,
et par cons\'equent on peut supposer $P$ distingu\'e dans $\Gamma$. 
Soit alors $P'\subset P$ un sous-groupe distingu\'e dans $\Gamma$. Si
l'on sait prouver l'\'enonc\'e pour le triplet $(\Gamma, P', \gamma)$ et
le triplet $(\Gamma/P', P/P', \gamma.P')$, alors on l'en d\'eduit
imm\'ediatement pour le triplet $(\Gamma, P, \gamma)$.
Comme un $p$-groupe est nilpotent, on voit qu'il suffit de prouver l'\'enonc\'e
lorsque $P$ est ab\'elien. Mais alors on a une d\'ecomposition $P$-\'equivariante  
$V=V^{P}\oplus \bigoplus_{\chi\neq 1}V_{\chi}$ de $V$ selon les
caract\`eres de $P$. Comme il est clair que 
$\tilde\tr(\gamma_{|V^{P}})=\frac{1}{|P|} \sum_{h\in
  P}\tilde\tr(\gamma h_{|V^{P}})$, il nous suffira de prouver que 
$\sum_{h\in
  P}\tilde\tr(\gamma h_{|\bigoplus_{\chi\neq
    1}V_{\chi}})=0$. D\'ecomposons la somme 
$$\bigoplus_{\chi\in \hat P,
  \chi\neq 1} V_{\chi} = \bigoplus_{[\chi]\in (\hat P/\gamma), [\chi]\neq 1}
V_{[\chi]} \;\hbox{ avec }\; V_{[\chi]}=\bigoplus_{\chi\in [\chi]} V_{\chi} $$
selon les orbites de $\gamma$ dans le dual $\hat{P}$ de $P$. Pour une
orbite $[\chi]$ de cardinal $>1$, la propri\'et\'e iv) du lemme
\ref{lemmeperm} montre que $\tilde\tr(\pi(\gamma h)_{|V_{[\chi]}})=0$
pour tout $h$. Pour une orbite de cardinal $1$, on a 
$$\sum_{h\in
  P}\tilde\tr(\gamma
h_{|V_{\chi}})=\tilde\tr(\gamma_{|V_{\chi}})\left(\sum_{h\in P}
\iota(\chi(h))\right)$$
et la somme dans le terme de droite est nulle d\`es que $\chi$ est non trivial.
\end{proof}

De l'\'egalit\'e (\ref{caracthetatilde}), il apparait clairement que si
$\pi'\subset \pi$, alors
$\tilde\theta_{\pi}=\tilde\theta_{\pi'}+\tilde\theta_{\pi/\pi'}$. Par
contre le domaine de d\'efinition de $\theta_{\pi'}$ peut \^etre plus
grand que celui de $\theta_{\pi}$. Par exemple, comme  les irr\'eductibles ont des
caract\`eres centraux, leurs caract\`eres de Brauer sont d\'efinis sur tout
\'element d'ordre premier \`a $\ell$ modulo $Z(G)$. 
 Ainsi on peut d\'efinir par lin\'earit\'e
% En particulier
% $\tilde\theta_{\pi}$ est la somme des caract\`eres de Brauer de ses
% constituents irr\'eductibles. Mais comme les irr\'eductibles ont des
% caract\`eres centraux, leurs caract\`eres de Brauer sont d\'efinis sur tout
% \'element d'ordre premier \`a $\ell$ modulo $Z(G)$. On obtient ainsi un
un  homomorphisme de groupes ab\'eliens
$$\RC(G,\o\FM_{\ell})\To{\tilde\theta}
\CC^{\infty}(G^{\rm crs}_{\ell'},\o\ZM_{\ell})^{G}$$  
o\`u nous avons pos\'e  $G^{\rm crs}_{\ell'}:=G^{\rm crs}_{\ell'/Z(G)}$ 
pour all\'eger les notations.
% et $G^{\rm ell}_{\ell'}:=G^{\rm ell}_{\ell'/Z(G)}$.

Par ailleurs, nous noterons
aussi $\RC^{\rm ent}(G,\o\QM_{\ell})$, resp. $\RC^{\rm
  ent}_{\ell'}(G,\o\QM_{\ell})$, le sous-groupe de
$\RC(G,\o\QM_{\ell})$ engendr\'e par les repr\'esentations irr\'eductibles
$\o\ZM_{\ell}$-enti\`eres au sens de \cite[II.4]{Vig}, resp. $\o\ZM_{\ell}$-enti\`eres
et \`a caract\`ere central d'ordre fini et premier \`a $\ell$. Le principe
de Brauer-Nesbitt \cite[II.5.11.b)]{Vig}, nous fournit un  morphisme de
r\'eduction 
 $r_{\ell}:\,\RC^{\rm ent}(G,\o\QM_{\ell})\To{} \RC(G,\o\FM_{\ell})$
 (aussi appel\'e ``application  de d\'ecomposition'').

\begin{prop}\label{propBrauer}
 Les diagrammes suivants commutent :
$$ \xymatrix{ \RC(G,\o\FM_{\ell})\ar[r]^{\tilde\theta} \ar[d]_{\theta}
& \CC^{\infty}(G^{\rm crs}_{\ell'},\o\ZM_{\ell})^{G} \ar[d]^{r_{\ell}}
 \\ \CC^{\infty}(G^{\rm crs},\o\FM_{\ell})^{G}
 \ar[r]_{_{|G^{\rm crs}_{\ell'}}} & \CC^{\infty}(G^{\rm
  crs}_{\ell'},\o\FM_{\ell})^{G}  } 
\;\;\hbox{ et }\;\;
\xymatrix{ \RC^{\rm ent}_{\ell'}(G,\o\QM_{\ell})\ar[r]^{\theta} \ar[d]_{r_{\ell}}
& \CC^{\infty}(G^{\rm crs},\o\ZM_{\ell})^{G} \ar[d]^{_{|G^{\rm crs}_{\ell'}}}
 \\  \RC(G,\o\FM_{\ell})\ar[r]^{\tilde\theta} 
&
\CC^{\infty}(G^{\rm crs}_{\ell'},\o\ZM_{\ell})^{G}}.
$$
% On a de plus les propri\'et\'es suivantes :
% \begin{enumerate}
% \item $r_{\ell}(\tilde\theta_{\pi}(\gamma)) = \theta_{\pi}(\gamma)$,
% \item
%   $\tilde\theta_{r_{\ell}\tilde\pi}(\gamma)=\theta_{\tilde\pi}(\gamma)$
%   pour tout $\tilde\pi \in \RC(G,\o\QM_{\ell})$ si $\gamma$ est
%   compact, ou pour tout $\tilde\pi \in R(G,\o\QM_{\ell})$ de caract\`ere
%   central d'ordre fini et premier \`a $\ell$.
% \end{enumerate}

\end{prop}
\begin{proof}
 %  L'existence de $\tilde\theta_{\pi}$ est maintenant prouv\'ee et son unicit\'e
% est \'evidente. Par ailleurs la caract\'erisation (\ref{caracthetatilde})
% et la propri\'et\'e iv) du lemme \ref{lemmeperm} montrent que
% $\tilde\theta_{\pi}$ est additive sur les suites exactes courtes, d'o\`u
% l'homomorphisme $\tilde\theta$ de l'\'enonc\'e.
Prouvons la commutativit\'e du premier
diagramme. Par les caract\'erisations
(\ref{caracthetatilde}) de $\tilde\theta$
et (\ref{caractheta}) de $\theta$, il suffit de prouver que pour tout \'el\'ement $\gamma\in
G^{crs}_{\ell'}$ et tout $H$ normalis\'e par $\gamma$, on a
$\tr(\pi(\gamma*\varepsilon_{H}))=
r_{\ell}(\tilde\tr(\pi(\gamma*\varepsilon_{H})))$, ce qui d\'ecoule de
la propri\'et\'e i) du lemme \ref{lemmeperm}.
Pour prouver la commutativit\'e du second diagramme de l'\'enonc\'e, il
suffit de partir de $\tilde\pi$ irr\'eductible et enti\`ere et de prouver que
pour tout $\gamma,H$ comme ci-dessus on a
$\tr(\tilde\pi(\gamma*\varepsilon_{H}))
=\tilde\tr(r_{\ell}\tilde\pi(\gamma*\varepsilon_{H}))$.
Or ceci d\'ecoule de la propri\'et\'e ii) du lemme \ref{lemmeperm}, \`a
condition que $\rho:= \gamma*\varepsilon_{H}$ satisfasse une \'equation
$\rho^{n+1}=\rho$ avec $n$ premier \`a $\ell$. Ceci est toujours le cas
si 
%$\gamma$ est compact (pas seulement ``modulo le centre'') d'ordre
%premier \`a $\ell$, ou si 
le caract\`ere central de $\pi$ est d'ordre fini et premier \`a $\ell$.
\end{proof}

%Nous verrons que dans le cas de $G=\GL_{d}(K)$, certaines propri\'et\'es
%classiques des caract\`eres de Brauer des groupes finis s'\'etendent au caract\`ere de Brauer
%que nous venons de d\'efinir. 

\alin{\'El\'ements non compacts}
  On peut \'etendre le caract\`ere de Brauer \`a certains \'el\'ements semi-simples
  r\'eguliers  $\gamma$ de la mani\`ere suivante. Soit $(P_{\gamma}^{+},P_{\gamma}^{-})$ la
paire de  \paras oppos\'es associ\'ee \`a $\gamma$, et
$M_{\gamma}=P_{\gamma}^{+}\cap P_{\gamma}^{-}$ leur composante de Levi
commune. Alors $\gamma$ est compact modulo le centre de
$M_{\gamma}$. Supposons-le d'ordre premier \`a $\ell$ modulo ce
centre. On pose alors
$$ \tilde\theta_{\pi}(\gamma):= \tilde\theta_{r_{P_{\gamma}}(\pi)}(\gamma)^{\rm ss} $$
o\`u $r_{P_{\gamma}}$ d\'esigne le foncteur de Jacquet associ\'e \`a
$P_{\gamma}$ et le signe $^{\rm ss}$ d\'esigne la semisimplifi\'ee. On
obtient ainsi une fonction $\tilde\theta_{\pi}$ d\'efinie sur l'ensemble
$G^{\rm rs}_{\ell'}$
des semi-simples r\'eguliers $\gamma$ tels que $\gamma \in
M_{\gamma,\ell'}^{\rm crs}$. On v\'erifie ais\'ement que cet ensemble est
ouvert et stable par conjugaison sous $G$, et que la fonction
$\tilde\theta_{\pi}$ est lisse et invariante sous $G$. Il r\'esulte
alors de \cite[Thm 7.4]{MS2} que si on remplace 
$G^{\rm crs}_{\ell'}$ par $G^{\rm rs}_{\ell'}$ dans le premier
diagramme de la proposition, alors le diagramme obtenu est encore
commutatif.
 Par contre, le
deuxi\`eme diagramme n'a pas de g\'en\'eralisation int\'eressante avec cette
d\'efinition, car les exposants des $\o\QM_{\ell}$-repr\'esentations ``int\'eressantes'' ne
sont pas d'ordre fini. Penser par exemple \`a la repr\'esentation de Steinberg.
% semble pouvoir \^etre en d\'efaut, car l'action de
%$Z(M_{\gamma})$ sur $r_{P_{\gamma}}(\pi)$ n'est g\'en\'eralement pas semi-simple.

\alin{\'Elements d'ordre divisible par $\ell$}
Le premier diagramme de la proposition \ref{propBrauer} montre que le
caract\`ere de Brauer d\'etermine la restriction du caract\`ere ordinaire
aux \'el\'ements d'ordre premier \`a $\ell$. En fait, comme pour les groupes
finis, le caract\`ere de Brauer d\'etermine enti\`erement le caract\`ere ordinaire.
Cela fait intervenir l'analogue suivant de 
la d\'ecomposition unique $x=x_{\ell}x_{\ell'}$ d'un \'el\'ement
d'ordre fini en le produit d'un $\ell$-\'el\'ement et d'un $\ell'$-\'el\'ement
commutants.

\begin{lem}
  Soit $\gamma$ un \'el\'ement compact modulo un sous-groupe ferm\'e central
  $Z$ de $G$. Il existe un couple $(\gamma_{\ell},\gamma_{\ell'})$ 
d'\'el\'ements de $G$ satisfaisant les propri\'et\'es suivantes :
\begin{itemize}
\item $\gamma=\gamma_{\ell}\gamma_{\ell'}=\gamma_{\ell'}\gamma_{\ell}$,
\item $\gamma_{\ell}$ est d'ordre fini \'egal \`a une puissance de $\ell$ dans
  $G/Z$,
\item $\gamma_{\ell'}$ est d'ordre premier \`a $\ell$ modulo $Z$
\end{itemize}
Cette d\'ecomposition est unique modulo l'action \'evidente de $Z$, et
l'on a $\gamma_{\ell},\gamma_{\ell'}\in Z.\o{\langle\gamma\rangle}$.
\end{lem}
\begin{proof}
  Cela d\'ecoule de l'existence de d\'ecompositions de Jordan
  $p$-topologiques pour les
  \'el\'ements compacts du groupe
  localement $p$-profini $G/Z$, \emph{cf} \cite{Spice}.
\end{proof}

Soit $(\pi,V)$ une $\o\FM_{\ell}$-repr\'esentation de longueur finie,
$Z$ un sous-groupe central ferm\'e de $G$ agissant de mani\`ere
semi-simple sur $\pi$  et $\gamma\in G$ un \'el\'ement compact modulo
$Z$. Choisissons une d\'ecomposition
$\gamma=\gamma_{\ell}\gamma_{\ell'}$ comme ci-dessus. Tout
pro-$p$-sous-groupe ouvert $H$ normalis\'e par $\gamma$ est encore
normalis\'e par $\gamma_{\ell}$ et $\gamma_{\ell'}$. De plus, le produit
$\pi(\gamma_{\ell})_{|V^{H}}\pi(\gamma_{\ell'})_{|V^{H}}$ est la
$\ell$-d\'ecomposition usuelle de l'automorphisme d'ordre fini
$\pi(\gamma)_{|V^{H}}$. On sait dans ces conditions, \emph{cf}
\cite[Lemma (17.8)]{CR}, que 
$$ \tr(\pi(\gamma)_{|V^{H}}) = \tr(\pi(\gamma_{\ell'})_{|V^{H}}. $$
Si $\gamma_{\ell'}$ est \emph{r\'egulier}, on en d\'eduit
$\theta_{\pi}(\gamma)=\theta_{\pi}(\gamma_{\ell'})$. 
N\'eanmoins, en
g\'en\'eral $\gamma_{\ell'}$ n'est pas n\'ecessairement r\'egulier.
Cependant,  supposons $H$ assez petit pour que $H\gamma\subset
G^{\rm crs}$ et pour  que $\theta_{\pi}$ soit constant sur $H\gamma$. 
Pour tout $h\in H\cap Z_{G}(\gamma)$, on a $(\gamma h)_{\ell'}=\gamma_{\ell'}h$.
Choisissons $h$ tel que $\gamma_{\ell'}h$ soit r\'egulier. On a alors
$\theta_{\pi}(\gamma)=\theta_{\pi}(\gamma h)=\theta_{\pi}(\gamma_{\ell'}
h)$. Ceci montre que $(\theta_{\pi})_{|G^{\rm crs}}$ est d\'etermin\'e par sa
restriction  $(\theta_{\pi})_{|G^{\rm crs}_{\ell'}}$.
Plus pr\'ecis\'ement on a prouv\'e :
\begin{pro} \label{ellelements}
  Soit $x\in \RC(G,\o\FM_{\ell})$. Si la restriction de $\theta_{x}$ \`a
  $G^{\rm rs}_{\ell'}$, resp. \`a $G^{\rm crs}_{\ell'}$, resp. \`a $G^{\rm
  ell}_{\ell'}$, est nulle, alors il en est de m\^eme de sa restriction \`a  $G^{\rm rs}$,
resp. \`a $G^{\rm crs}$, resp. \`a $G^{\rm  ell}$.
\end{pro}

Voici maintenant deux propri\'et\'es classiques des
caract\`eres de Brauer des groupes finis :
\begin{itemize}
\item \emph{Injectivit\'e :} l'application $\tilde\theta:\,
  \RC(G,\o\FM_{\ell})\To{} \CC(G_{\ell'},\o\ZM_{\ell})^{G}$ est injective.
\item \emph{Surjectivit\'e :} soit $x\in\RC(G,\o\FM_{\ell})$,
  prolongeons $\tilde\theta_{x}$ \`a $G$ en posant
  $\bar\theta_{x}(g):=\tilde\theta_{x}(g_{\ell'})$. Alors la fonction
  obtenue est le caract\`ere d'un \'el\'ement $\tilde{x}\in \RC(G,\o\QM_{\ell})$. 
\end{itemize}

Ces propri\'et\'es s'\'etendent facilement \`a un groupe $p$-adique compact
modulo son centre, bien que notre caract\`ere de Brauer ne soit d\'efini que
sur les \'el\'ements r\'eguliers. Cela d\'ecoule en effet de la densit\'e des
\'elements r\'eguliers, \emph{cf} la 
preuve du th\'eor\`eme \ref{theoLJ}. 

Pour un groupe r\'eductif $p$-adique $G$ plus g\'en\'eral, il r\'esulte de
\cite{VigWalds} que pour $\ell$ \emph{assez grand}, le caract\`ere de Brauer
$\tilde\theta$ \'etendu \`a tout $G^{\rm rs}$ comme plus haut d\'efinit une
injection de $\RC(G,\o\FM_{\ell})$ dans $\CC^{\infty}(G^{\rm
  rs}_{\ell'},\o\ZM_{\ell})$. Les seules conditions explicites donn\'ees
dans \emph{loc. cit.} sont $\ell> n$, pour $G=\GL_{n}(K)$ ou ${\rm
  SL}_{n}(K)$. 

N\'eanmoins, que ce soit pour l'injectivit\'e ou la surjectivit\'e, il
para\^{\i}t plus utile d'obtenir un analogue modulaire de
\cite{KazhdanCusp}. 

\begin{ques} \label{questBrauer} Soit $\RC_{I}(G,\o\FM_{\ell})$ le sous-groupe de
  $\RC(G,\o\FM_{\ell})$ engendr\'e par les induites paraboliques
  propres.
  \begin{enumerate}
  \item A-t-on $\ker (\tilde\theta_{|G^{\rm ell}_{\ell'}} \otimes \QM)=
    \RC_{I}(G,\o\FM_{\ell})\otimes \QM$ ?
  \item Soit $\pi \in \RC(G,\o\FM_{\ell})$ et $\bar\theta_{\pi}$ le
    prolongement \`a $G^{\rm ell}$ de $\tilde\theta_{\pi}$ d\'efini par 
$$ \bar\theta_{\pi}(\gamma):=\tilde\theta_{\pi}(\gamma_{\ell'}h)$$
o\`u $h\in Z_{G}(\gamma)\cap H_{\pi,\gamma}$ est tel que
$\gamma_{\ell'}h$ soit r\'egulier. Est-ce que $\bar\theta_{\pi}$ est la
restriction \`a $G^{\rm ell}$ du caract\`ere d'un \'el\'ement de
$\RC(G,\o\QM_{\ell})$ ?
  \end{enumerate}
\end{ques}
Concernant le point i),  on v\'erifie en effet facilement, et par les
  m\^emes arguments que pour le caract\`ere ordinaire \cite{VanDijk}, que le caract\`ere de
  Brauer d'une induite parabolique propre est nul sur les \'el\'ements elliptiques.
  
Une r\'eponse \`a ces deux questions permettrait par exemple d'en d\'eduire
la surjectivit\'e de la r\'eduction de Brauer-Nesbitt $r_{\ell}$.
Dans les paragraphes suivants, nous r\'epondrons affirmativement \`a ces
deux questions dans le cas de $G=\GL_{d}(K)$. Contrairement aux
arguments d'analyse harmonique de Kazhdan dans \cite{KazhdanCusp}, nous
utiliserons les r\'esultats de classification des repr\'esentations de M.-F. Vign\'eras.

\subsection{Repr\'esentations de $\GL_{d}(K)$. Classifications de Vign\'eras et Zelevinsky}

Dans ce paragraphe nous rappelons l'\'enonc\'e de classification des
repr\'esentations lisses irr\'eductibles de $G_{d}:=\GL_{d}(K)$ d\^u \`a
Zelevinsky lorsque les coefficients sont $\o\QM_{\ell}$ et \`a Vign\'eras
lorsque les coefficients
sont $\o\FM_{\ell}$. Nous en d\'eduisons facilement que
l'homomorphisme de r\'eduction $r_{\ell}:\, \RC^{\rm
  ent}(G_{d},\o\QM_{\ell})\To{}\RC(G_{d},\o\FM_{\ell})$ est surjectif, ce qui
sera crucial pour d\'efinir $\LJ_{\oFl}$ au
prochain paragraphe. Au passage, nous prouvons quelques propri\'et\'es des
repr\'esentations de Speh et ``superSpeh'', qui seront utilis\'ees pour la
d\'efinition de $^{z}\JL_{\oFl}$.
Dans ce qui suit, la lettre $C$ d\'esignera le corps
$\o\FM_{\ell}$ ou le corps $\o\QM_{\ell}$. Les repr\'esentations
irr\'eductibles consid\'er\'ees sont \`a coefficients dans $C$ sauf pr\'ecision contraire.

\alin{Mod\`eles de Whittaker et d\'eriv\'ees}
Fixons un caract\`ere non trivial $\psi : K\To{} \o\ZM_{\ell}^{\times}$.
\`A toute partition  $\lambda=(\lambda_{1}\geq\lambda_{2}\geq \cdots \geq
\lambda_{t})\in \Lambda(d)$  de $d$ est alors associ\'e un $\o\ZM_{\ell}$-caract\`ere $\psi_{\lambda}$
du sous-groupe unipotent maximal sup\'erieur $U_{d}$ de $G_{d}$, qui
induit \`a son tour un caract\`ere \`a valeurs dans $C$, encore not\'e
$\psi_{\lambda}$, \emph{cf.} \cite[V.5]{VigInduced}.
Pour une $C$-repr\'esentation $\pi$ de $G_{d}$, on d\'efinit
 $$\Lambda(\pi) :=\left\{\lambda\in \Lambda(d),\, m_{\pi,\lambda}:={\rm
   dim} \left({\rm Hom}_{U_{d}}(\pi,\psi_{\lambda}) \right)\neq 0\right\}
.$$
Par r\'eciprocit\'e de Frobenius tout morphisme $\varphi\in{\rm
  Hom}_{U_{d}}(\pi,\psi_{\lambda})$ induit un morphisme $\pi\To{}{\rm
  Ind}_{U_{d}}^{G_{d}}(\psi_{\lambda})$ qui, s'il est injectif, est
appel\'e $\lambda$-mod\`ele de Whittaker.
Rappelons que lorsque $\pi$ est irr\'eductible et $(d)\in \Lambda(\pi)$,  
on dit que $\pi$ est ``g\'en\'erique'' ou encore ``non d\'eg\'en\'er\'ee''. On
sait alors \cite[III.5.10.3)]{Vig} que $m_{\pi,(d)}=1$ ; c'est l'unicit\'e du mod\`ele de Whittaker.

% admet un $\lambda$-mod\`ele
% de Whittaker si ${\rm Hom}_{U_{d}}(\pi,\psi_{\lambda})\neq 0$, et on
% dit que ce mod\`ele est unique si $\dim({\rm
%   Hom}_{U_{d}}(\pi,\psi_{\lambda}))=1$. Nous noterons $\Lambda(\pi)$
% l'ensemble des partitions $\lambda$ pour lesquelles $\pi$ admet un
% $\lambda$-mod\`ele de Whittaker. Rappelons qu'un $(d)$-mod\`ele de
% Whittaker est en g\'en\'eral simplement appel\'e ``mod\'ele de Whittaker'' et
% qu'une repr\'esentation $\pi$ admettant un $(d)$-mod\`ele de Whittaker est
% appel\'ee ``non d\'eg\'en\'er\'ee'' ou encore ``g\'en\'erique''.

Plus g\'en\'eralement, munissons l'ensemble $\Lambda(d)$ de l'ordre
partiel d\'efini par $\lambda \geq \lambda'$ \ssi\ 
$\forall i=1,\cdots, t, \,\sum_{j=1}^{i}\lambda_{j}\geq \sum_{j=1}^{i}\lambda'_{j}$.
La th\'eorie des d\'eriv\'ees de Gelfand-Kazhdan montre alors \cite[V.5]{VigInduced}
que \emph{pour toute $\pi$ irr\'eductible, l'ensemble
  $\Lambda(\pi)$ admet un plus grand \'el\'ement, que l'on notera
  $\lambda_{\pi}$}.
Concr\`etement, la partition $\lambda_{\pi}$ est donn\'ee par les plus grandes
d\'eriv\'ees successives. Plus pr\'ecis\'ement, $\lambda_{\pi,1}$ est l'ordre de
la plus grand d\'eriv\'ee de $\pi$,
$\lambda_{\pi,2}=\lambda_{\pi^{(\lambda_{\pi,1})},1}$ est celui de la d\'eriv\'ee
$\pi^{(\lambda_{\pi, 1})}$ qui est une repr\'esentation de
$G_{d-\lambda_{\pi,1}}$, etc.

Notons que, comme dans l'article de Zelevinsky \cite{Zel} sur les
repr\'esentations complexes, une fois prouv\'e le
th\'eor\`eme de classification (rappel\'e plus bas), on en d\'eduit
\cite[V.12]{VigInduced} que \emph{le $\lambda_{\pi}$-mod\`ele de
  Whittaker  est unique, \emph{i.e.} $m_{\pi,\lambda_{\pi}}=1$.}

\alin{D\'eriv\'ees et induites paraboliques}
Pour d\'efinir les induites paraboliques normalis\'ees, il nous faut choisir une
racine carr\'ee de $q$ dans $\o\ZM_{\ell}$.
Comme d'habitude, pour $\pi_{i}$ une
repr\'esentation de $G_{d_{i}}$, $i=1,\cdots, t$, 
on note  $\pi_{1}\times \pi_{2}\times\cdots \times\pi_{t}$ l'induite
 normalis\'ee de la repr\'esentation
 $\pi_{1}\otimes\pi_{2}\otimes\cdots\otimes \pi_{t}$ du sous-groupe
de Levi diagonal par blocs  $G_{d_{1}}\times\cdots\times G_{d_{t}}$ de
$G_{d_{1}+\cdots + d_{t}}$, le
long du parabolique triangulaire par blocs sup\'erieur correspondant,
not\'e $P_{d_{1},\cdots,d_{t}}$.

La formule ``de Leibniz'' pour la d\'eriv\'ee d'une induite parabolique
\cite[III.1.10]{Vig} montre que la partition $\lambda_{\pi_{1}}+ \lambda_{\pi_{2}}
+\cdots + \lambda_{\pi_{t}} \in \Lambda(d_{1}+\cdots+d_{t})$ est le plus
grand \'el\'ement de l'ensemble
$\Lambda(\pi_{1}\times\cdots\times\pi_{t})$ et que de plus, on a
$m_{\pi_{1}\times\cdots\times\pi_{t}, \lambda_{\pi_{1}}+ \lambda_{\pi_{2}}
+\cdots + \lambda_{\pi_{t}}}=1$ si l'on sait que chaque
$m_{\pi_{i},\lambda_{\pi_{i}}}$ vaut $1$.

\alin{Segments (super)cuspidaux et repr\'esentations (super)Speh}
\label{defsuperSpeh} 
%Nous noterons $G_{d}$ pour $\GL_{d}(K)$.
 Pour tout $d$, la lettre $\nu$ d\'esignera le caract\`ere $g\mapsto q^{-{\rm val}_{K} \circ
   {\rm det}(g)}$ de $G_{d}$ \`a valeurs dans $\ZM_{\ell}^{\times}$. 
Un $C$-segment (super)cuspidal est une paire $\Delta=(\pi,r)$ form\'ee d'une
$C$-repr\'esentation irr\'eductible (super)cuspidale de $G_{d}$ et d'un entier positif
$r$. Comme la seule d\'eriv\'ee non nulle de $\pi$ est $\pi^{(d)}$, la formule de Leibniz montre que l'ensemble  $\Lambda(\pi\times\pi\nu\times \cdots
\times\pi\nu^{r-1})$ est form\'e de $d$-partitions de $dr$, \emph{i.e.}
dont les entr\'ees sont divisibles par $d$. L'\'el\'ement minimal de cet
ensemble est la partition $(d^{(r)}):=(d,d,\cdots, d)$.

Lorsque $C=\oQl$, ou plus g\'en\'eralement lorsque les $\pi\nu^{i}$ sont
deux \`a deux non isomorphes, on montre facilement que l'induite $\pi\times\pi\nu\times \cdots
\times\pi\nu^{r-1}$ poss\`ede une unique sous-repr\'esentation
irr\'eductible, que nous noterons $\delta_{r}(\pi)$.
% et appellerons ``repr\'esentation de Speh''. 
Parmi les sous-quotients irr\'educibles de
cette induite, elle est caract\'eris\'ee par l'une ou l'autre des
propri\'et\'es suivantes :
% l'induite $\pi\times\pi\nu\times \cdots \times\pi\nu^{r-1}$ contient
% une  sous-repr\'esentation irr\'eductible  $\delta_{r}(\pi)$
% dont la classe d'isomorphisme est uniquement d\'etermin\'ee par les conditions
% suivantes :
\begin{enumerate}
\item[m1)] Son module de Jacquet est donn\'e par $r_{P_{d,d,\cdots, d}}(\delta_{r}(\pi)) 
\simto \pi\otimes\pi\nu\otimes \cdots \otimes\pi\nu^{r-1}$
\item[m2)] Sa partition de Whittaker est donn\'ee par $\lambda_{\delta_{r}(\pi)}= (d,d,\cdots, d)$.
%l'ordre de la plus grande d\'eriv\'ee de ${\rm Sp}_{r}(\pi)$ est $d$.
\end{enumerate}

Lorsque $C=\oFl$, il n'y a pas de d\'efinition aussi simple de
$\delta_{r}(\pi)$,
 et l'auteur n'en connait pas qui n'utilise
pas la th\'eorie des types. 
La d\'efinition de  \cite[V.9.1]{VigInduced} n'est pas tout \`a fait
correcte car la propri\'et\'e m1) ci-dessus ne suffit pas toujours \`a
isoler une unique repr\'esentation. Cependant, en r\'eordonnant les
arguments de \emph{loc. cit.}, on aboutit \`a la d\'efinition suivante,
dont l'\'enonc\'e a \'et\'e adapt\'e \`a l'usage que nous ferons de ces
repr\'esentations par la suite.

% , on utilise une bijection ``naturelle''
% entre repr\'esentations irr\'eductibles de support cuspidal
% $\pi\times\cdots\times\pi\nu^{r}$ et certains modules simples d'une alg\`ebre de
% Hecke affine. La repr\'esentation  $\delta_{r}(\pi)$ est alors celle qui correspond
% au caract\`ere ``trivial''. 
% De mani\`ere \'equivalente, dans \cite{MinSech} on utilise un foncteur
% $\kappa_{\rm max}$ comme dans \cite{SchZ} \`a valeurs dans les
% repr\'esentations d'un certain  $\GL_{d'}$ sur
% un corps fini. On d\'efinit alors $\delta_{r}(\pi)$ comme l'unique
% sous-repr\'esentation irr\'eductible de $\pi\times\cdots\times\pi\nu^{r-1}$ dont le
% $\kappa_{\rm max}$ contient une certaine sous-repr\'esentation
% explicitement d\'efinie par James (not\'ee $S(\sigma,(m))$ dans
% \cite[III.2.4]{Vig}). 

%Voici une autre d\'efinition \'equivalente, dont l'\'enonc\'e nous sera utile
%pour l'\'etude de $\LJ_{\oFl}$.
 On rappelle auparavant
\cite[III.5.10.2]{Vig} que \emph{toute
$\oFl$-repr\'esentation cuspidale $\pi$ admet un rel\`evement $\tilde\pi$ \`a
$\oQl$}, c'est-\`a-dire une $\oQl$-repr\'esentation contenant un $\oZl
G$-sous-module de type fini $\omega$ g\'en\'erateur, dont la r\'eduction
$\omega\otimes_{\oZl}\oFl$ est isomorphe \`a $\pi$. 

\begin{pro}
  Soient $\pi$ et $\tilde\pi$ comme ci-dessus. La repr\'esentation
  $\delta_{r}(\tilde\pi)$ est enti\`ere. Sa r\'eduction
  $r_{\ell}(\delta_{r}(\tilde\pi))$ est irr\'eductible, satisfait les
  propri\'et\'es m1) et m2) ci-dessus, et est ind\'ependante, 
\`a isomorphisme pr\`es,  du choix de $\tilde\pi$.
\end{pro}
\begin{proof}
La repr\'esentation
  $\delta_{r}(\tilde\pi)$ est enti\`ere, comme toute repr\'esentation dont
  le support cuspidal est form\'e de repr\'esentations  enti\`eres,
  \emph{cf.} \cite[II.4.14]{Vig}.
 La propri\'et\'e m2) de
  $\delta_{r}(\tilde\pi)$ et la compatibilit\'e des d\'eriv\'ees (exactes) \`a la
  r\'eduction mod $\ell$ impliquent que
  $\Lambda(\delta_{r}(\pi))=\{(d,\cdots,d)\}$. Comme on a aussi
  $m_{\delta_{r}(\pi), (d,\cdots, d)}=1$, il s'ensuit que
  $r_{\ell}(\delta_{r}(\tilde\pi))$ est irr\'eductible
  (et m\^eme sa restriction au sous-groupe mirabolique est irr\'eductible). 
Elle v\'erifie bien m2), et v\'erifie aussi m1) par compatibilit\'e des
foncteurs de Jacquet (exacts) avec la r\'eduction modulo $\ell$.
Il reste \`a voir que
  $r_{\ell}(\delta_{r}(\tilde\pi))$ est ind\'ependante des choix.

% Notons que si l'on fixe
%   $\tilde\pi$, l'ind\'ependance du r\'eseau $\omega$ est facile, puisque
%   celui-ci est unique \`a un multiple pr\`es.
% Choisissons un $\oFl$-type simple $(J,\kappa\otimes\sigma)$ contenu
% dans l'induite $\pi\times\cdots\times\pi{\nu^{r-1}}$. Ici, $J$ est
% un sous-groupe ouvert compact de $G$, $\kappa$ est une
% $\oFl$-repr\'esentation de $J$ dont la restriction au
% pro-$p$-radical $J^{1}$ de $J$ est irr\'eductible, et
% $\sigma=\pi_{0}\otimes\pi_{0}\otimes\cdots\otimes\pi_{0}$ est une
% repr\'esentation cuspidale irr\'eductible de $J/J^{1}\simeq
% \GL_{d'/r}(k')^{r}$ pour une certaine extension finie $k'$ de $k$ et
% un entier $d'$. On sait \cite[III.4.20, 4.29]{Vig} qu'il existe des rel\`evements
% $\tilde\kappa$ et $\tilde\sigma=\tilde\pi_{0}\otimes\cdots\otimes\tilde\pi_{0}$
% de $\kappa$ et $\sigma$ \`a $\oQl$ tels que le $\oQl$-type simple
% $(J,\tilde\kappa\otimes\tilde\sigma)$ soit contenu dans
% $\tilde\pi\times\cdots\times\tilde\pi\nu^{r-1}$.
Choisissons un $\oFl$-type simple $(J^{\circ},\lambda)$ contenu
dans l'induite $\pi\times\cdots\times\pi{\nu^{r-1}}$, \emph{cf.}
 \cite[Lemme 5.12]{Vig} o\`u la notation $(J_{m}^{\circ},\lambda_{m})$
 est utilis\'ee. 
 On sait \cite[III.4.20, 4.29]{Vig} qu'il existe un rel\`evement
 $\tilde\lambda$ de $\lambda$ \`a $\oQl$ tel que le $\oQl$-type simple
$(J^{\circ},\tilde\lambda)$ soit contenu dans
$\tilde\pi\times\cdots\times\tilde\pi\nu^{r-1}$. Fixons un $\oZl$-r\'eseau
stable $\lambda_{\oZl}$ dans $\tilde\lambda$, et notons
$\HC(G,\lambda_{\oZl})$ son alg\`ebre de Hecke. Dans \cite[III.5.6-5.7]{Vig}
est construite une famille d'isomorphismes ``canoniques''
\begin{equation}\label{hecke}
\HC_{\oZl}(\wt\SG_{r},q')\simto \HC(G,\lambda_{\oZl})
\end{equation}
depuis
l'alg\`ebre de Hecke affine \'etendue de type $\wt{A}_{r-1}$ sp\'ecialis\'ee en une
certaine puissance $q'$ de $q$. Ces isomorphismes sont des versions
enti\`eres de ceux de Bushnell-Kutzko \cite[(5.6.6)]{BK}.
Rappelons que l'on a une d\'ecomposition en produit tensoriel de
$\oZl$-modules
$$\HC_{\oZl}(\wt\SG_{r},q') = \HC_{\oZl}(\SG_r,q')\otimes_{\oZl}
\oZl[X_{1}^{\pm 1},\cdots, X_{r}^{\pm 1}] = :\HC^{0}\otimes_{\oZl} \AC
$$
o\`u $\HC^{0}$ est la sous-alg\`ebre de Hecke de type $A_{r-1}$ et $\AC$
est une sous-alg\`ebre commutative. 

La th\'eorie de Bushnell et Kutzko nous dit que $\tilde{M}:={\rm
  Hom}_{J^{\circ}}(\tilde\lambda,\delta_{r}(\tilde\pi))$ est un
$\HC(G,\tilde\lambda)$-module simple \cite[Thm. (7.5.7)]{BK}, 
et que sa restriction \`a $A$ est
d\'etermin\'ee par $r_{P_{d,\cdots, d}}(\delta_{r}(\tilde\pi))=
\tilde\pi\otimes\tilde\pi\nu\otimes\cdots\otimes\tilde\pi\nu^{r-1}$
\cite[Thms. (7.6.1),(7.6.20)]{BK}. En particulier,
 $\tilde{M}$ est un caract\`ere  dont la restriction \`a
$\AC$ est donn\'ee par $X_{i}\mapsto q^{i-1}t$ o\`u $t\in\oZl^{\times}$
d\'epend de la normalisation de l'isomorphisme (\ref{hecke}). Il s'ensuit
que la restriction \`a 
$\HC^{0}$ doit \^etre le caract\`ere ``trivial'' $e_{w}\mapsto q^{l(w)}$ et que
la restriction \`a $\HC(G,\lambda_{\oZl})$ est \`a valeurs enti\`eres.

Par r\'eduction modulo $\ell$, on constate que $M:= {\rm
  Hom}_{J^{\circ}}(\lambda,r_{\ell}(\delta_{r}(\tilde\pi)))\neq 0$. Il
r\'esulte alors de la propri\'et\'e de ``quasi-projectivit\'e'' de
$\cind{J^{\circ}}{G}{\lambda}$  que $M$ est un module simple sur
$\HC(G,\lambda)$ \cite[Thm IV.2.5, 3)]{Vig}, donc n\'ecessairement 
\'egal \`a $r_{\ell}(\tilde{M})$.
Ainsi, la restriction de $M$ \`a $\AC$ est d\'etermin\'ee par
$\pi\otimes\pi\nu\otimes\cdots\otimes\pi\nu^{r-1}$ et sa restriction \`a
$\HC^{0}$ est le caract\`ere ``trivial''. Ceci d\'etermine enti\`erement
$M$, et \`a nouveau par \cite[Thm IV.2.5, 3)]{Vig},
 cela d\'etermine aussi $r_{\ell}(\delta_{r}(\tilde\pi))$ qui est
donc bien ind\'ependante de $\tilde\pi$.
\end{proof}

\begin{defn}
  Nous noterons $\delta_{r}(\pi):=r_{\ell}(\delta_{r}(\tilde\pi))$ ``la''
  repr\'esentation donn\'ee par la proposition ci-dessus. Une telle
  repr\'esentation sera appel\'ee  une \emph{repr\'esentation de Speh} ou
\emph{repr\'esentation superSpeh} selon
 que $\pi$ est cuspidale ou supercuspidale.
\end{defn}

\begin{rem}
  Il devrait \^etre possible d'\'ecrire $\delta_{r}(\pi)$ comme image d'un
  op\'erateur d'entrelacement ``explicite''
  $\pi\nu^{r-1}\times\pi\nu^{r-2}\times\cdots\times\pi\To{}
  \pi\times\cdots\times\pi\nu^{r-1}$. Mais l'auteur n'a pas \'et\'e
  capable d'\'ecrire un tel op\'erateur lorsque les $\pi\nu^{i}$ ne sont
  pas deux \`a deux distincts. La d\'efinition ci-dessus a l'inconv\'enient
  minime de ne d\'efinir en fait  qu'une classe d'isomorphisme de
  repr\'esentations irr\'eductibles.
Cependant, on peut aussi d\'efinir $\delta_{r}(\pi)$ comme ``unique
sous-repr\'esentation irr\'eductible de
$\pi\times\pi\nu\cdots\times\pi\nu^{r-1}$ v\'erifiant une certaine
condition''. Pour exprimer la condition, on utilise un foncteur
$\kappa_{\rm max}$ convenable, comme dans \cite[Sec. 5]{SchZ} et
\cite{MinSech}. Un tel foncteur est exact, \`a valeurs dans les
repr\'esentations d'un $\GL_{d'}(k')$ avec $k'$ fini.
La condition est alors que $\kappa_{\rm max}(\delta_{r}(\pi))$
contient (et m\^eme est \'egal \`a) la sous-repr\'esentation d\'efinie
explicitement par James  comme image d'un op\'erateur d'entrelacement explicite,
  \emph{cf.} \cite[III.2.4]{Vig} o\`u cette repr\'esentation  est
  not\'ee $S(\sigma,(m))$.
\end{rem}

%Bien-s\^ur dans le cas $C=\oQl$, les repr\'esentations de Speh sont toutes superSpeh.
% Alternativement, on peut proc\'eder par rel\`evement :
% on choisit un rel\`evement $\tilde\pi$ de $\pi$ \`a $\oQl$
% \cite[III.5.10.2]{Vig}, on sait que $\delta_{r}(\tilde\pi)$ est
% $\ell$-enti\`ere et de r\'eduction irr\'eductible (par la propri\'et\'e m2)) et
% on d\'efinit $\delta_{r}(\pi):=r_{\ell}(\delta_{r}(\tilde\pi))$. La
% difficult\'e est alors de prouver que la d\'efinition ne d\'epend pas du
% choix du rel\`evement, ce que l'auteur ne sait pas faire sans les types.

Notons que  la repr\'esentation $\delta_{r}(\pi)$ v\'erifie bien les
propri\'et\'es m1) et m2) ci-dessus, mais n'est g\'en\'erale\-ment caract\'eris\'ee
par aucune des deux. Cependant, une fois prouv\'e le th\'eor\`eme de
classification ci-dessous, il appara\^{\i}t que ces deux conditions
caract\'erisent $\delta_{r}(\pi)$. Par ailleurs, on v\'erifie ais\'ement que
la $d$-\`eme d\'eriv\'ee est donn\'ee par $\delta_{r}(\pi)^{(d)}\simeq
\delta_{r-1}(\pi)$ et que $m_{\delta_{r}(\pi),\lambda_{\delta_{r}(\pi)}}=1$.

%Par r\'eciprocit\'e de Frobenius, la repr\'esentation $\delta_{r}(\pi)$ est
%contenue dans l'induite $\pi\times\pi\nu\times \cdots \times\pi\nu^{r-1}$.
%De plus, on v\'erifie ais\'ement que la $d$-\`eme d\'eriv\'ee est donn\'ee par
%$\delta_{r}(\pi)^{(d)}\simeq \delta_{r-1}(\pi)$.
%Nous appellerons $\delta_{r}(\pi)$ une \emph{repr\'esentation de Speh} ou
%\emph{repr\'esentation superSpeh}
% \footnote{Cette terminologie n'est pas standard,
%   mais parait plus adapt\'ee que ``repr\'esentation de (super)Speh'' ou
%   ``repr\'esentation de (super)Zelevinsky''}
%selon  que $\pi$ est cuspidale ou supercuspidale.
%  Dans le cas
% $C=\o\QM_{\ell}$, la premi\`ere propri\'et\'e suffit \`a caract\'eriser $\delta_{r}(\pi)$.

%, et que $\lambda_{\rm Sp_{r}(\pi)}= (d,d,\cdots, d)$.

\alin{R\'eduction et rel\`evements} \label{terminologie}
Nous adoptons et extrapolons la terminologie  de Vign\'eras. Nous dirons qu'une
$\o\QM_{\ell}$-repr\'esentation est 
\begin{itemize}
\item \emph{$\ell$-enti\`ere} si elle admet un r\'eseau stable comme dans \cite[II.4]{Vig}.
\item \emph{$\ell$-irr\'eductible} si elle est $\ell$-enti\`ere et si la
  r\'eduction de tout r\'eseau stable est irr\'eductible.
\item \emph{$\ell$-supercuspidale} si elle est $\ell$-irr\'eductible et de r\'eduction supercuspidale.
\item \emph{$\ell$-superSpeh} si elle est  $\ell$-irr\'eductible de
  r\'eduction superSpeh.
\item \emph{congrue} \`a une autre repr\'esentation (modulo $\ell)$, si les deux
  repr\'esentations sont $\ell$-enti\`eres et leurs r\'eductions
  semi-simplifi\'ees sont isomorphes.
\item \emph{strictement congrue} \`a une autre repr\'esentation, si
  elle lui est congrue et si leurs caract\`eres centraux
  co\"{\i}ncident sur la matrice diagonale $\varpi. I_{d}$, o\`u $I_{d}$
  d\'esigne la matrice identit\'e de taille $d$.
\end{itemize}
Concernant les repr\'esentations supercuspidales, on sait que :
\begin{enumerate}
\item Une $\o\FM_{\ell}$-repr\'esentation cuspidale se rel\`eve \`a
  $\o\QM_{\ell}$ et tout rel\`evement est supercuspidal, \cite[III.5.10.2)]{Vig}.
\item Une $\o\QM_{\ell}$-repr\'esentation supercuspidale $\pi$ est $\ell$-enti\`ere \ssi\ son
  caract\`ere central est $\ell$-entier. Dans ce cas, $\pi$ est
  $\ell$-irr\'eductible, \cite[III.1.1.d)]{Vig}, et on a les propri\'et\'es
  suivantes, \cite[2.3]{VigAENS} :
  \begin{enumerate}
   \item le nombre $m(\pi)$ de repr\'esentations strictement congrues \`a
     $\pi$ est inf\'erieur \`a la plus grande puissance de $\ell$ divisant
     le nombre   $a(\pi):=\frac{d}{t(\pi)}(q^{t(\pi)}-1)$ dans lequel $t(\pi)$
     d\'esigne le nombre 
     de caract\`eres non ramifi\'es
    $\psi:K^{\times}\To{}\o\QM_{\ell}^{\times}$ tels que $\pi\simeq
    \pi\otimes(\psi\circ\rm det)$.
  \item $\pi$ est $\ell$-supercuspidale \ssi\ l'in\'egalit\'e pr\'ec\'edente
    est une \'egalit\'e.
  \end{enumerate}
\end{enumerate}

% Soit $\varpi$ une uniformisante de $K$. On dira qu'une repr\'esentation irr\'eductible
% de $G_{d}$ est $\varpi$-normalis\'ee si son caract\`ere central vaut $1$
% sur la matrice diagonale ${\rm Diag}(\varpi,\cdots,\varpi)$.

% \begin{lemme} 
%   \begin{enumerate}
%   \item Toute repr\'esentation de Speh est la torsion d'une
%     repr\'esentation de Speh $\varpi$-normalis\'ee par un caract\`ere non ramifi\'e.
%   \item Une $\o\QM_{\ell}$-repr\'esentation de Speh $\varpi$-normalis\'ee est
%     $\ell$-irr\'eductible.  Sa r\'eduction est une
%     $\o\FM_{\ell}$-repr\'esentation de Speh ($\varpi$-normalis\'ee) \ssi\ son support
%     supercuspidal est $\ell$-supercuspidal.
%   \item Tout rel\`evement d'une
%     $\o\FM_{\ell}$-repr\'esentation de Speh est une
%     $\o\QM_{\ell}$-repr\'esentation de Speh.
%   \end{enumerate}
% \end{lemme}

\begin{lem} \label{lemrelev}
  Toutes les propri\'et\'es pr\'ec\'edentes sont vraies en rempla\c{c}ant partout
  (super)cuspidale par (super)Speh.
% , et $d$ par $d'$ dans la
%   formule pour $a(\pi)$, o\`u $d'=d/r$ si $\pi=\delta_{r}(\pi_{0})$.
\end{lem}
\begin{proof}
% i) Soit $\delta_{r}(\pi)$ une $\o\FM_{\ell}$-repr\'esentation
% de Speh. Choisissons un rel\`evement $\tilde\pi$ de la repr\'esentation
% cuspidale $\pi$ \`a $\QM_{\ell}$. La repr\'esentation $\delta_{r}(\tilde\pi)$
% est $\ell$-enti\`ere, comme d'ailleurs tous les sous-quotients de
% $\tilde\pi\times\cdots\times\tilde\pi\nu^{r-1}$. Comme l'ensemble
% $\Lambda(\delta_{r}(\tilde\pi))$ est un singleton, elle est
% $\ell$-irr\'eductible et sa r\'eduction est isomorphe \`a $\delta_{r}(\pi)$
% puisqu'elle satisfait les propri\'et\'es m1) et
% m2) relatives \`a $(\pi,r)$.

i) Par d\'efinition, une repr\'esentation de Speh se rel\`eve, \emph{cf.}
proposition \ref{defsuperSpeh}. V\'erifions que
ses rel\`evements sont tous des repr\'esentations de  Speh.
Soit $\tilde\delta$
un rel\`evement de $\delta=\delta_{r}(\pi)$. Posant $d'=d/r$,
on voit que  $r_{P_{d',\cdots,  d'}}(\tilde\delta)$ est un rel\`evement de
$\pi\otimes\pi\nu\otimes\cdots\otimes \pi{\nu^{r-1}}$. Or, par
exactitude des d\'eriv\'ees, on a
$\Lambda(\tilde\delta)=\Lambda(\pi)=\{(d',d',\cdots,d')\}$. Le
param\`etre de Zelevinsky de $\tilde\delta$ est donc  un segment,  donc
$\sigma$ est une repr\'esentation de Speh.

ii) 
Pour le pr\'eambule de la propri\'et\'e ii), partons de $\delta=\delta_{r}(\pi)$ une $\o\QM_{\ell}$-repr\'esentation
de Speh. On sait qu'elle est $\ell$-enti\`ere \ssi\ son support
cuspidal l'est donc \ssi\ $\pi$ est $\ell$-enti\`ere. On v\'erifie ais\'ement que cela
\'equivaut \`a ce que son caract\`ere central soit $\ell$-entier. Dans ce
cas, d'apr\`es la proposition \ref{defsuperSpeh}, sa r\'eduction est
irr\'eductible, \'egale par d\'efinition \`a $\delta_{r}(\o\pi)$ o\`u $\o\pi$
est la r\'eduction de $\pi$.

Passons aux propri\'et\'es ii)(a) et ii)(b). Gardons les notations
ci-dessus, et soit  $\delta'$ une repr\'esentation $\ell$-enti\`ere de m\^eme
r\'eduction que $\delta$. 
Par le point i), il existe une $\oQl$-repr\'esentation $\pi'$ cuspidale
enti\`ere de $G_{d/r}$ telle que  $\delta'=\delta_{r}(\pi')$. Par la
propri\'et\'e m1) d'une repr\'esentation de Speh, $\pi'$ est uniquement
d\'etermin\'ee par $\delta'$ et est congrue \`a $\pi$.
On obtient de la sorte une injection de l'ensemble des
repr\'esentations congrues \`a $\delta$ dans celui des
repr\'esentations congrues \`a $\pi$. D'apr\`es la proposition
\ref{defsuperSpeh}, c'est m\^eme une bijection.
Par l'\'egalit\'e $\delta(\varpi I_{d})= \pi(\varpi I_{d'})^{r}
q^{-d'r(r-1)/2}$, on en d\'eduit une bijection entre l'ensemble des
repr\'esentations \emph{strictement} congrues \`a $\delta$ et l'ensemble des
repr\'esentations $\pi'$ congrues \`a $\pi$ v\'erifiant $\pi(\varpi
I_{d})^{r}=\pi'(\varpi I_{d})^{r}$.
Notons que cette derni\`ere condition, jointe \`a la congruence 
$\pi(\varpi
I_{d})\equiv \pi'(\varpi I_{d})[\ell]$ \'equivaut \`a la condition 
$\pi(\varpi I_{d})^{\ell^{{\rm val}_\ell(r)}}=\pi'(\varpi I_{d})^{\ell^{ {\rm val}_\ell(r)}}$.
Il s'ensuit que $m(\delta)=\ell^{{\rm val}_\ell(r)}.m(\pi)$.

Par ailleurs, comme $\delta_{r}(\pi\otimes(\psi\circ {\rm det}))=
\delta_{r}(\pi)\otimes (\psi\circ {\rm det})$, on a $t(\delta)=t(\pi)$.

Notons $[n]_{ \ell}:= \ell^{{\rm val}_\ell(n)}$ la plus grande
puissance de $\ell$ divisant l'entier $n$.
On a obtenu l'in\'egalit\'e
$$ m(\delta))=[r]_{\ell} m(\pi)\leq [r]_{\ell}
\left[\frac{d'}{t(\pi)}(q^{t(\pi)}-1)\right]_{\ell} =
\left[\frac{d}{t(\delta)}(q^{t(\delta)}-1)\right]_{\ell}
$$
avec \'egalit\'e \ssi\ $\pi$ est $\ell$-supercuspidale, donc \ssi\
$\delta$ est $\ell$-superSpeh.
\end{proof}

\alin{Classification} \label{classification}
Un multi-$C$-segment (super)cuspidal est un multi-ensemble de
$C$-segments (super)cuspidaux, \emph{i.e.} un ensemble de $C$-segments avec multiplicit\'es, ou
en termes plus rigoureux,
un \'el\'ement du mono\"{\i}de libre
de base l'ensemble des $C$-segments au sens ci-dessus. Soit
$a=\{(\pi_{i},r_{i}), i=1,\cdots, t\}$ un multi-segment supercuspidal, o\`u
$\pi_{i}$ est une supercuspidale de $G_{d_{i}}$. Il d\'etermine une
partition $\lambda_{a}=(d_{1}^{(r_{1})}+\cdots + d_{t}^{(r_{t})})$ de
l'entier $d(a):=r_{1}d_{1}+\cdots + r_{t}d_{t}$, que nous appellerons ``longueur'' de $a$ .

Notons 
$\pi(a):=\delta_{r_{1}}(\pi_{1})\times\cdots \times
\delta_{r_{t}}(\pi_{t})$, qui est une repr\'esentation de $G_{d(a)}$,
induite depuis le parabolique standard associ\'ee \`a la partition
\emph{transpos\'ee} de $\lambda_{a}$.
L'ensemble $\Lambda(\pi(a))$ a pour plus grand
\'el\'ement $\lambda_{a}$
et on a $m_{\pi(a),\lambda_{a}}=1$. Par cons\'equent la
repr\'esentation $\pi(a)$ poss\`ede un unique sous-quotient irr\'eductible
$\langle a \rangle$ tel que $\lambda_{\langle a
  \rangle}=\lambda_{a}$.

Le th\'eor\`eme de classification affirme que 
\begin{center}
  \emph{l'application $a\mapsto \langle a\rangle$ induit une bijection
    de l'ensemble des multi-$C$-segments supercuspidaux de longueur
    $d$ sur l'ensemble $\Irr{C}{G_{d}}$.}
\end{center}

Lorsque $C=\o\QM_{\ell}$, c'est\cite[(6.1).d)]{Zel}, et lorsque
$C=\o\FM_{\ell}$, c'est \cite[V.12]{VigInduced} ou \cite{MinSech}.
Notons que la d\'efinition originale de $a$ et $\pi(a)$ dans
 \cite[Thm (6.1).a)]{Zel} coincide bien avec celle ci-dessus en vertu
 de \cite[Thm (8.1)]{Zel}.

\begin{prop} \label{propdecomp}
 Soit $G=\GL_{d}(K)$ et $C=\o\FM_{\ell}$ ou $\o\QM_{\ell}$. 
  Notons $\RC_{I}(G,C)$ le sous-groupe de
  $\RC(G,C)$  engendr\'e par les repr\'esentations induites 
paraboliques propres et $\RC_{\Delta}(G,C)$ celui engendr\'e par les
repr\'esentations superSpeh. On a une d\'ecomposition
$$ \RC(G,C)= \RC_{I}(G,C)\oplus \RC_{\Delta}(G,C).$$
\end{prop}
\begin{proof}
% Le th\'eor\`eme de classification fournit deux bases de
% $\RC(G,C)$ index\'ees par les multi-segments supercuspidaux de longueur
% $d$. En effet, si l'on \'ecrit la famille de ``modules
% standards'' $(\pi(a))_{a}$ dans 
% la base usuelle des irr\'eductibles $(\langle a\rangle)_{a}$, et
%  si l'on
% munit l'ensemble des multi-segments supercuspidaux de longueur $d$ de l'ordre partiel
% $a\leq a' \Leftrightarrow (\lambda_{a}\leq \lambda_{a'}$, alors la
% matrice carr\'ee obtenue  est triangulaire sup\'erieure avec des $1$ sur
% la diagonale, donc est inversible.
% Or, la repr\'esentation $\pi(a)$ est une induite parabolique propre 
% \ssi\ $a$ contient au moins deux segments, et est superSpeh
% \ssi\ $a$ contient un seul segment. On en d\'eduit 
% que 
Pour une partition $\lambda$ de $n$, notons $\RC(G,C)_{\leq \lambda}$
le sous-groupe de $\RC(G,C)$ engendr\'e par les classes de repr\'esentations
irr\'eductibles $\pi$ telles que $\lambda_{\pi}\leq \lambda$. 
Soit $\pi$ une telle repr\'esentation et $a$ le multisegment associ\'ee
par le th\'eor\`eme de classification rappel\'e ci-dessus. On a donc
$\pi=\langle a\rangle$. Si $\pi$ n'est pas superSpeh,
alors la repr\'esentation $\pi(a)$ est induite parabolique propre, et
par construction on a
$$ [\pi] =[\langle a\rangle] \in [\pi(a)] + \sum_{\lambda'<\lambda}
\RC(G,C)_{\leq \lambda'}.$$
On en d\'eduit que
$$\RC(G,C)_{\leq \lambda} \subset \RC_{\Delta}(G,C) + \RC_{I}(G,C) +
\sum_{\lambda' < \lambda}
  \RC(G,C)_{\leq \lambda'}.$$
Puisque $\RC(G,\lambda)_{\leq 1^{(d)}}\subset \RC_{\Delta}(G,\lambda)$,
et puisque l'ensemble des partitions est fini, il vient par r\'ecurrence l'\'egalit\'e
\begin{equation}
\RC(G,C)= \RC_{I}(G,C) + \RC_{\Delta}(G,C).\label{eq:1}
\end{equation}
Reste \`a voir que
$\RC_{I}(G,C) \cap \RC_{\Delta}(G,C)=\{0\}$. Pour cela, il suffit de
voir que $\RC_{I}(G,C)$ est engendr\'e par les repr\'esentations $\pi(a)$
associ\'ees aux multisegments $a$ contenant au moins deux segments.
Or, par r\'ecurrence sur $d$, l'\'egalit\'e (\ref{eq:1}) implique la
suivante 
$$\RC_{I}(G,C)= \sum_{M<G}i_{M}^{G}(\RC_{\Delta}(M,C)).$$
Ici, $M$ d\'ecrit les sous-groupes de Levi diagonaux par blocs et
$i_{M}^{G}$ d\'esigne l'induction parabolique normalis\'ee le long du
parabolique sup\'erieur.
Mais si $\delta=\delta_{1}\otimes\cdots\otimes\delta_{r}$ est une
repr\'esentation superSpeh de $M =G_{d_{1}}\times\cdots\times
G_{d_{r}}$, avec $\delta_{i}=\delta_{r_{i}}(\pi_{i})$, alors
 $[i_{M}^{G}(\delta)]=[\pi(a)]$ pour $a =
 \{(\pi_{i},r_{i}), i=1,\cdots, r\}$.

% On utilise ici la d\'ecomposition selon le support supercuspidal
% $$\RC(G,C)=\bigoplus_{[M,\pi]}
% \RC(G,C)_{[M,\pi]}$$
% o\`u $[M,\pi]$ d\'ecrit l'ensemble des classes de $G$-conjugaison de paires
% Levi-supercuspidale, et l'indice $[M,\pi]$ d\'esigne le groupe engendr\'e
% par les repr\'esentations de support supercuspidal  $[M,\pi]$.
% Comme les sous-quotients de l'induite parabolique d'une repr\'esentation
% irr\'eductible ont tous le m\^eme support supercuspidal,  on a $\RC_{I}(G,C) =
% \bigoplus_{[M,\pi]}
% (\RC_{I}(G,C)\cap\RC(G,C)_{[M,\pi]})$
% Par ailleurs, deux repr\'esentations superminimales $\delta,\delta'$ ont m\^eme
% support supercuspidale  \ssi\ il existe un caract\`ere non ramifi\'e
% $\psi$ de $G$ tel que $\delta'=\delta\otimes \psi$. 

%  En effet, soit
% $\pi\in\Irr{\o\FM_{\ell}}{G}$ et $a$ son multi-segment supercuspidal,
% de sorte que $\pi\simeq \langle a\rangle$.  Si $\langle a \rangle =
% \pi(a)$, on a deux possibilit\'es : soit $a$ contient au moins deux
% segments, auquel cas  $\pi\in\RC_{I}(G,\o\FM_{\ell})$, soit $a$
% contient un seul segment, auquel cas $\pi\in \RC_{\Delta}(G,\o\FM_\ell)$. Dans chacun
% de ces cas, la d\'ecomposition cherch\'ee est donc triviale. 
% On raisonne alors par r\'ecurrence sur $\lambda_{\pi}$. En effet 
% on sait que tout sous-quotient irr\'eductible $\rho$ de $\pi(a)$
% diff\'erent de $\pi=\langle a \rangle$ v\'erifie
% $\lambda_{\rho}<\lambda_{\pi}$. La r\'ecurrence montre donc que $\pi\in
% \RC_{I}+\RC_{\Delta}(G,\o\FM_\ell)$.
\end{proof}

\begin{coro} \label{rlsurjectif}
  Pour le groupe $G=\GL_{d}(K)$, le morphisme de r\'eduction de
  Brauer-Nesbitt $r_{\ell}:\, \RC^{\rm
  ent}_{\ell'}(G,\o\QM_{\ell})\To{}\RC(G,\o\FM_{\ell})$ est surjectif.
\end{coro}
\begin{proof}
Comme l'induction parabolique commute \`a $r_{\ell}$, 
 une r\'ecurrence sur $d$ nous ram\`ene \`a prouver que $\RC_{\Delta}(G,\o\FM_\ell)$ est
 dans l'image de $r_{\ell}$, ce qui d\'ecoule de la d\'efinition
 \ref{defsuperSpeh} des repr\'esentations superSpeh..
\end{proof}

\subsection{Repr\'esentations de $D^{\times}$}
Ici $D$ d\'esigne une alg\`ebre \`a division de centre $K$ et de dimension
$d^{2}$. Le groupe $D^{\times}$ \'etant compact modulo son centre,
certaines propri\'et\'es des groupes finis s'y \'etendent facilement.

\begin{prop}\label{propD1}
  \begin{enumerate}
  \item L'homomorphisme
    $\tilde\theta:\,\RC(D^{\times},\o\FM_{\ell})\To{}
    \CC^{\infty}(D^{\rm ell},\o\ZM_{\ell})^{D^{\times}}$ est injectif.
  \item Toute $\o\FM_{\ell}$-repr\'esentation irr\'eductible se rel\`eve \`a
    $\o\QM_{\ell}$. En particulier, l'homorphisme de r\'eduction
    $r_{\ell}:\; \RC(D^{\times},\o\QM_{\ell})\To{}
    \RC(D^{\times},\o\FM_{\ell})$ est surjectif.
  \end{enumerate}
\end{prop}
\begin{proof}
i) Soit $x\in\ker(\tilde\theta)$. Comme les
$\o\FM_{\ell}$-repr\'esentations lisses irr\'eductibles de $D^{\times}$
sont d'image finie,  l'\'el\'ement $x$ est
l'inflation d'une repr\'esentation virtuelle $y$ d'un quotient fini $D^{\times}/H$ de
$D^{\times}$ par un pro-$p$-groupe. Les \'el\'ements elliptiques \'etant denses dans $D^{\times}$,
la projection induit une surjection de $D^{\rm ell}_{\ell'}$ sur
l'ensemble des \'el\'ements d'ordre premier \`a $\ell$ de $D^{\times}/H$. Ainsi le
caract\`ere de Brauer (usuel) de $y$ est nul, et par cons\'equent $y=0$ (voir
par exemple \cite[Cor. (17.10)]{CR}). Donc $x=0$ aussi.

ii) C'est le th\'eor\`eme de Fong-Swan appliqu\'e aux quotients finis de
$D^{\times}$, lesquels sont tous r\'esolubles.
\end{proof}

Dans la proposition ci-dessous, on utilise la m\^eme terminologie que
celle du paragraphe \ref{terminologie}, sauf les notions de
$\ell$-supercuspidales et $\ell$-superSpeh, qui n'ont gu\`ere
d'int\'er\^et ici. 
%Pour la notion de
%$\ell$-congruence \emph{stricte}, on demande que 
En particulier, on sait que  $\rho\in\Irr{\o\QM_{\ell}}{D^{\times}}$
est $\ell$-enti\`ere \ssi\ son caract\`ere central l'est, et dans ce cas, on note
\begin{itemize}
\item $t(\rho)$ le nombre de caract\`eres $\psi:\,
  K^{\times}\To{}\o\QM_{\ell}^{\times}$ tels que $\rho\simeq
  \rho\otimes(\psi\circ{\rm Nrd})$.
\item $m(\rho)$ le nombre de repr\'esentations irr\'eductibles qui sont
  strictement congrues \`a $\rho$.
\item $a(\rho):=\frac{d}{t(\rho)}(q^{t(\rho)}-1)$.
\end{itemize}

\begin{prop}\label{propD2}
  Soit $\rho\in\Irr{\o\QM_{\ell}}{D^{\times}}$ $\ell$-enti\`ere. Alors $m(\rho)$ est
  inf\'erieur \`a la plus grande puissance de $\ell$ qui divise $a(\rho)$,
  et lui est \'egal \ssi\ $\rho$ est $\ell$-irr\'eductible.
\end{prop}
\begin{proof}
 Nous adaptons \`a $D^{\times}$ les arguments de Vign\'eras pour les repr\'esentations du
 groupe de Weil de $K$. 
On peut supposer, quitte \`a tordre par un caract\`ere non ramifi\'e, que le
caract\`ere central de $\rho$ est trivial sur l'uniformisante $\varpi$, \emph{i.e.} que
$\rho$ est une repr\'esentation de $D^{\times}/\varpi^{\ZM}$. Soit
$\PC_{D}$ le radical de $\OC_{D}$. On a
une filtration
$$ 1+\PC_{D}\subset \OC_{D}^{\times} \subset
D^{\times}/\varpi^{\ZM} $$
de quotients successifs isomorphes \`a $\FM_{q^{d}}^{\times}$ et
$\ZM/d\ZM$, l'action de ce dernier sur le pr\'ec\'edent \'etant donn\'ee par
le Frobenius.

Soit $\tau$ un facteur irr\'eductible de $\rho_{|1+\PC_{D}}$ et
$N_{\tau}$ le normalisateur de la classe d'isomorphisme de $\tau$ dans
$D^{\times}/\varpi^{\ZM}$. Comme $1+\PC_{D}$
 est un pro-$p$-groupe, la dimension de $\tau$ est une puissance de
$p$. Comme un $p$-Sylow de $N_{\tau}/(1+\PC_{D})$ est cyclique
l'argument de \cite[Lemme 1.19]{VigLuminy} montre que $\tau$ admet un
prolongement $\tilde\tau$ \`a $N_{\tau}$. 
Soit alors $\xi$ une sous-représentation irréductible de $N_{\tau}$
agissant sur $\Hom_{1+\PC_{D}}(\tilde\tau,\rho)$. Comme l'ensemble
d'entrelacement $I_{\tilde\tau\otimes\xi}$ de $\tilde\tau\otimes\xi$
est égal à $N_{\tau}$, l'homorphisme
 $\tilde\tau\otimes \xi\To{}\rho$ induit par adjonction 
un isomorphisme
$\cind{N_{\tau}}{D^{\times}/\varpi^{\ZM}}{\tilde\tau\otimes\xi}\simto \rho$.
La représentation $\xi$ est de la forme $\xi\simeq
\cind{J}{N_{\tau}}{\chi}$ pour un groupe $J$ contenant
$N_{\tau}\cap \OC_{D}^{\times}$  et
un caractère $\chi$ de $J$ trivial sur $1+\PC_{D}$. 
On peut alors écrire
$$\rho=\cind{J}{D^{\times}/\varpi^{\ZM}}{\tilde\tau_{|J}\otimes \chi}.$$
L'ensemble d'entrelacement
de $\tilde\tau_{|J}\otimes\chi$ est
$I_{\tilde\tau\otimes\tilde\chi}=N_{\tau}\cap N_{\chi}$, où
$N_{\chi}$ désigne le normalisateur de $\chi$.
 On a donc $N_{\tau}\cap N_{\chi}=J$ par irréductibilité de $\rho$.

Calculons maintenant l'entier $t(\rho)$. Soit $\psi$ un caractère de
$D^{\times}/\varpi^{\ZM}\OC_{D}^{\times}$. La représentation
$\rho\psi$ est isomorphe à $\rho$ \ssi\ il existe $x\in N_{\tau}$
tel que $\psi_{|J}\chi={\chi}^{x}$. Pour un tel $x$, il faut alors que
$\chi^{-1}\chi^{x}$ soit trivial sur $J\cap\OC_{D}^{\times}=N_{\tau}\cap\OC_{D}^{\times}$.
En identifiant $J/(1+\PC_{D})$ au sous-groupe
$(N_{\tau}\cap\OC_{D}^{\times})/(1+\PC_{D})\rtimes J/(N_{\tau}\cap
\OC_{D}^{\times})$ du groupe
$N_{\tau}/(1+\PC_{D})\simeq (N_{\tau}\cap\OC_{D}^{\times})/(1+\PC_{D})\rtimes N_{\tau}/(N_{\tau}\cap
\OC_{D}^{\times})$, et en notant que $(N_{\tau}\cap
\OC_{D}^{\times})$ est abélien, on
constate que $\chi^{x}\chi^{-1}$ est alors trivial sur $J$ tout
entier, et donc que $x\in J$. Il s'ensuit que 

$$t(\rho) = [D^{\times}/\varpi^{\ZM}: \OC_{D}^{\times}J].$$

Analysons maintenant la réduction $r_{\ell}(\rho)$. Elle est de la
forme $\cind{J}{D^{\times}/\varpi^{\ZM}}{r_{\ell}(\tilde\tau_{|J})\otimes
  r_{\ell}(\chi)}$. 
Commet $\tau$
est une
représentation irréductible  du \emph{pro-$p$-groupe}
$1+\PC_{D}$, la représentation $r_{\ell}(\tilde\tau_{|J})$ est
irréductible, et son normalisateur est encore $N_{\tau}$. 
On en
déduit que la longueur de $r_{\ell}(\rho)$ est égale à $[N_{\tau}\cap N_{r_{\ell}(\chi)} : J]$.
En particulier :
\begin{center}
  $\rho$ est $\ell$-irréductible \ssi\ $N_{\tau}\cap
  N_{r_{\ell}(\chi)}=J$.
\end{center}

Soit maintenant $\rho'$ une représention congrue à $\rho$. Comme elle
contient $\tau$, on peut la mettre sous la forme
$\rho'=\cind{J'}{D^{\times}/\varpi^{\ZM}}{\tilde\tau_{|J'}\otimes 
  \tilde\chi'}$.
Comme $\rho'$ a la
même dimension que $\rho$, $J'$ a le même indice que $J$ dans
$N_{\tau}$, et puisqu'il contient $N_{\tau}\cap \OC_{D}^{\times}$ et
que le quotient $N_{\tau}/(N_{\tau} \cap \OC_{D}^{\times})$ est
cyclique, on a $J=J'$.
Il s'ensuit que $r_{\ell}(\chi')$ et
$r_{\ell}(\chi)$ sont conjugués sous $N_{\tau}$. On en déduit
une bijection entre l'ensemble des $\rho'$ congrues à $\rho$ et
l'ensemble des $\chi'$ congrus à $\chi$ modulo conjugaison par
$N_{\tau}\cap N_{r_{\ell}(\chi)}$. En particulier on a 
$$ m(\rho)\leq m(\chi) \hbox{ avec \'egalit\'e \ssi\ }
N_{\tau}\cap N_{r_{\ell}(\chi)}=J.$$

Reste à calculer le nombre  $m(\chi)$ de caractères de
$J$ congrus à $\chi$. Il est clair que c'est la plus
grande puissance de $\ell$ divisant l'ordre de l'abélianisé de
$J/(1+\PC_{D})$. Pour calculer cet ordre, nous pouvons utiliser la
construction explicite de Broussous : nous pouvons supposer que
$(J,\tilde\tau_{|J}\otimes \chi)$ est de la forme
\cite[(10.1.4)]{Broussous}. 
Par construction, il existe des entiers $f'$, $d'$ et
$e'$ de produit égal à $d$
tels que 
\begin{equation}
J/(1+\PC_{D})=\FM_{q^{f'd'}}^{\times}\rtimes m\ZM/e'd'\ZM,\label{brou}
\end{equation}
où le
générateur de $\ZM/e'd'\ZM$ agit par le Frobenius relatif à
$\FM_{q^{f'}}$. Avec les notations de \emph{loc. cit}, $f'$ est le
degré résiduel de l'extension $F[\beta]$, $e'$ est son indice de
ramification, et $(d')^{2}$ est la dimension sur $F[\beta]$ du commutant $B$ de
$F[\beta]$ dans $D$ (une algèbre à division de centre
$F[\beta]$). L'entier $m$ est celui de la définition
\cite[(10.1.4)]{Broussous}. C'est un diviseur de $d'$ qui satisfait la relation
$$f'm=[D^{\times}/\varpi^{\ZM}\OC_{D}^{\times}J]=t(\rho).$$
Maintenant l'abélianisé de $J/(1+\PC_{D})$
s'identifie \`a $\FM_{q^{f'm}}^{\times}\times
m\ZM/e'd'\ZM$ (le morphisme quotient \'etant induit par la norme).
Par les égalités précédentes, son ordre est bien  $(q^{t(\rho)}-1)d/t(\rho)$.
\end{proof}

\`A partir de \ref{brou}, la discussion de
\cite{VigLuminy} pages 423 et 424 fournit l'analogue suivant de
\cite[1.20]{VigLuminy}.

\begin{prop}
 Soit $\rho$ une $\o\QM_{\ell}$-repr\'esentation irr\'eductible
 $\ell$-enti\`ere de $D^{\times}$. Il existe une
 $\o\FM_{\ell}$-repr\'esentation irr\'eductible $\bar\rho$ de $D^{\times}$
 et un entier $r$ tels que 
$$ r_{\ell}([\rho])= [\bar\rho] + [\bar\rho \nu] + \cdots +
[\bar\rho\nu^{r-1}] $$
o\`u $\nu$ d\'esigne le caract\`ere $g\mapsto q^{{\rm val\circ Nrd}(g)}$.
\end{prop}

\section{Correspondances $\ell$-modulaires}

Nous reprenons les notations de l'introduction. En particulier $G$
d\'esigne le groupe $\GL_{d}(K)$ et $D$ une
alg\`ebre \`a division de centre $K$ et de dimension $d^{2}$.

\subsection{Preuve des \'enonc\'es principaux}

\alin{Langlands-Jacquet classique et int\'egralit\'e} \label{lentier}
% Le principe de base vient du transfert des classes de conjugaison de
% $D^{\times}$ vers $G$. L'application ``polyn\^ome minimal'' induit en
% effet une bijection des classes de conjugaison elliptiques de
% $D^{\times}$ vers celles de $G$. Celle-ci induit \`a son tour un isomorphisme
% $$ \CC^{\infty}(G^{\rm ell},\o\QM_{\ell})^{G}\simto\CC^{\infty}(D^{\rm ell},\o\QM_{\ell})^{D^{\times}}.$$ 
On rappelle que $\LJ_{\oQl}$ est l'unique fl\`eche rendant
commutatif le diagramme suivant
\begin{equation}
\xymatrix{\CC^{\infty}(G^{\rm ell},\o\QM_{\ell})^{G} \ar[r]^{\sim} &
  \CC^{\infty}(D^{\rm ell},\o\QM_{\ell})^{D^{\times}} \\
\RC(G,\o\QM_{\ell}) \ar[u]^{\theta^{G}} \ar@{..>}[r]_{\LJ_{\oQl}} &  \RC(D^{\times},\o\QM_{\ell}) \ar[u]_{\theta^{D^{\times}}}
}\label{diagLJ}
\end{equation}
et que cette fl\`eche envoie une irr\'eductible $\pi$ sur $0$ ou sur $\pm[\rho]$
pour une irr\'eductible $\rho$ de $D^{\times}$. Identifions les deux centres $Z(G)$ et
$Z(D^{\times})$ \`a $K^{\times}$. Comme leur action \`a droite (ou \`a
gauche) pr\'eserve les lieux elliptiques respectifs, on voit que $\pi$
et $\rho$ ont des caract\`eres centraux \'egaux.
Or, si $\pi$ est $\ell$-enti\`ere, en particulier son caract\`ere central
est \`a valeurs dans $\oZl$.  Comme une repr\'esentation irr\'eductible
de $D^{\times}$ est enti\`ere \ssi\ son caract\`ere central l'est, il
s'ensuit que $\LJ_{\oQl}$ envoie 
$\RC^{\rm ent}(G,\o\QM_{\ell})$ dans $\RC^{\rm ent}(D^{\times},\o\QM_{\ell})$.

% \begin{theo}\label{theoLJ}
% Il existe un morphisme de groupes ab\'eliens  $\RC(G,\o\FM_{\ell})
% \To{\LJ_{d,\o\FM_{\ell}}} \RC(D^{\times},\o\FM_{\ell})$ rendant
% commutatifs les deux diagrammes suivants :
% $$
% \xymatrix{\CC^{\infty}(G^{\rm ell}_{\ell'},\o\ZM_{\ell})^{G} \ar[r]^{\sim} &
%   \CC^{\infty}(D^{\rm ell}_{\ell'},\o\ZM_{\ell})^{D^{\times}} \\
%   \RC(G,\o\FM_{\ell}) \ar[u]^{\tilde\theta^{G}}
%   \ar@{..>}[r]_{\LJ_{d,\o\FM_{\ell}}} &  \RC(D^{\times},\o\FM_{\ell})  \ar[u]_{\tilde\theta^{D^{\times}}}
% }
% \;\,\hbox{ et }\,\; 
% \xymatrix{
% \RC^{\rm ent}(G,\o\QM_{\ell}) \ar[d]_{r_{\ell}^{G}} \ar[r]_{\LJ_{d}} &
% \RC^{\rm ent}(D^{\times},\o\QM_{\ell}) \ar[d]^{r_{\ell}^{D^{\times}}} \\
% \RC(G,\o\FM_{\ell})  \ar@{..>}[r]_{\LJ_{d,\o\FM_{\ell}}} &
% \RC(D^{\times},\o\FM_{\ell})  \\
% }.
% $$
% De plus, ce morphisme est uniquement d\'etermin\'e par l'un ou l'autre de
% ces diagrammes.
% \end{theo}

% Rappelons que la notation $\tilde\theta$ d\'esigne le caract\`ere de
% Brauer, \emph{cf} le th\'eor\`eme \ref{theoBrauer} et la proposition \ref{propBrauer}.

%\begin{proof}
\alin{Preuve du th\'eor\`eme \ref{theoLJ}} \label{preuvetheoLJ}
Montrons l'existence d'un unique morphisme $\LJ_{\o\FM_{\ell}}$
rendant commutatif le second diagramme de l'\'enonc\'e du th\'eor\`eme.
 L'unicit\'e r\'esulte de la surjectivit\'e de $r_{\ell}^{G}$ du corollaire
 \ref{rlsurjectif}. Pour l'existence, toujours gr\^ace \`a cette surjectivit\'e, il suffit de prouver
que $\ker(r_{\ell}^{G})\subset  \ker(r_{\ell}^{D}\circ \LJ_{\oQl})$.
Soit donc $x\in \ker(r_{\ell}^{G})$. D'apr\`es le second diagramme de la proposition
\ref{propBrauer}, on a ${\theta_{x}}_{|G^{\rm ell}_{\ell'}}\equiv
0$. On en d\'eduit que
$\tilde\theta_{r_{\ell}^{D^{\times}}(\LJ_{\oQl}(x))}= {\theta_{\LJ_{\oQl}(x)}}_{|D^{\rm ell}_{\ell'}} \equiv
0$. D'apr\`es le i) de la proposition \ref{propD1} on a bien
 $r_{\ell}^{D^{\times}}(\LJ_{\oQl}(x))=0$.

La commutativit\'e du premier diagramme de l'\'enonc\'e du th\'eor\`eme d\'ecoule maintenant de celle du
diagramme (\ref{diagLJ}) et de celle du second diagramme de la proposition
\ref{propBrauer}.
Enfin, le fait que ce premier diagramme suffise \`a caract\'eriser
$\LJ_{\o\FM_{\ell}}$ provient \`a nouveau de l'injectivit\'e de $\tilde\theta^{D^{\times}}$.
\findem
%\end{proof}

\alin{Jacquet-Langlands classique et r\'eduction mod $\ell$} 
Les propri\'et\'es caract\'eristiques de l'application $\JL_{\oQl}:\,
\Irr{\oQl}{D^{\times}}\To{}\Irr{\oQl}{G}$  rappel\'ees
au paragraphe \ref{corclassic} font de son prolongement
par lin\'earit\'e, que nous noterons encore $\JL_{\oQl}$,  une
\emph{section remarquable} de $\LJ_{\oQl}$.
Comme les
$\oQl$-repr\'esentations irr\'eductibles de la ``s\'erie discr\`ete'' de $G$  
-- \emph{i.e.} les repr\'esentations de Steinberg g\'en\'eralis\'ees --
 sont $\ell$-enti\`eres \ssi\ leur caract\`ere
central l'est, on voit que l'application
$\JL_{\oQl}$  envoie une repr\'esentation enti\`ere sur une repr\'esentation
enti\`ere. Par contre, l'exemple suivant montre que la section
$\JL_{\oQl}$ de $\LJ_{\oQl}$ n'est pas
compatible \`a la r\'eduction modulo $\ell$.

\begin{exe}\label{exempleLJ}
  Soit $K=\QM_{5}$, $\ell=3$ et $d=2$, et soit
  $\chi:\FM_{5^{2}}^{\times}\To{} \o\QM_{3}^{\times}$ un caract\`ere d'ordre
  $3$. Il se prolonge \`a un caract\`ere de $\OC_{D}^{\times}Z(D^{\times})$ qui par
  induction fournit une $\o\QM_{3}$-repr\'esentation irr\'eductible $\rho$
  de dimension $2$ de $D^{\times}$. Cette repr\'esentation est \'evidemment congrue modulo
  $3$ \`a l'induite du caract\`ere trivial, laquelle se d\'ecompose en la
  somme directe du caract\`ere trivial $1$ et du caract\`ere ``signe''
  $\varepsilon$ de  $D^{\times}$. Or on a
$$ \JL_{\oQl}([\rho] - [1] -[\varepsilon])= [\pi] - [{\rm St}] - [\varepsilon\otimes
{\rm St}]$$
o\`u ${\rm St}$ d\'esigne la repr\'esentation de Steinberg et $\pi$ d\'esigne
une repr\'esentation supercuspidale de r\'eduction unipotente, \emph{cf}
\cite[Ex. II.11.1]{Vig}. La r\'eduction
modulo $\ell$ du terme de droite est \'egale \`a $-[{\rm Ind}_{B}^{G}(1)]$
et n'est donc pas nulle.
\end{exe}

% \ali \label{pteclassiques}
% Rappelons quelques propri\'et\'es remarquables de Langlands-Jacquet sur
% $\o\QM_{\ell}$.
% \begin{enumerate}
% \item $\ker\LJ_{\oQl}=\RC_{I}(G,\o\QM_{\ell})$
% \item $\LJ_{\oQl}$ est surjective et admet une section remarquable
%   $\JL_{\oQl} : \RC(D^{\times},\o\QM_{\ell})\To{} \RC(G,\o\QM_{\ell})$
%   qui envoie les irr\'eductibles sur les s\'eries discr\`etes, \`a un signe
%   uniforme pr\`es.
% \item Pour $\pi\in\Irr{\o\QM_{\ell}}{G}$, la repr\'esentation virtuelle
%   $\LJ_{\oQl}(\pi)$ est soit nulle, soit irr\'eductible au signe pr\`es.
% \end{enumerate}
% Ce que nous notons ici $\JL_{\oQl}=\pm |\JL_{\oQl}|$ est donc la \emph{correspondance de
% Jacquet-Langlands}, dont l'existence est d\'emontr\'ee dans
% \cite{DKV}. C'est d'ailleurs de l'existence de $|\JL_{\oQl}|$ jointe \`a la
% classification de Zelevinsky que l'on d\'eduit celle de $\LJ_{\oQl}$,
% \emph{cf} \cite[Sec. 2]{lt} et \cite{Badu2}. 
% Cependant, il est important de remarquer que la section $\JL_{\oQl}$ ci-dessus
% \emph{n'est pas compatible} \`a la r\'eduction mod $\ell$. 

%Voici comment se g\'en\'eralisent les points i) et ii).

Dans ce m\^eme exemple, on voit qu'il n'existe pas de section de $\LJ_{\oFl}$
envoyant les irr\'eductibles sur des irr\'eductibles. En effet, la
seule repr\'esentation $\pi\in\Irr{\oFl}{G}$ telle que
$\LJ_{\oFl}[\pi]=\pm [1_{D^{\times}}]$ est la repr\'esentation
triviale $\pi=1_{G}$. Or on a $\LJ_{\oFl}[1_{G}]=-[1_{D^{\times}}]$.
Le th\'eor\`eme suivant permet n\'eanmoins d'exhiber en toute g\'en\'eralit\'e 
une section de $\LJ_{\oFl}$ envoyant les irr\'eductibles sur des
irr\'eductibles au signe pr\`es.

\begin{theo}\label{ptesLJ}
L'homomorphisme $\LJ_{\o\FM_{\ell}}$ est surjectif, et son noyau est
$\RC_{I}(G,\o\FM_{\ell})$. Il envoie la classe d'une repr\'esentation
superSpeh de $G$ sur celle d'une irr\'eductible de $D^{\times}$
au signe pr\`es, 
induisant une bijection
$$ \{\o\FM_\ell\hbox{-repr\'esentations superSpeh de  }G\}\simto
\{\o\FM_\ell\hbox{-repr\'esentations irr\'eductibles de } D^{\times}\}.$$ 
\end{theo}

Notons que la surjectivit\'e de $\LJ_{\o\FM_{\ell}}$ d\'ecoule de celle de $\LJ_{\oQl}$
et du point ii) de la proposition \ref{propD1}. On a aussi d\'ej\`a
remarqu\'e que $\RC_{I}(G,\o\FM_{\ell})$ est inclus dans le noyau de $\LJ_{\oQl}$.
La proposition suivante est une \'etape importante de la preuve du th\'eor\`eme.

\begin{prop} \label{propptesLJ}
  Soit $\pi$ une $\o\QM_{\ell}$-repr\'esentation superSpeh
  $\ell$-enti\`ere de $G$. Notons $\rho:=|LJ_{\oQl}(\pi)|$ sa
  correspondante dans $\Irr{\o\QM_{\ell}}{D^{\times}}$. Alors $\pi$
  est $\ell$-superSpeh \ssi\ $\rho$ est $\ell$-irr\'eductible.
\end{prop}
\begin{proof}
La preuve est tr\`es clairement inspir\'ee de celle du lemme 2.3.a) de
\cite{VigLanglands}. En effet, si $\pi'$ est congrue \`a
$\pi$, et si $\rho'$ d\'esigne sa correspondante, alors le deuxi\`eme
diagramme du th\'eor\`eme \ref{theoLJ} montre que $\rho'$ est
congrue \`a $\rho$. Utilisant les notations du paragraphe
\ref{terminologie}, on en d\'eduit l'in\'egalit\'e $ m(\pi) \leq m(\rho)$.
Par ailleurs, la correspondance de Jacquet-Langlands \'etant compatible
\`a la torsion par les caract\`eres, on a $t(\pi)=t(\rho)$ et donc
$a(\pi)=a(\rho)$.

Supposons alors que $\pi$ est $\ell$-supercuspidale. D'apr\`es le lemme
\ref{lemrelev}, on a $m(\pi)= \ell^{{\rm
    val}_{\ell}(a(\pi))}$. Gr\^ace \`a la proposition \ref{propD2}, on en d\'eduit
que $m(\rho)=\ell^{{\rm val}_{\ell}(a(\rho))}$, puis que $\rho$ est
$\ell$-irr\'eductible.

Il s'ensuit que $\LJ_{\o\FM_{\ell}}$ envoie les superSpeh sur
des irr\'eductibles de $D^{\times}$.
Par ailleurs on sait que la restriction de $\LJ_{\o\FM_{\ell}}$ \`a
$R_{\Delta}(G,\o\FM_{\ell})$ est surjective, puisque
$\LJ_{\o\FM_{\ell}}$ est  surjective et
nulle sur $\RC_{I}(G,\Delta)$.
On en d\'eduit que l'application 
\begin{equation}
 \{\o\FM_\ell\hbox{-repr. superSpeh de  }G\}\To{|\LJ_{\o\FM_{\ell}}|}
\{\o\FM_\ell\hbox{-repr. irr\'eductibles de } D^{\times}\}\label{app}
\end{equation}
est surjective.

Revenons maintenant \`a la repr\'esentation $\pi$ de l'\'enonc\'e et supposons
que $\rho$ est $\ell$-irr\'eductible. De ce qui
pr\'ec\`ede, on d\'eduit que $\pi$ est congrue \`a une repr\'esentation
$\ell$-superSpeh. La caract\'erisation du lemme \ref{lemrelev}
montre alors que $\pi$ elle-m\^eme est $\ell$-superSpeh.
\end{proof}

\begin{proof}[\indent Fin de la preuve du th\'eor\`eme \ref{ptesLJ}]
Vue la preuve de la proposition pr\'ec\'edente, il reste \`a \'etablir
l'injectivit\'e de (\ref{app}). Cela revient \`a montrer que deux
$\o\QM_{\ell}$-repr\'esentations $\ell$-superSpeh dont les
correspondantes $\rho$ et $\rho'$ sont congrues, sont elle m\^eme
congrues. Mais cela d\'ecoule \`a nouveau des crit\`eres du lemme
\ref{lemrelev} et de la proposition \ref{propD2}.
%Mais celle-ci d\'ecoule du lemme
%\ref{caracdiff} et de la proposition \ref{ellelements}.
\end{proof}

Avant de passer \`a la preuve du th\'eor\`eme \ref{theoJL},
 signalons les deux corollaires suivants du th\'eor\`eme
\ref{ptesLJ},
qui ne concernent que le groupe $G$, mais que
l'auteur ne sait pas prouver directement.

\begin{coro}
  Le noyau de la restriction 
$\tilde\theta_{|G^{\rm
  ell}_{\ell'}}:\, \RC(G,\o\FM_{\ell})\To{}\CC^{\infty}(G^{\rm
ell}_{\ell'},\o\ZM_{\ell})^{G}$ du caract\`ere de Brauer $\tilde\theta$ aux \'el\'ements
elliptiques (d'ordre premier \`a $\ell$) est exactement
$\RC_{I}(G,\o\FM_{\ell})$.
\end{coro}

\begin{coro}
  La question \ref{questBrauer} ii) a une r\'eponse affirmative pour $G=\GL_{d}(K)$.
\end{coro}

\alin{Preuve du th\'eor\`eme \ref{theoJL}} \label{preuvetheoJL}
Notons $^{z}\JL_{\oFl}$
l'inverse de la bijection du th\'eor\`eme \ref{ptesLJ}. D'apr\`es ce
th\'eor\`eme, $^{z}\JL_{\oFl}$ v\'erifie bien les
deux propri\'et\'es annonc\'ees dans l'\'enonc\'e du th\'eor\`eme \ref{theoJL}, 
et est uniquement d\'etermin\'ee par ces propri\'et\'es.

Le reste de l'\'enonc\'e du th\'eor\`eme \ref{theoJL} d\'ecoule du diagramme
suivant.
$$
\xymatrix{
\{\o\QM_\ell\hbox{-repr. irr\'eductibles de } D^{\times}\}
\ar@{<-}[r]_-{Z_{\oQl}\circ\JL_{\oQl}}^-{\sim} &
\{\o\QM_\ell\hbox{-repr.  de Speh de }G\} \\
\{\o\QM_\ell\hbox{-repr. $\ell$-irr\'eductibles de } D^{\times}\}
\ar@{^(->}[u] \ar@{->>}[d]_{r_{\ell}}
\ar@{<-}[r]_-{Z_{\oQl}\circ\JL_{\oQl}}^-{\sim} &
\{\o\QM_\ell\hbox{-repr. $\ell$-superSpeh de  }G\}
\ar@{^(->}[u] \ar@{->>}[d]_{r_{\ell}} \\
\{\o\FM_\ell\hbox{-repr. irr\'eductibles de } D^{\times}\}
\ar@{<-}[r]_-{^{z}\JL_{\oQl}}^-{\sim} &
\{\o\FM_\ell\hbox{-repr. superSpeh de  }G\}
}
$$
Dans ce diagramme, la bijection du haut vient simplement de ce que
 l'involution $Z_{\oQl}$ induit une bijection
entre s\'eries discr\`etes et repr\'esentations de Speh. L'inverse de la
bijection du haut est donn\'ee par $|\LJ_{\oQl}|$.
Le fait que cette
bijection induise la bijection du milieu d\'ecoule alors de la proposition 
\ref{propptesLJ}. La commutativit\'e du carr\'e du bas provient enfin de
la d\'efinition de $^{z}\JL_{\oFl}$ comme inverse de $|\JL_{\oFl}|$,
laquelle est induite par $|\LJ_{\oQl}|$ via la r\'eduction modulo $\ell$ d'apr\`es le
th\'eor\`eme \ref{theoLJ}.
\findem

\alin{Preuve du corollaire \ref{coroLD}}
Nous rappelons d'abord l'\'enonc\'e de la correspondance de Langlands
modulo $\ell$ pour $\GL_{d}(K)$, d\^u \`a Vign\'eras \cite[1.8]{VigLanglands}.
Comme dans le cas
$\ell$-adique ou complexe, la correspondance est \'etablie en deux
temps. D'abord pour les (super)cuspidales, puis pour les autres
repr\'esentations \`a partir de la classification. 
Il s'av\`ere que  pour rendre \`a la correspondance une certaine 
compatibilit\'e \`a la r\'eduction modulo $\ell$, il faut la 
composer avec l'involution de Zelevinsky.
Notons donc  $\tau_{\oQl}^{G}:= \sigma_{\oQl}^{G}\circ Z_{\oQl}$ cette
compos\'ee et appelons-la ``correspondance de Zelevinski sur $\oQl$''.
Rappelons que son image est l'ensemble
${\rm Rep}^{d}_{\o\QM_{\ell}}(WD_{K})$ des classes d'isomorphisme de
 $\oQl$-repr\'esentation de Weil-Deligne
$\tau=(\tau^{\rm ss},N)$ de
dimension $d$. Une telle repr\'esentation est dite enti\`ere si sa partie
semi-simple l'est. On sait alors que $\tau^{G}_{\oQl}$ et
$\sigma^{G}_{\oQl}$ respectent les repr\'esentations enti\`eres.
On dispose d'une application de r\'eduction modulo $\ell$ 
$$R_{\ell}:\, {\rm Rep}^{d}_{\o\QM_{\ell}}(WD_{K})^{\rm ent}\To{}{\rm
  Rep}^{d}_{\o\FM_{\ell}}(WD_{K}) $$ 
qui sur la partie semi-simple est donn\'ee par $r_{\ell}$. Cette
application est jug\'ee \'evidente dans \cite[1.8]{VigLanglands}, mais on
pourra en trouver une construction dans \cite[4.1.8]{lt}.
Parall\`element on a une application
$$ J_{\ell}:\, \Irr{\o\QM_{\ell}}{G_{d}} \To{}
\Irr{\o\FM_{\ell}}{G_{d}}$$
qui envoie une repr\'esentation $\ell$-enti\`ere $\pi$  sur l'unique
constituent $\o\pi$ de $r_{\ell}\pi$ qui v\'erifie
$\lambda_{\pi}=\lambda_{\o\pi}$. D'apr\`es \cite[Thm V.12]{VigInduced},
cette application est surjective.
Le th\'eor\`eme principal de \cite{VigLanglands} dit  alors :
\begin{thm} {Il existe une bijection
$$ \tau_{\o\FM_{\ell}}^{G}:\, \Irr{\o\FM_{\ell}}{\GL_{d}(K)} \simto {\rm Rep}^{d}_{\o\FM_{\ell}}(WD_{K})   $$
uniquement d\'etermin\'ee par la propri\'et\'e $\tau_{\o\FM_{\ell}}^{G} \circ
J_{\ell} = R_{\ell} \circ \tau_{\oQl}^{G}$.  }
\end{thm}
La correspondance de Langlands-Vign\'eras est alors d\'efinie par
$\sigma^{G}_{\oFl}:= \tau^{G}_{\oFl}\circ Z_{\oFl}$.

Dans l'\'enonc\'e du corollaire \ref{coroLD}, nous avons d\'efini
$\sigma^{D}_{\oFl}:=\sigma^{G}_{\oFl}\circ
\JL_{\oFl}=\tau^{G}_{\oFl}\circ {^{z}\JL}_{\oFl}$. Il s'agit bien
d'une injection $\Irr{\oFl}{D^{\times}}\injo 
{\rm Rep}_{\o\FM_{\ell}}^{d}(WD_{K})$. Remarquons que sa d\'efinition ne
fait intervenir que la restriction de $\tau^{G}_{\oFl}$ aux
repr\'esentations superSpeh. Pour d\'eterminer l'image de
$\sigma^{D}_{\oFl}$, nous aurons besoin du lemme suivant.
%introduisons la notion suivante.
% Une $\oQl$-repr\'esentation de Weil-Deligne
% $\sigma$ sera dite \emph{$\ell$-ind\'ecomposable} si elle est enti\`ere et
% si son image par 
% $R_{\ell}$ est ind\'ecomposable.

\begin{lem}
  Soit $\sigma=(\sigma^{\rm ss},N)$ une $\oQl$-repr\'esentation de
  Weil-Deligne enti\`ere. On a \'equivalence entre
  \begin{enumerate}
  \item $R_{\ell}(\sigma)$ est ind\'ecomposable ($\sigma$ est alors dite
    $\ell$-ind\'ecomposable).
  \item $\sigma$ est ind\'ecomposable et ${\rm
      long}(r_{\ell}(\sigma^{\rm ss}))= {\rm long}(\sigma^{\rm 
  ss})$.
\item Il existe une repr\'esentation $\ell$-irr\'eductible 
$\lambda$ de $W_{K}$ telle que
$\sigma^{\rm ss}=\bigoplus_{i=0}^{r-1} \lambda(-i)$ et $N$ est donn\'e
par la compos\'ee de la projection \'evidente $\sigma^{\rm ss} \twoheadrightarrow 
\bigoplus_{i=1}^{r-1} \lambda(-i)$ suivie de l'injection \'evidente
$\bigoplus_{i=1}^{r-1} \lambda(-i) \injo \sigma^{\rm ss}(-1)$.
\item $(\tau^{G}_{\oQl})^{-1}(\sigma)$ est une repr\'esentation
  $\ell$-superSpeh de $G$.
  \end{enumerate}
\end{lem}
\begin{proof}
  L'\'equivalence entre les trois premiers points est laiss\'ee au
  lecteur. Pour y incorporer le point iv), rappelons que
  $\tau_{\oQl}^{G}$ induit une bijection
$$ \tau^{G}_{\oQl}:\;\{\o\QM_\ell\hbox{-repr.  de Speh de }G\} 
\To{\sim} 
\{\o\QM_\ell\hbox{-repr.  de dim. $d$ ind\'ecomposables de }WD_{K}\}. $$
 Compte tenu du lemme \ref{lemrelev} et de la compatibilit\'e de
 $\tau^{G}_{\oQl}$ \`a la torsion et aux congruences, il nous suffira de
 prouver le crit\`ere num\'erique de $\ell$-ind\'ecomposabilit\'e d'une
 repr\'esentation ind\'ecomposable $\sigma$ suivant.

\emph{Soit $t(\sigma)$ le nombre de caract\`eres non ramifi\'es $\psi$ de $W_{K}$
tels que $\sigma\psi\simeq \sigma$. Soit $m(\sigma)$ le nombre de
repr\'esentations de Weil-Deligne enti\`eres $\sigma'$  telles que
$R_{\ell}(\sigma)=R_{\ell}(\sigma')$ et
${\rm det}(\sigma'(\varphi))={\rm det}(\sigma(\varphi))$ o\`u $\varphi$ est un rel\`evement de
Frobenius fix\'e. Alors $m(\sigma)$ est inf\'erieur \`a la plus grande
puissance de
$\ell$ divisant le nombre
$a(\sigma):=\frac{d}{t(\sigma)}(q^{t(\sigma)}-1)$, et lui est \'egal 
\ssi\ $\sigma$ est
$\ell$-ind\'ecomposable.}

En effet, \'ecrivons $\sigma^{\rm ss}$ sous la forme
$\bigoplus_{i=0}^{r-1} \lambda(-i)$ avec $\lambda$ irr\'eductible et $N$
comme dans le point iii). On voit que $\sigma$ est
$\ell$-ind\'ecomposable \ssi\ $\lambda$ est $\ell$-irr\'eductible. D'apr\`es
\cite[2.3]{VigAENS}, cela \'equivaut \`a $m(\lambda)=\ell^{{\rm val}_{\ell}(a(\lambda))}$,
avec des notations similaires \`a ci-dessus. Or, par un argument
similaire \`a celui du lemme \ref{lemrelev}, on v\'erifie
que $m(\sigma)= \ell^{{\rm val}_\ell(r)}m(\lambda)$ et
$t(\sigma)=t(\lambda)$, d'o\`u l'on conclut \`a nouveau comme dans le
lemme \ref{lemrelev}.
\end{proof}

Le lemme justifie le carr\'e commutatif sup\'erieur du
diagrammes suivant. 
$$
\xymatrix{
\{\o\QM_\ell\hbox{-repr.  de Speh de }G\} 
\ar[r]_-{\tau^{G}_{\oQl}}^-{\sim} &
\{\o\QM_\ell\hbox{-repr.  de dim. $d$ ind\'ecomposables de }WD_{K}\} \\
\{\o\QM_\ell\hbox{-repr. $\ell$-superSpeh de  }G\}
\ar@{^(->}[u] \ar@{->>}[d]_{r_{\ell}}
\ar[r]_-{\tau^{G}_{\oQl}}^-{\sim} &
\{\o\QM_\ell\hbox{-repr.  de dim. $d$, $\ell$-ind\'ecomposables de }WD_{K}\} 
\ar@{^(->}[u] \ar@{->>}[d]_{R_{\ell}} \\
\{\o\FM_\ell\hbox{-repr. superSpeh de  }G\}
\ar[r]_-{\tau^{G}_{\oFl}}^-{\sim} &
\{\o\FM_\ell\hbox{-repr.  de dim. $d$ ind\'ecomposables de }WD_{K}\} 
}
$$
Le carr\'e inf\'erieur est la traduction du th\'eor\`eme
rappel\'e ci-dessus, puisque pour toute $\oQl$-repr\'esentation de Speh
enti\`ere $\delta$, on a simplement $J_{\ell}(\delta)=r_{\ell}(\delta)$.
Finalement on a identifi\'e l'image de $\sigma^{D}_{\oFl}$ \`a ${\rm
  Rep}_{\oFl}^{d}(WD_{K})^{\rm indec}$. La caract\'erisation en termes
de rel\`evements d\'ecoule de la caract\'erisation analogue dans le th\'eor\`eme \ref{theoJL}.

\findem

\subsection{Langlands-Jacquet mod $\ell$ et effectivit\'e au signe pr\`es}

\label{seceff}

Dans cette section, nous \'etudions la question \ref{questionLJ}. 
Remarquons qu'elle rel\`eve purement de la th\'eorie des
repr\'esentations de $G$. En effet, soit $\pi\in \Irr{\oFl}{G}$,
il s'agit de savoir si la projection de la classe
$[\pi]$ sur $\RC_{\Delta}(G,\o\FM_{\ell})$ modulo
$\RC_{I}(G,\o\FM_{\ell})$ est une combinaison lin\'eaire positive ou
n\'egative de repr\'esentations superSpeh.
En cons\'equence, il y a ``beaucoup'' de cas o\`u la question admet
presque trivialement une r\'eponse positive.
Avant d'exposer ces cas, nous avons besoin d'un raffinement de la
proposition \ref{propdecomp}.

\alin{Support supercuspidal}
Soit $\pi\in\Irr{\oFl}{G}$. On peut trouver un \levi $M=G_{d_{1}}\times
\cdots\times G_{d_{r}}$ de $G$ et une repr\'esentation
\emph{supercuspidale}
$\tau=\tau_{1}\otimes\tau_{2}\otimes\cdots\otimes\tau_{r}$ de $M$
telle que  $\pi$ apparaisse comme sous-quotient de l'induite 
$\tau_{1}\times\tau_{2}\times\cdots\times\tau_{r}$. 
D'apr\`es \cite[V.4]{VigInduced}, cette paire $(M,\tau)$ est
\emph{unique \`a conjugaison pr\`es}. La classe de conjugaison de
$(M,\tau)$ est appel\'ee \emph{support supercuspidal} de $\pi$.

D\'efinissons alors $\RC(G,\oFl)_{\tau}$ comme le
sous-groupe de $\RC(G,\oFl)$ engendr\'e par les irr\'eductibles dont le
support supercuspidal contient $(M,\tau)$, et posons 
$\RC_{\Delta}(G,\oFl)_{\tau}:=\RC(G,\oFl)_{\tau}\cap\RC_{\Delta}(G,\oFl)$
et de m\^eme
$\RC_{I}(G,\oFl)_{\tau}:=\RC(G,\oFl)_{\tau}\cap\RC_{I}(G,\oFl)$.

\begin{lem}\label{decompssc} 
%Soit  $\tau=\tau_{1}\otimes\cdots\otimes\tau_{r}$ comme ci-dessus. 
On a  une d\'ecomposition
$ \RC(G,\oFl)_{\tau} = \RC_{\Delta}(G,\oFl)_{\tau} \oplus
\RC_{I}(G,\oFl)_{\tau},$ o\`u le rang de
$\RC_{\Delta}(G,\oFl)_{\tau}$ est le nombre de repr\'esentations
supercuspidales $\tau_{0}$ de $G_{d/r}$ telles que la paire $(M,\tau)$
soit 
conjugu\'ee \`a la paire $(G_{d/r}^{r},
\tau_{0}\otimes\tau_{0}\nu\otimes\cdots\otimes\tau_{0}\nu^{r-1})$.
Plus pr\'ecis\'ement, ce nombre est donn\'e comme suit.
\begin{enumerate}
\item Si $\tau$ est conjugu\'ee \`a
  $\tau_{0}\otimes\tau_{0}\nu\otimes\cdots\otimes\tau_{0}\nu^{r-1}$
  pour $\tau_{0}$ supercuspidale de $G_{d/r}$, alors posons $r_{0}$ le
  cardinal de l'orbite $\{\tau_{0}\nu^{i}, i\in\ZM\}$.
  \begin{enumerate}
  \item Si $r_0$ divise $r$, alors
    $\rk(\RC_{\Delta}(G,\oFl)_{\tau})=r_0$.
  \item Si $r_0$ ne divise pas $r$, alors     $\rk(\RC_{\Delta}(G,\oFl)_{\tau})=1$.
  \end{enumerate}
\item Sinon,     $\rk(\RC_{\Delta}(G,\oFl)_{\tau})=0$.
\end{enumerate}
\end{lem}
\begin{proof}
La d\'ecomposition se prouve exactement comme celle de la proposition 
\ref{propdecomp}, une fois qu'on a remarqu\'e la cons\'equence suivante de
la  propri\'et\'e d'``unicit\'e'' du support supercuspidal : si $\pi$
appara\^{\i}t comme sous-quotient d'une repr\'esentation $i$ paraboliquement
induite d'une repr\'esentation irr\'eductible, alors tous les
sous-quotients de $i$ ont le m\^eme support supercuspidal que $\pi$.

Le rang de $\RC_{\Delta}(G,\oFl)_{\tau}$ est le nombre de
repr\'esentations superSpeh de support supercuspidal contenant $(M,\tau)$.
Vue la condition  m1)  dans la caract\'erisation des repr\'esentations
superSpeh au paragraphe 
\ref{defsuperSpeh}, ce nombre est bien celui annonc\'e dans le lemme. La
discussion d\'etaill\'ee est alors \'el\'ementaire.
% De m\^eme sous l'hypoth\`ese i)(b), on voit que
% $\tau_{0}$ est uniquement d\'etermin\'ee, et qu'il y a donc une seule
% superSpeh de support supercuspidal contenant $(M,\tau)$.
Pr\'ecisons seulement que sous l'hypoth\`ese i)(a),  les superSpeh dans
$\RC(G,\oFl)_{\tau}$ sont exactement
\begin{equation}
\delta_{r}(\tau_{0}),
\delta_{r}(\tau_{0}\nu)=\delta_{r}(\tau_{0})\nu ,\cdots,
\delta_{r}(\tau_{0}\nu^{r_{0}-1}) =
\delta_{r}(\tau_{0})\nu^{r_{0}-1}.\label{enusuperspeh} 
\end{equation}
\end{proof}

Par commodit\'e, nous dirons qu'une repr\'esentation $\pi\in\Irr{\oFl}{G}$
est \emph{elliptique de type $(\tau_{0},r)$}
  si $\tau_{0}$ est une repr\'esentation supercuspidale de $G_{d/r}$ et le
 support supercuspidal de $\pi$ contient la paire  
 $(
 (G_{d/r})^{r},\tau_{0}\otimes\tau_{0}\nu\otimes\cdots\otimes\tau_{0}\nu^{r-1})$.

\begin{coro} \label{coroeff}
  Soit $\pi\in\Irr{\oFl}{G}$. 
  \begin{enumerate}
  \item Si $\pi$ n'est pas elliptique, on a $\LJ_{\oFl}(\pi)=0$.
  \item Si $\pi$ est elliptique de type $(\tau_{0},r)$, alors dans
   les  cas ci-dessous, 
  $\LJ_{\oFl}(\pi)$ est effective au signe pr\`es.
  \begin{enumerate}
  \item Si $\pi$ se rel\`eve \`a $\oQl$ (exemple : $\pi$ g\'en\'erique ou de Speh).
  \item Si $\tau_{0}$ n'est pas isomorphe \`a $\tau_{0}\nu^{r}$.
  \item Si $\pi\nu\simeq \pi$.
\end{enumerate}
  \end{enumerate}
\end{coro}
\begin{proof}
Le cas i) d\'ecoule du point ii) du lemme. Dans le cas ii)(a), choisissons
un rel\`evement $\wt\pi$ de $\pi$. On a alors
$\LJ_{\oFl}(\pi)=r_{\ell}(\LJ_{\oQl}(\wt\pi))$. Or on sait que
$\LJ_{\oQl}(\wt\pi)$ est effective au signe pr\`es.

Le cas ii) (b) est justiciable du cas i)(b) du lemme. Enfin le dernier
cas rel\`eve du cas i)(a) du lemme. On a vu dans ce cas en
(\ref{enusuperspeh}) que les superSpeh de m\^eme support supercuspidal
que $\pi$ forment une orbite $(\delta,\delta\nu,\cdots
\delta\nu^{r_{0}-1})$ sous l'action de $\nu$ par torsion. \'Ecrivons
alors
$$ [\pi]= \sum_{i=0}^{r_{0}-1} a_{i}[\delta\nu^{i}] \hbox{ mod }
\RC_{I}(G,\oFl).$$
Comme $\RC_{I}(G,\oFl)$ est stable par torsion par $\nu$, l'hypoth\`ese
$[\pi]=[\pi\nu]$ implique que tous les $a_{i}$ sont \'egaux. Ils ont
donc \emph{a fortiori} le m\^eme signe.
\end{proof}

On peut maintenant \'ecr\'emer encore un peu en utilisant le support
cuspidal.

\begin{prop}
  Soit $\pi\in \Irr{\oFl}{G}$ une repr\'esentation elliptique. Si le
  support cuspidal de $\pi$ est diff\'erent de son support
  supercuspidal, alors $\LJ_{\oFl}(\pi)$ est effective au signe pr\`es.
\end{prop}
\begin{proof}
  Soit $\lambda=\lambda_{1}\otimes\cdots\otimes\lambda_{s}$ un \'el\'ement
  du support cuspidal de $\pi$. D'apr\`es \cite[Thm V.10]{VigInduced}, il existe un
  multisegment cuspidal $b$ de support $\{\lambda_{1},\cdots, \lambda_{s}\}$
 tel que $\pi\simeq \langle b\rangle$. 

Supposons dans un premier temps que l'ensemble $\{\lambda_{1},\cdots,
\lambda_{s}\}$ n'est pas connexe, au sens de \cite[V.3]{VigInduced}. \'Ecrivons
$b$ comme somme $b=b_{1}\sqcup\cdots \sqcup b_{k}$ de composantes
connexes. Par construction, la repr\'esentation $\langle b\rangle$
appara\^{\i}t avec multiplicit\'e $1$ dans l'induite $\langle
b_{1}\rangle\times \langle b_{2}\rangle \times\cdots\times \langle
b_{k}\rangle$. Or, d'apr\`es \cite[Prop V.3]{VigInduced}, cette induite est irr\'eductible. La
repr\'esentation $\pi$ est donc induite parabolique propre et on a
$\LJ_{\oFl}(\pi)=0$.

Supposons maintenant que l'ensemble $\{\lambda_{1},\cdots,
\lambda_{s}\}$ est connexe. Cela signifie en particulier que tous les $\lambda_{i}$ sont
de la forme $\lambda\nu^{j_{i}}$ pour un certain entier $j_{i}$. Par
hypoth\`ese, $\lambda$ n'est pas supercuspidale. D'apr\`es
\cite[III.5.14]{Vig}, elle est donc de la forme ``Steinberg
g\'en\'eralis\'ee''. Par construction, une telle repr\'esentation v\'erifie
$\lambda\simeq \lambda\nu$. Par cons\'equent $\pi\simeq\pi\nu$ et on
peut appliquer le point ii)(c) du corollaire pr\'ec\'edent.
\end{proof}

En utilisant les alg\`ebres de Hecke de types de Bushnell-Kutzko comme
dans \cite[IV]{VigInduced}, la proposition pr\'ec\'edente permet en principe\footnote{Les
  d\'etails restent toutefois \`a \'ecrire.} de ramener la
question d'effectivit\'e de $\LJ_{\oFl}(\pi)$ au cas o\`u $\pi$ est
elliptique \emph{superunipotente}, c'est-\`a-dire au cas o\`u le support
cuspidal de $\pi$ contient la paire $(\GM^{d},
1\otimes\nu\otimes\cdots\otimes\nu^{d-1})$, \`a torsion par un caract\`ere
$\chi\otimes\chi\otimes\cdots\otimes\chi$ pr\`es.
Notons alors $\varepsilon$ l'ordre de $q$ dans $\Fl^{\times}$. Le
corollaire \ref{coroeff} dit que $\LJ_{\oFl}(\pi)$ est effectif si
$\varepsilon=1$ ou si $\varepsilon$ ne divise pas $d$. Le premier cas ``non
trivial'' d''effectivit\'e est donc donn\'e par la proposition suivante.

\begin{prop} \label{effnontrivial}
  Soit $\pi$ elliptique superunipotente. Si l'ordre $\varepsilon$ de
  $q$ dans $\Fl^{\times}$ est $d$, alors
  $\LJ_{\oFl}(\pi) $ est effective au signe pr\`es.
\end{prop}
\begin{proof}
Nous allons simplement expliciter la classification dans ce cas
particulier, c'est-\`a-dire la matrice de passage entre ``modules
standard'' et ``modules simples''. 
En prenant $\chi=\nu^{\frac{1-d}{2}}$ ci-dessus, on voit
qu'il s'agit d'expliciter la classification des sous-quotients de
l'induite $\nu^{\frac{1-d}{2}}\times\cdots\times\nu^{\frac{d-1}{2}}=
{\rm Ind}_{B}^{G}(1)$. On peut trouver une classification  dans
\cite[2.17]{VigLuminy} dans un langage
diff\'erent, mais sans le calcul de la
matrice de passage.

Puisque l'ordre multiplicatif de $q$ est $d$, 
un multisegment de support $\nu^{\frac{1-d}{2}}\otimes\cdots\otimes\nu^{\frac{d-1}{2}}$ s'\'ecrit
 $a=\{ (i_{k},r_{k}), k=1,\cdots,|a|\}$ o\`u $|a|$ est le nombre de
 segments de $a$, les $i_{k}$ sont des \'el\'ements de $\ZM/d\ZM$, les
 $r_{k}$ sont des entiers positifs de somme $d$ et on a 
$\ZM/d\ZM=\bigsqcup_{k=1}^{|a|} [i_{k},i_{k}+r_{k}-1] \hbox{ mod }
d$. 
% Les $r_{k}$ d\'eterminent une partition $\lambda_{a}^{t}$ de $d$ qui,
% comme sa notation le sugg\`ere, est la transpos\'ee de la partition not\'ee
% $\lambda_{a}$ dans le paragraphe \ref{classification}. 
Pour chaque $k= 1,\cdots,|a|$, notons $P(a,k)$ le \para standard
$P_{r_{k},r_{k+1},\cdots, r_{|a|},r_{1},\cdots, r_{k-1} }$. Par
d\'efinition de $\pi(a)$, on a dans $\RC(G,\oFl)$ l'\'egalit\'e
$$ [\pi(a)] = \left[ {\rm Ind}_{P(a,k)}^{G}(\nu^{i_{k}})\right].$$
Notre but est de calculer les multiplicit\'es d\'efinies par les formules
$[\pi(a)]=\sum_{b}m(b,a)\langle b\rangle$.
La premi\`ere remarque importante est que ces multiplit\'es sont $1$ ou
$0$, car le caract\`ere
$\nu^{\frac{1-d}{2}}\otimes\cdots\otimes\nu^{\frac{d-1}{2}}$ est
r\'egulier pour l'action du groupe de Weyl, et il y a une seule
repr\'esentation non-superunipotente, \`a savoir la non-d\'eg\'en\'er\'ee (ou
``Steinberg g\'en\'eralis\'ee''), qui est cuspidale.
Remarquons maintenant que les segments $a$ comme ci-dessus sont en bijection 
$a\mapsto I_{a}:=\{i_{1},\cdots, i_{|a|}\}$
avec les sous-ensembles \emph{non vides} de $\ZM/d\ZM$.
Par cette bijection, les repr\'esentations superSpeh correspondent aux
singletons et la non-d\'eg\'en\'er\'ee \`a l'ensemble plein.
Voici alors la formule de multiplicit\'e
\begin{equation}
  \label{mult}
  m(b,a)=1 \Leftrightarrow I_{b} \subseteq I_{a}.
\end{equation}
Admettons un instant cette formule. La matrice inverse d\'efinie par
$\langle a\rangle = \sum_{b} n(b,a) [\pi(b)]$ est alors facile \`a
calculer. On v\'erifie en effet que 
$$\cas{n(b,a)}{(-1)^{I_{a}\setminus I_{b}}}{I_{b}\subseteq I_{a}}{0}{sinon}{}$$
En particulier, on obtient la formule
$$ \langle a \rangle = (-1)^{|a|-1}\left(\sum_{k=1}^{|a|} [\nu^{i_{k}}]\right)
\hbox{ mod } \RC_{I}(G,\oFl).$$

Il reste \`a prouver la formule (\ref{mult}). \`A un sous-ensemble $I$
de $\ZM/d\ZM$ muni d'un \'el\'ement $i\in I$ on associe un \para standard
$P(I,i)$ de la mani\`ere suivante. Consid\'erons la bijection $\pi_{i}:\, x\in \{1, \cdots ,
d-1\} \mapsto x+i\in \ZM/d\ZM\setminus\{i\}$. Le sous-ensemble
$J:=\pi_{i}^{-1}(I\setminus\{i\})$ de $\{1,\cdots, d-1\}$ 
d\'etermine une matrice de Jordan $N_{J}=\sum_{j\in J}E_{j,j+1}$, et on
d\'efinit $P(I,i)$ comme le plus grand \para standard dont le radical
unipotent contient
$N_{J}$. On a alors $P(a,k)=P(I_{a},i_{k})$ et $[\pi(a)]=\left[{\rm
  Ind}_{P(I_{a},i_{k})}^{G}(\nu^{i_{k}})\right]$.

Supposons alors que
$I_{b}\subseteq I_{a}$ et soit $i$ un \'el\'ement de $I_{b}$. On a une
inclusion  ${\rm  Ind}_{P(I_{b},i)}^{G}(\nu^{i})\subseteq {\rm
  Ind}_{P(I_{a},i)}^{G}(\nu^{i})$ d'o\`u en particulier $m(b,a)=1$.

Supposons au contraire que $I_{b}$ n'est pas inclus dans $I_{a}$. Deux
cas se pr\'esentent. Si l'intersection $I_{b}\cap I_{a}$ contient un
\'el\'ement $i$, les deux repr\'esentations ${\rm
  Ind}_{P(I_{b},i)}^{G}(\nu^{i})$ et  ${\rm
  Ind}_{P(I_{a},i)}^{G}(\nu^{i})$ sont contenues dans 
${\rm  Ind}_{B}^{G}(\nu^{i})$ qui est de multiplicit\'e $1$. Leur
intersection est ${\rm  Ind}_{P(I_{b}\cap I_{a},i)}^{G}(\nu^{i})$. Par
hypoth\`ese, $I_{b}\cap I_{a}$ est strictement inclus dans $I_{b}$, donc
pour tout sous-quotient irr\'eductible $\pi$ 
de ${\rm  Ind}_{P(I_{b}\cap I_{a},i)}^{G}(\nu^{i})$, on a
$\lambda_{\pi}<\lambda_{b}$. En particulier, $\langle b\rangle$ n'est
pas sous-quotient de ${\rm  Ind}_{P(I_{b}\cap I_{a},i)}^{G}(\nu^{i})$,
donc ne l'est pas plus de ${\rm  Ind}_{P(I_{a},i)}^{G}(\nu^{i})$, et
on a bien $m(b,a)=0$.
Il nous reste \`a consid\'erer le cas $I_{a}\cap I_{b}=\emptyset$. Fixons
$i\in I_{b}$. Si  $I_{b}$ contient un autre \'el\'ement $j$, on peut
appliquer le cas pr\'ec\'edent au multisegment $a_{j}$ correspondant \`a
l'ensemble $I_{a}\cup \{j\}$. En effet, le cas pr\'ec\'edent dit que
$m(b,a_{j})=0$ donc \emph{a fortiori} $m(b,a)=0$ puisque
$\pi(a)\subset \pi(a_{j})$. Dans le cas contraire,  $I_{b}=\{i\}$ est
un singleton avec $i\notin I_{a}$. Il s'agit de voir que le caract\`ere
$\nu^{i}$ n'intervient pas dans $r_{P_{1,\cdots, 1}}(\pi(a))$, ce qui r\'esulte d'un calcul
de foncteur de Jacquet.
\end{proof}

\begin{coro}
  Si l'ordre $\varepsilon$ de $q$ dans $\Fl^{\times}$ n'est pas un
  diviseur \emph{propre} de $d$, alors $\LJ_{\oFl}$ envoie toute
  repr\'esentation irr\'eductible sur une effective au signe pr\`es.
\end{coro}
\begin{proof}
  Le cas $\varepsilon=1$ et le cas o\`u $\varepsilon$ ne divise pas $d$
  d\'ecoulent respectivement des points ii)(b) et ii)(c) du corollaire
  \ref{coroeff}.
Supposons donc $\varepsilon=d$. Fixons un diviseur $r$ de $d$ et une
repr\'esentation supercuspidale $\tau_{0}$ de $G_{d/r}$. Notons $r_{0}$ le
cardinal de l'orbite de $\tau_{0}$ par torsion par les puissances de
$\nu$. C'est un diviseur de $d$. Supposons que ce soit un diviseur
\emph{propre} de $r$. Alors, d'apr\`es \cite[III.5.14 i)]{Vig},  la repr\'esentation de Steinberg g\'en\'eralis\'ee
${\rm St}(\tau_{0},r_{0})$ serait une repr\'esentation cuspidale non
supercuspidale  de $G_{dr_{0}/r}$. Or ceci est impossible puisque
l'hypoth\`ese $\varepsilon=d$ implique que le pro-ordre de
$G_{dr_{0}/r}$ est inversible dans $\Fl$, donc que toute cuspidale est
supercuspidale (et m\^eme projective modulo le centre).
On a donc deux cas. Soit $r_{0}$ ne divise pas $r$, auquel cas on peut
appliquer le lemme i)(b) \ref{decompssc}. Soit $r_{0}=r$, auquel cas
nous laisserons le lecteur se convaincre que l'argument de la
proposition pr\'ec\'edente s'adapte sans autre difficult\'e que celles
inh\'erentes aux notations.
\end{proof}

%  L'exemple \ref{exempleLJ} montre qu'une irr\'eductible n'est pas n\'ecessairement
% envoy\'ee sur une irr\'eductible au signe pr\`es. Plus pr\'ecis\'ement, avec les
% notations de l'exemple on a
% $\LJ_{\oQl}([\pi])=-[1]-[\varepsilon]$.
% Voil\`a ce que nous esp\'erons n\'eanmoins.

% \begin{conj}
%   L'image par $\LJ_{\o\FM_{\ell}}$ de la classe d'une repr\'esentation irr\'eductible
%   est, au signe pr\`es, la classe d'une repr\'esentation effective.
% \end{conj}

\alin{Une strat\'egie possible dans le cas $\ell>n$}
Voici quelques \'etapes d'une strat\'egie possible. Chaque \'etape semble
non triviale, et comme de toutes fa\c{c}ons la derni\`ere \'etape \'echappe \`a
l'auteur, il n'y a pas grand sens \`a les d\'evelopper ici.

\begin{enumerate}
\item Pour chaque bloc de $\Mo{\oFl}{G}$, ramener la question \`a une
  question de $q$-alg\`ebres de Schur. Pour le bloc unipotent, une grande
  partie du travail est faite dans \cite{VigSchur}.
\item Prouver que pour $\ell> n$, la matrice de d\'ecomposition de
  la $\oQl$-alg\`ebre de Schur  en une racine $\varepsilon$-\`eme de
  l'unit\'e vers la $\oFl$-alg\`ebre de Schur en $q$ est l'identit\'e
  (analogue affine de la ``conjecture de James'').
\item Utiliser l'interpr\'etation g\'eom\'etrique par Ginzburg-Vasserot des multiplicit\'es des
  modules simples dans les modules standard d'une alg\`ebre de Schur en
  une racine de l'unit\'e.
\item ``Inverser'' la matrice de polyn\^omes de Kazhdan-Lusztig \`a laquelle
  on s'est ainsi ramen\'e. 
\end{enumerate}

Voici \`a quoi ressemble l'interpr\'etation g\'eom\'etrique des multiplicit\'es.
Soit $\zeta$ une racine $\varepsilon$-\`eme de l'unit\'e et  $s$ la
matrice diagonale ${\rm Diag}(1,\zeta,\cdots, \zeta^{d-1})\in 
M_{d}(\CM)$. Consid\'erons la vari\'et\'e alg\'ebrique affine complexe
$$\NC_{s}:=\{N\in M_{d}(\CM),\,\hbox{
  nilpotente  et t.q. } sNs^{-1}=\zeta N\} .$$
Le centralisateur $C(s)$ est connexe et agit par conjugaison sur $\NC_{s}$ avec un
nombre fini d'orbites, toutes simplement connexes. Consid\'erons le groupe de Grothendieck
$\KC:=\KC_{C(s)}(\NC_{s})$ des faisceaux constructibles
$C(s)$-\'equivariants sur $\NC_{s}$. Il admet deux bases remarquables
index\'ees par les orbites. Soit $j^{\OC}$ l'inclusion d'une orbite
$\OC$ dans $\NC_{s}$. La premi\`ere base naturelle est donn\'ees par les faisceaux
$[\OC]:=j^{\OC}_{!}(\CM)$ et la seconde par les complexes de faiceaux
$[IC(\OC)]:=j^{\OC}_{!*}(\CM)$. Un fait remarquable est que les
$j^{\OC}_{!*}(\CM)$ n'ont de cohomologie qu'en degr\'es pairs, de sorte
que la matrice de passage $(\langle [\OC'],[IC(\OC)]\rangle)_{\OC',\OC}$
des $[\OC]$ vers les $[IC(\OC)]$ est \`a
coefficients positifs. Notons que cette matrice est triangulaire si
l'on ordonne (partiellement) les orbites selon l'inclusion des adh\'erences.

Les travaux de Ginzburg, Vasserot et Ariki \cite{GV} \cite{Vasserot}
\cite{Ariki} exhibent une bijection $a\mapsto
\OC_{a}$ entre les multisegments de support
$(1,\zeta,\cdots,\zeta^{d-1})$ et les orbites de $C(s)$ dans $\NC_{s}$.
telle que pour tous multisegments $a,b$ on a 
$$ m(b,a)= \langle  [\OC_{a}],[IC(\OC_{b})] \rangle .$$
Dans cette bijection, les segments $\{(\zeta^{i}, d)\}$
correspondent aux orbites maximales, dont les adh\'erences sont les
composantes irr\'eductibles de $\NC_{s}$.

\begin{exe}
  Dans le cas $\varepsilon=d$, la vari\'et\'e $\NC_{s}$ est celle des
  matrices de la forme $\sum_{i=1}^{d} x_{i}E_{i,i+1}$ o\`u au moins un
  $x_{i}$ est nul, et o\`u l'on convient que $E_{d,d+1}=E_{d,1}$. Ainsi,
  $\NC_{s}$ est la r\'eunion des hyperplans de coordonn\'ees dans
  $\CM^{d}$. Les orbites sont param\'etr\'ees par les sous-ensemble de
  $\{1,\cdots, d\}$. \`A $I$ correspond $\OC_{I}=\{ (x_{1},\cdots, x_{d})\in\CM^{d},
  i\in I \Leftrightarrow  x_{i}=0\}$. 
 Dans ce cas, le calcul des
  complexes d'intersection est facile : on a simplement
  $j^{\OC}_{!*}(\CM)=j^{\o\OC}_{*}(\CM)$ o\`u $j^{\o\OC}$ est
  l'inclusion de l'adh\'erence de l'orbite $\OC$. Comme $\OC_{J}\subset
  \o\OC_{I}\Leftrightarrow  J\supset I$, on retrouve ainsi la matrice
  de passage calcul\'ee dans la proposition \ref{effnontrivial}.
\end{exe}

Dans le cas o\`u $\varepsilon$ divise strictement $d$, la g\'eom\'etrie des
orbites est beaucoup plus compliqu\'ee. Par exemple les inclusions
$\OC\injo \o{\OC}$ ne sont g\'en\'eralement pas affines, et les $j_{!}^{\OC}(\CM)$ ne
sont pas n\'ecessairement pervers d\'ecal\'es, ce qui emp\^eche de pr\'evoir
simplement les signes dans la matrice inverse par une formule du type
${\rm sgn}(\langle[IC(\OC')],[\OC]\rangle)= (-1)^{\rm codim(\OC',\OC)}$.

Cependant, d'apr\`es Lusztig \cite[Sec. 11]{Lusztig90}, les $\langle [\OC'],[IC(\OC)]\rangle$ sont les
valeurs en $1$ de certains polyn\^omes de Kazhdan-Lusztig. Plus
pr\'ecis\'ement, soit $\wt\SG_{d}$ le groupe de Weyl affine \'etendu de type
$\wt{A_{d-1}}$. Notons $_{\varepsilon}[\wt\SG_{d}]_{\varepsilon}$
l'ensemble des repr\'esentants de longueur maximale des doubles classes
modulo le sous-groupe parabolique (fini)
$(\SG_{d/\varepsilon})^{\varepsilon}$. Munissons-le de l'ordre de
Bruhat. Alors il existe un isomorphisme $\OC\mapsto w_{\OC}$ du poset des orbites de $C(s)$
dans $\NC_{s}$ sur un id\'eal du poset
$_{\varepsilon}[\wt\SG_{d}]_{\varepsilon}$, tel que
$$ \langle [\OC'],[IC(\OC)]\rangle = P_{w_{\OC},w_{\OC'}}(1),$$
o\`u $P_{w,w'}$ d\'esigne le polyn\^ome de Kazhdan-Lusztig usuel associ\'e \`a
deux \'el\'ements de $\wt\SG_{d}$, \emph{cf.}  \cite{Henderson}. Pour des petites valeurs de $d$, on
peut donc esp\'erer calculer par ordinateur les coefficients qui nous
int\'eressent dans la matrice inverse.

%\section{Lubin-Tate non-ab\'elien et modulo $\ell$} \label{modl}

%\bibliography{artbiblio}

\begin{thebibliography}{10}

\bibitem{Ariki}
S.~Ariki.
\newblock On the decomposition numbers of the {H}ecke algebra of {$G(m,1,n)$}.
\newblock {\em J. Math. Kyoto Univ.}, 36(4):789--808, 1996.

\bibitem{Badu1}
I.~Badulescu.
\newblock Orthogonalit\'e des caract\`eres pour ${GL}(n)$ sur un corps local de
  caract\'eristique non nulle.
\newblock {\em Manuscripta Math.}, 101:49--70, 2000.

\bibitem{Badu2}
I.~Badulescu.
\newblock Jacquet-{L}anglands et unitarisabilit\'e.
\newblock {\em J. Inst. Math. Jussieu}, 6(3):349--379, 2007.

\bibitem{BorelCorv}
A.~Borel.
\newblock Automorphic {$L$}-functions.
\newblock In {\em Automorphic forms, representations and {$L$}-functions
  ({P}roc. {S}ympos. {P}ure {M}ath., {O}regon {S}tate {U}niv., {C}orvallis,
  {O}re., 1977), {P}art 2}, Proc. Sympos. Pure Math., XXXIII, pages 27--61.
  Amer. Math. Soc., Providence, R.I., 1979.

\bibitem{Broussous}
P.~Broussous.
\newblock Extension du formalisme de {B}ushnell-{K}utzko au cas d'une algèbre à
  division.
\newblock {\em Proc. London Math. Soc.}, 77(3):292--326, 1998.

\bibitem{BK}
C.J. Bushnell and P.C. Kutzko.
\newblock {\em The admissible dual of ${GL}(n)$ via compact open subgroups}.
\newblock Number 129 in Annals Math. Studies. P.U.P., 1993.

\bibitem{CR}
C.~Curtis and I.~Reiner.
\newblock {\em Methods of Representation theory {I}}.
\newblock Wiley Interscience, 1988.

\bibitem{lt}
J.-F. Dat.
\newblock Th\'eorie de {L}ubin-{T}ate non-ab\'elienne et repr\'esentations
  elliptiques.
\newblock {\em Invent. Math.}, 169:75--152, 2007.

\bibitem{lefschetz}
J.-F. Dat.
\newblock Op\'erateur de {L}efschetz sur les tours de {D}rinfeld et
  {L}ubin-{T}ate, 2009.
\newblock {\rm Preprint},
  http://www.math.jussieu.fr/~dat/recherche/travaux.html.

\bibitem{DKV}
P.~Deligne, D.~Kazhdan, and M.-F. Vign{\'e}ras.
\newblock Repr\'esentations des alg\`ebres centrales simples {$p$}-adiques.
\newblock In {\em Representations of reductive groups over a local field},
  Travaux en Cours, pages 33--117. Hermann, Paris, 1984.

\bibitem{GV}
V.~Ginzburg and E.~Vasserot.
\newblock Langlands reciprocity for affine quantum groups of type {$A_n$}.
\newblock {\em Internat. Math. Res. Notices}, (3):67--85, 1993.

\bibitem{Henderson}
A.~Henderson.
\newblock Nilpotent orbits of linear and cyclic quivers and {K}azhdan-{L}usztig
  polynomials of type {A}.
\newblock {\em Represent. Theory}, 11:95--121 (electronic), 2007.

\bibitem{KazhdanCusp}
D.~Kazhdan.
\newblock Cuspidal geometry of {$p$}-adic groups.
\newblock {\em J. Analyse Math.}, 47:1--36, 1986.

\bibitem{Korman}
J.~Korman.
\newblock A character formula for compact elements (the rank one case).
\newblock {\em preprint arXiv:math.RT/0409292v1}, 2004.

\bibitem{Lusztig90}
G.~Lusztig.
\newblock Canonical bases arising from quantized enveloping algebras.
\newblock {\em J. Amer. Math. Soc.}, 3(2):447--498, 1990.

\bibitem{MS2}
R.~Meyer and M.~Solleveld.
\newblock Characters and growth of admissible representations of reductive
  $p$-adic groups.
\newblock {\em preprint arXiv:math.RT/0908...}, 2009.

\bibitem{MS1}
R.~Meyer and M.~Solleveld.
\newblock Resolutions for reductive groups over $p$-adic fields via their
  buildings.
\newblock {\em preprint arXiv:math.RT/0902.4856v2}, 2009.

\bibitem{MinSech}
A.~Minguez and V.~S\'echerre.
\newblock Repr\'esentations $\ell$-modulaires des formes int\'erieures de ${\rm
  gl}_n$ sur un corps $p$-adique, $\ell \neq p$.
\newblock {\em En pr\'eparation}, 2009.

\bibitem{SS2}
P.~Schneider and U.~Stuhler.
\newblock Representation theory and sheaves on the {B}ruhat-{T}its building.
\newblock {\em Publ. Math. I.H.\'E.S}, 85:97--191, 1995.

\bibitem{SchZ}
P.~Schneider and E.-W. Zink.
\newblock ${K}$-types for the tempered components of a $p$-adic general linear
  group.
\newblock {\em J. reine angew. Math.}, 517:161--208, 1999.

\bibitem{Spice}
L.~Spice.
\newblock Topological {J}ordan decompositions.
\newblock {\em J. Algebra}, 319(8):3141--3163, 2008.

\bibitem{VanDijk}
G.~van Dijk.
\newblock Computation of certain induced characters of {${p}$}-adic groups.
\newblock {\em Math. Ann.}, 199:229--240, 1972.

\bibitem{Vasserot}
E.~Vasserot.
\newblock Affine quantum groups and equivariant {$K$}-theory.
\newblock {\em Transform. Groups}, 3(3):269--299, 1998.

\bibitem{VigLuminy}
M.-F. Vign{\'e}ras.
\newblock \`{A}\ propos d'une conjecture de {L}anglands modulaire.
\newblock In {\em Finite reductive groups ({L}uminy, 1994)}, volume 141 of {\em
  Progr. Math.}, pages 415--452. Birkh\"auser Boston, Boston, MA, 1997.

\bibitem{VigSheaves}
M.-F. Vign\'eras.
\newblock Cohomology of sheaves on the building and ${R}$-representations.
\newblock {\em Invent. Math.}, 127:349--373, 1997.

\bibitem{VigInduced}
M.-F. Vign{\'e}ras.
\newblock Induced {$R$}-representations of {$p$}-adic reductive groups.
\newblock {\em Selecta Math. (N.S.)}, 4(4):549--623, 1998.

\bibitem{VigLanglands}
M.-F. Vign\'eras.
\newblock Correspondance de {L}anglands semi-simple pour ${GL(n,F)}$ modulo
  $\ell\neq p$.
\newblock {\em Invent. Math.}, 144:177--223, 2001.

\bibitem{VigAENS}
M.-F. Vign{\'e}ras.
\newblock La conjecture de {L}anglands locale pour {${\rm GL}(n,F)$} modulo
  {$l$} quand {$l\not= p,\ l>n$}.
\newblock {\em Ann. Sci. \'Ecole Norm. Sup. (4)}, 34(6):789--816, 2001.

\bibitem{VigSchur}
M.-F. Vign{\'e}ras.
\newblock Schur algebras of reductive {$p$}-adic groups. {I}.
\newblock {\em Duke Math. J.}, 116(1):35--75, 2003.

\bibitem{VigWalds}
M.-F. Vign{\'e}ras and J.-L. Waldspurger.
\newblock Premiers r\'eguliers de l'analyse harmonique mod {$l$} d'un groupe
  r\'eductif {$p$}-adique.
\newblock {\em J. Reine Angew. Math.}, 535:165--205, 2001.

\bibitem{Vig}
M.F. Vign\'eras.
\newblock {\em Repr\'esentations $l$-modulaires d'un groupe $p$-adique avec $l$
  diff\'erent de $p$}.
\newblock Number 137 in Progress in Math. Birkh\"auser, 1996.

\bibitem{Zel}
A.V. Zelevinsky.
\newblock Induced representations on reductive p-adic groups {II}.
\newblock {\em Ann.Sci.Ec.Norm.Sup}, 13:165--210, 1980.

\end{thebibliography}

\appendix

\section{Caract\`ere d'une repr\'esentation modulo $\ell$
  d'un groupe $p$-adique}

\def\jtem{\item  \hskip 18pt}
\def\doublemap{\mathrel{\null _{\raise.2ex\hbox{$\textstyle\rightarrow$}}
  ^{\ellower.2ex\hbox{$\textstyle\rightarrow$}}}}

\def\date{le\space\the\day \ifcase\month\or janvier \or f\'evrier\or mars\or
avril\or mai\or juin\or juillet\or ao\^ut\or septembre\or octobre\or
novembre\or d\'ecembre\fi\ {\oldstyle\the\year}}

\def\Z{{\bf Z}}
\def\N{{\bf N}}
\def\C{{\bf C}}
\def\Q{{\bf Q}}
\def\R{{\bf R}}
\def\G_m{{\bf G}_m}
\def\Rbar{\overline R}
\def\Zbar{{\overline{\bf Z}}_{\l}}
\def\Qbar{{\overline{\bf Q}}_{\l}}
\def\Fbar{{\overline{\bf F}}_{\l}}
\def\GL{{\bf GL}}
\def\PGL{{\bf PGL}}
\def\R{{\bf R}}
\def\e{{\varepsilon}}
\def\X{\chi}
\def\I{{\cal I}}
\def\F{{\bf F}}
\let\s\sigma
\let\g\gamma
\let\G\Gamma
\let\a\alpha
\let\d\delta
\let\w\omega
\let\W\Omega
\let\L\Lambda
\let\t\tau
\def\Ad{\mathop{\rm Ad}\nolimits}
\def\cl{\mathop{\rm cl}\nolimits}
\def\ad{\mathop{\rm ad}\nolimits}
\def\diag{\mathop{\rm diag}\nolimits}
\def\Co{\mathop{\rm Coeff_G(X)}\nolimits}
\def\pp{\mathop{\rm pour \ presque \ tout}\nolimits}
\def\trace{\mathop{\rm trace}\nolimits}
\def\Wh{\mathop{\rm Wh}\nolimits}
\def\ppcm{\mathop{\rm p.p.c.m.}\nolimits}
\def\modulo{\mathop{\rm modulo}\nolimits}
\def\Cusp{\mathop{\rm Cusp}\nolimits}
\def\cusp{\mathop{\rm cusp}\nolimits}
\def\Supercusp{\mathop{\rm Supercusp}\nolimits}
\def\sign{\mathop{\rm sign}\nolimits}
\def\Fr{\mathop{\rm Fr}\nolimits}
\def\Gr{\mathop{\rm Gr}\nolimits}
\def\gcd{\mathop{\rm gcd}\nolimits}
\def\lcm{\mathop{\rm lcm}\nolimits}
\def\Br{\mathop{\rm Br}\nolimits}
\def\Entr{\mathop{\rm Supp}\nolimits}
\def\Supp{\mathop{\rm Entr}\nolimits}
\def\Im{\mathop{\rm Im}\nolimits}
\def\Gal{\mathop{\rm Gal}\nolimits}
\def\SL{\mathop{\rm SL}\nolimits}
\def\val{\mathop{\rm val}\nolimits}
\def\Tot{\mathop{\rm Tot}\nolimits}
\def\nrd{\mathop{\rm nrd}\nolimits}
\def\Hom{\mathop{\rm Hom}\nolimits}
\def\limind{\mathop{\rm lim \,ind }\nolimits}
\def\limproj{\mathop{\rm lim \,proj }\nolimits}
\def\Ad{\mathop{\rm Ad}\nolimits}
\def\id{\mathop{\rm id}\nolimits}
\def\ind{\mathop{\rm ind}\nolimits}
\def\lcm{\mathop{\rm l.c.m.}\nolimits}
\def\End{\mathop{\rm End}\nolimits}
\def\Ind{\mathop{\rm Ind}\nolimits}
\def\Ker{\mathop{\rm Ker}\nolimits}
\def\Coker{\mathop{\rm Coker}\nolimits}
\def\Aut{\mathop{\rm Aut}\nolimits}
\def\LieG{\mathop{\cal G}\nolimits}
\def\pourtout{\mathop{\rm\ pour \, tout\  }\nolimits}
\def\St{\mathop{\rm St}\nolimits}
\def\Alg{\mathop{\rm Alg}\nolimits}
\def\tr{\mathop{\rm tr}\nolimits}
\def\limproj{\mathop{\oalign{lim\cr$\longleftarrow$\cr}}}
\def\vol{\mathop{\rm vol}\nolimits}
\def\tg{\mathop{\rm tg}\nolimits}
\def\indc{\mathop{\rm ind_c}\nolimits}
\def\ch{\mathop{\rm ch}\nolimits}
\def\Ch{\mathop{\rm Ch}\nolimits}
\def\sh{\mathop{\rm sh}\nolimits}
\def\cite{{\bf [V]}}
\def\meas{\mathop{\rm meas}\nolimits}
\def\Lie{\mathop{\rm Lie}\nolimits}
\def\Ad{\mathop{\rm Ad} \nolimits}
\def\exp{\mathop{\rm exp} \nolimits}
\def\Ens{\mathop{\rm Ens} \nolimits}
\def\log{\mathop{\rm log} \nolimits}
\def\cInd{\mathop{\rm cInd} \nolimits}
\def\Res{\mathop{\rm Res} \nolimits}
\def\res{\mathop{\rm res} \nolimits}
\def\Ext{\mathop{\rm Ext} \nolimits}
\def\Ext{\mathop{\rm Ext} \nolimits}
\def\Ind{\mathop{\rm Ind} \nolimits}
\def\Irr{\mathop{\rm Irr} \nolimits}
\def\mod{\mathop{\rm mod} \nolimits}
\def\Mod{\mathop{\rm Mod} \nolimits}
\def\id{\mathop{\rm id} \nolimits}
\def\Adm{\mathop{\rm Adm} \nolimits}
\def\Coadm{\mathop{\rm Coadm} \nolimits}
\def\Int{\mathop{\rm Int} \nolimits}
\def\lim{\mathop{\rm lim} \nolimits}
\def\Tot{\mathop{\rm Tot} \nolimits}
\def\Card{\mathop{\rm Card} \nolimits}
\def\inf{\mathop{\rm inf} \nolimits}
\def\sup{\mathop{\rm sup} \nolimits}

\hfuzz=3pt
\overfullrule=0mm

%\input LO

%\centerline {{\bf Appendice : caract\`ere d'une repr\'esentation modulo $\ell$  d'un groupe p-adique}}

\bigskip

\centerline {Texte \'ecrit en mars 1998, r\'evis\'e en janvier 2010}

\bigskip  \centerline {Marie-France Vign\'eras}

\bigskip  
Soient  $\ell \neq p$ deux nombres premiers distincts. 
On note  $\Qbar$ une cl\^oture alg\'ebrique du corps ${\bf Q}_{\ell}$ des nombres $\ell$-adiques, $\Zbar$ et $\Fbar$ l'anneau des entiers et le corps r\'esiduel. Soient $F$ un corps local non archim\'edien de caract\'eristique $0$, de corps 
r\'esiduel fini   de caract\'eristique $p$ et 
 $R$  un corps alg\'ebriquement clos contenant $\Fbar$.
 Soit $G$ le groupe des points $F$-rationnels d'un $F$-groupe r\'eductif connexe. 

On note $C_c^\infty(X;R)$ le $R$-module des fonctions localement constantes \`a support 
compact sur $X$, pour toute partie ferm\'ee $X$ de $G$.
On note $\Mod_RG$ l'ensemble des $R$-repr\'esentations lisses de $G$, et
$\Irr_RG$ l'ensemble des classes d'isomorphisme  des 
$R$-repr\'esentations  irr\'eductibles lisses  de $G$.

 On choisit, comme on le peut, une $R$-mesure de Haar $dg$ sur $G$.
 Soient $f \in C_c^\infty (G;R)$ et $(\pi,V) \in \Mod_RG$ admissible, i.e $\dim_{R}V^{K} $ fini pour tout sous-groupe ouvert $K$ de $G$. 
L'endomorphisme  $\pi(f dg)$ de $V$ est de rang fini. Sa trace
$\tr (\pi (f dg))$  d\'efinit une forme lin\'eaire sur $C_c^\infty(G;R)$, d\'ependant
du choix 
 de la mesure de Haar $dg$, appel\'ee la trace de $\pi$. 
Les $R$-repr\'esentations irr\'eductibles de $G$ sont admissibles ([Vig], II.2.8) et leurs
traces   sont des formes lin\'eaires  lin\'eairement ind\'ependantes 
 sur $C_c^\infty(G;R)$ ([Vig] I.6.13).

L'ensemble $G^{reg}$  des \'el\'ements semisimples r\'eguliers
  est un ouvert dense de $G$  invariant par conjugaison. 

\bigskip {\bf Th\'eor\`eme  1 \ } CARACTERE \ {\sl Soit $\pi$ une $R$-repr\'esentation de 
$G$ de longueur finie. 
Il existe une fonction localement constante $\X_{\pi}: G^{reg} \to R$ telle que 
$$\tr (\pi (f  dg))=\int_G \X_{\pi}(g) \ f(g) \ dg$$
pour tout $f \in C_c^\infty(G^{reg};R)$. 
On dit que $\X_{\pi}$ est le   caract\`ere  de $\pi$. }

\bigskip  Le caract\`ere de $\pi$ ne d\'epend pas du choix de la mesure de 
Haar $dg$, il est invariant par conjugaison par $G$, et il est unique.

\bigskip  
Pour les $\Fbar$-repr\'esentations de longueur finie d'un groupe fini, 
R. Brauer a introduit
un autre caract\`ere qui prend ses valeurs dans $\Zbar$, et qui permet
de d\'ecrire le groupe de Grothendieck 
de ces $\Fbar$-repr\'esentations. Nous allons
donner une construction analogue pour le groupe $p$-adique $G$,   utilisant 
la  m\'ethode de Howe pour construire le caract\`ere, avec l'espoir d'applications 
au groupe de Grothendieck des $\Fbar$-repr\'esentations de longueur finie de $G$.

Une $\Zbar$-structure d'une repr\'esentation lisse de type fini de $G$ sur un $\Qbar$-espace vectoriel
$V$ est un $\Zbar$-module libre $G$-stable, contenant une $\Qbar$-base de $V$, 
et qui est  de type fini comme $\Zbar G$-module.

Nous  supposons pour simplifier que le centre de $G$ est compact (le cas g\'en\'eral ne pose aucune difficult\'e).
 Notons $G_{\l'}^{ell}$ l'ensemble des \'el\'ements 
 elliptiques $\l$-r\'eguliers de $G$ (voir III.2 et 4), 
$\Lambda$ l'id\'eal maximal de $\Zbar$ et $r_{\Lambda}:\Zbar \to \Fbar$
la r\'eduction modulo $\Lambda$.

\bigskip {\bf Th\'eor\`eme 2 \ } CARACTERE DE BRAUER \ {\sl 
Pour toute repr\'esentation $\pi \in \Mod_{\Fbar}G$ de longueur finie, 
il existe une fonction 
$$\phi_{\pi}: G_{\l'}^{ell} \to \Zbar$$ 
ayant les propri\'et\'es suivantes :

1) $\phi_{\pi}$ est localement constante, invariante par $G$-conjugaison.

2) $\phi_{\pi}$ rel\`eve le caract\`ere $\X_{\pi}$ de $\pi$ sur $G_{\l'}^{ell}$,
$$r_{\Lambda} \circ \phi_{\pi} =\X_{\pi}.$$

3)  $\phi_{\pi}$ ne d\'epend que de la semi-simplification de $\pi$.

4) Si $\pi$ est la r\'eduction modulo $\Lambda$ d'une $\Zbar$-structure d'une 
$\Qbar$-repr\'esentation $\Pi$ de $G$, alors $\phi_{\pi}$ est \'egal \`a la restriction du
 caract\`ere $\X_{\Pi}$
de $\Pi$ sur $G_{\l'}^{ell}$. 

On dit que $\phi_{\pi}$ est un caract\`ere de Brauer de $\pi$.}

\bigskip On  d\'efinit  un caract\`ere de Brauer $\phi_{\pi}$ comme ceci.
Soit $g\in G_{\l'}^{ell}$ et soit
$(K_n)_{n\geq 0}$ 
  une suite d\'ecroissante pro-p-sous-groupes ouverts compacts normalis\'es  par $g$  d'intersection $\{1\}$; notons $\mu_n$ la $R$-mesure de Haar
sur $G$ telle que le volume de $K_n$ est $1$, et $1_{gK_n}$ la fonction caract\'eristique de $gK_n$; posons $e_{gK_n}:=1_{gK_n}\mu_n.$ 
L'endomorphisme  $\pi(e_{gK_n})$ de l'espace de dimension finie $V(\pi)^{K_n}$, est
 d'ordre fini premier \`a $\l$. 
Pour tout $n $ assez grand,
 le caract\`ere de Brauer de $\pi(e_{gK_n})$ est \'egal \`a 
$\phi_{\pi}(g)$. 

Le caract\`ere de Brauer est unique, si l'on sait que chaque $\Fbar$-repr\'esentation
 irr\'eductible de $G$ apparait dans la r\'eduction modulo $\Lambda$ d'une $\Qbar$-repr\'esentation
 irr\'eductible de $G$ ayant une $\Zbar$-structure.   
  
\bigskip  Les th\'eor\`emes 1 et  2 sont d\'emontr\'es dans la partie   II en admettant l'existence des certaines filtrations s\'epar\'ees d\'ecroissantes de $G$ en pro-$p$-sous-groupes ouverts compacts (propositions I.1 et II.1).  Ces propositions sont d\'emontr\'ees (III.4,5) en utilisant  la th\'eorie de Kirillov-Howe et un lemme de Harish-Chandra (III.3) qui supposent que la caract\'eristique de $F$ est $0$. C'est le seul endroit o\`u cette hypoth\`ese apparait.

\bigskip \bigskip 

\centerline{ {\bf I \ Caract\`ere d'une representation } \ }

\bigskip 
On connait deux d\'emonstrations  diff\'erentes du th\'eor\`eme 1, 
lorsque $R$ est le corps des nombres complexes, l'une due 
\`a R.Howe, l'autre \`a Harish-Chandra. 
Les deux d\'emonstrations s'\'etendent sans probl\`emes lorsque $\C$ est remplac\'e par $R$.

Nous rappelons la d\'emonstration de Howe,  valable sur un corps commutatif contenant
 des racines de l'unit\'e d'ordre une puissance quelconque de $p$, que
l'on  adaptera dans la partie II pour d\'efinir le  caract\`ere de Brauer. 
 
Une suite d\'ecroissante de  sous-groupes  de $G$ d'intersection triviale est appel\'ee filtration s\'epar\'ee d\'ecroissante  de sous-groupes de $G$.

Une $R$-repr\'esentation $\s$ d'un pro-$p$-sous-groupe $K$ de $G$ est appel\'ee entrelac\'ee  avec la repr\'esentation triviale d'un  pro-$p$-sous-groupe $K'$ de $G$ s'il existe $g\in G$ tel que la restriction de $\s$ \`a $K\cap gK'g^{-1}$ contient la repr\'esentation triviale.  On dit que $g\in G$ entrelace $\s$ avec elle-m\^eme
   si  $\s$ et la repr\'esentation  $g(\s)$ de $gKg^{-1}$ restreintes \`a $K\cap gKg^{-1}$ ont un composant irr\'eductible isomorphe.

\bigskip {\bf I.1 \ Proposition \ } {\sl Il existe une filtration s\'epar\'ee d\'ecroissante   de pro-p-sous-groupes ouverts $(K_n)_{n \geq 0}$ de $G$ ayant la propri\'et\'e  de finitude suivante :

  Pour tout entier $n_o\geq 0$ et tout
$\gamma \in G^{reg}$, il existe un entier $n_1 \geq n_{0}$ et un voisinage $V(\gamma)$ de $\gamma$ dans $G^{reg}$
tels que pour tout entier $n\geq n_1$, l'ensemble des classes d'isomorphisme des $R$-repr\'esentations
irr\'eductibles $\s$ de $K_n$,
 entrelac\'ees avec la repr\'esentation triviale de $K_{n_o}$, et  avec elles-m\^eme par un \'el\'ement
 de $ V(\gamma)$, est fini.}
 
La d\'emonstration est donn\'ee en III.4.

 \bigskip 
{\bf I.2 \ } {\sl D\'emonstration du th\'eor\`eme 1 \ }  La proposition I.1 implique
le th\'eor\`eme 1,  par l'argument  simple suivant  de Howe. 

Soient  $\pi$ une $R$-repr\'esentation lisse de longueur finie d'espace $V(\pi)$ de $G$ et $\gamma \in G^{reg}$. On
choisit, comme on le peut,  un entier  $n_o\geq 0$ tel que  $\pi$ est
engendr\'ee par ses vecteurs $K_{n_o}$-invariants.
On applique la proposition I.1, et l'on obtient un entier $n_1$ et un voisinage  $V(\gamma)$.
On
choisit un entier
$n \geq n_1$. Toute repr\'esentation $\s\in \Irr_R K_n$ contenue dans $\pi$
est entrelac\'ee avec  la repr\'esentation triviale de $K_{n_o}$ ([Vig] I.8.1). 
Par la proposition I.1, il n'existe qu'un ensemble fini $S$ de $\s\in \Irr_R K_n$ contenues dans $\pi$
tels que $g$ entrelace $\s$ avec elle-m\^eme.  La restriction de $\pi$ au groupe $K_n$ est semi-simple. 
On d\'ecompose l'espace $V(\pi)$ de $\pi$ comme une somme directe de ses parties
$\s$-isotypiques $V(\pi)_{\s}$ pour $\s\in \Irr_R K_n$, et l'on note $p_{\s}$ la
projection de  $V(\pi)$ sur $V(\pi)_{\s}$,
$$V(\pi)=\oplus_\s V(\pi)_{\s}, \ \ \ p_{\s}:V(\pi) \to V(\pi)_{\s}.$$
Soit $g \in V(\gamma)$. L'intersection  $V(\pi)_\s\cap \pi(g)V(\pi)_{\s}$
est stable par $K_n\cap gK_ng^{-1}$. Si elle est non vide, $g$
entrelace $\s$ ([Vig] 1.8.1-2). Si elle est vide, la trace de
l'endomorphisme
$p_{\s}
\pi(g)\in \End V(\pi)$ est nulle.  La fonction sur $ V(\gamma)$ d\'efinie par
$$g \mapsto \X(g) := \sum_{\s\in S} \tr(p_{\s} \pi(g))  $$
est localement constante. 
Si $f \in C_c^\infty (G;R)$ on a 
$$\tr(\pi(f\,dg))=\sum_{\s} \int_G f(g)\tr(p_{\s} \pi(g)) dg.$$
Si le support de $f$ est contenu dans $V(\gamma)$, 
 on a 
$$\tr(\pi(f\,dg))= \int_G f(g) \X(g) dg.$$
Cette formule montre que $\X$ est canonique, quoique sa construction ne le fut pas. 
Par cette  m\'ethode, on   d\'efinit une fonction localement constante 
$$\X_{\pi} : G^{reg} \to R$$
 qui v\'erifie le th\'eor\`eme 1.

\bigskip\bigskip
\centerline{{\bf II \ Caract\`ere modulaire d'une repr\'esentation }}

\bigskip

On dit qu'un \'el\'ement de $G$ est  elliptique  s'il est 
semi-simple r\'egulier, de $G$-centralisateur compact modulo le centre.
On note $G^{ell}$ l'ouvert des \'el\'ements 
 elliptiques de $G$. 

On suppose pour simplifier que le centre est compact et l'on note $G^{ell}_{\l'}$
l'ensemble ouvert des \'el\'ements  elliptiques  $\l$-r\'eguliers
(III.2, 4).

\bigskip {\bf II.1 \ Proposition \ } {\sl Pour tout $g \in G^{ell}_{\l'}$, il existe
 une filtration s\'epar\'ee d\'ecroissante   de pro-p-sous-groupes ouverts compacts $(K_{g,n})_{n \geq 0}$ 
de $G$ normalis\'es par $g$,  ayant les propri\'et\'es
suivantes :

\medskip a) Pour tout 
$\gamma \in G^{ell}_{\l'}$, il existe  un voisinage $V'(\gamma)$ de $\gamma$ dans $G^{ell}_{\l'}$ tel que
 $K_{g,n}=K_{\gamma,n}$ pour tout entier $n \geq 0 $ et tout $g \in V'(\gamma)$.

\medskip b)  Pour tout $x\in G$ et  pour tout entier $n \geq 0$, on a $xK_{g,n}x^{-1}=
K_{xgx^{-1},n}$.

\medskip c) Pour tout entier $n_o\geq 0$ et tout
$\gamma \in G^{ell}_{\l'}$, il existe un entier $n_1 \geq n_{0}$ et un voisinage $V''(\gamma)$ de $\gamma$ dans 
$G^{ell}_{\l'}$
tels que pour tout entier $n\geq n_1$, l'ensemble des classes d'isomorphisme des repr\'esentations
irr\'eductibles $\s$ de $K_{\gamma,n}$
 entrelac\'ees avec la repr\'esentation triviale de $K_{\gamma,n_o}$ et avec elle-m\^emes par un \'el\'ement de  $ V''(\gamma)$  est fini.

}

\bigskip {\bf II.2 \  \ } Nous admettons  la proposition 
qui sera d\'emontr\'ee  en III.5, et nous d\'efinissons le caract\`ere de Brauer.

\bigskip  Soit  $\pi$ une $R$-repr\'esentation de $G$ de longueur finie (donc admissible) d'espace $V(\pi)$ et $\gamma \in
G^{ell}_{\l'}$. 
 Pour simplifier,  pour tout entier $m\geq 0$, on notera $K_m:=K_{\gamma,m}$ et 
$V(\pi)^{(m)}$ l'espace des vecteurs $K_m$-invariants de $V(\pi)$.
On
choisit, comme on le peut,  un entier  $n_o\geq 0$ tel que  $\pi$ est
engendr\'ee par $V(\pi)^{(n_o)}$. 

\bigskip On applique la proposition. On obtient un entier $n_1$ et des voisinages
$ V'(\gamma) $ et $ V''(\gamma)$ de $\gamma$ et l'on pose $V(\gamma):= 
 V'(\gamma)  \cap  V''(\gamma)$. Par c), le nombre de classes d'isomorphisme
des  repr\'esentations
irr\'eductibles de $K_{n_1}$ 
 entrelac\'ees avec la repr\'esentation triviale de $K_{n_o}$ et avec elles-m\^eme par un
\'el\'ement
 de $V(\gamma)$, est fini.   On choisit  un entier
$m_1\geq n_1$ tel que ces  repr\'esentations
 soient triviales sur $K_{m_1}$.

 Soit $g \in V(\gamma)$  et soit $n$ un entier tel que $n \geq m_1 \geq n_{1}$. L'ensemble  
 des repr\'esentations de
$\Irr_RK_{n_1}$,  non triviales sur $K_{m_1}$ mais triviales sur $K_n$ est fini. 
Soit $E$ l'ensemble fini des repr\'esentations $\sigma \in  \Irr_RK_{n_1}$  contenues dans $\pi$ non triviales sur $K_{m_1}$ mais triviales sur $K_n$. Comme $g$ normalise les
$K_m$, il agit sur $E $.  Les representations de $E$ sont entrelac\'ees avec la representation
triviale de $K_{n_o}$ car $V(\pi)^{(n_{0})} $ engendre $\pi$  ([Vig], I.8.1). Par le choix de $m_1$, une repr\'esentation de $E $
n'est pas entrelac\'ee par un \'el\'ement de $V(\gamma)$, donc $g$ agit
sans points fixes dans $E$, i.e  toute orbite $C$ de $g$ dans $E$ a au moins $2$ \'el\'ements.
On note 
$$V(\pi)_C := \oplus_{\s\in C} V(\pi)_{\s}$$
On d\'ecompose l'espace vectoriel de dimension finie $V(\pi)^{(n)}$  comme une somme directe  stable par
$\pi(g)$,
$$V(\pi)^{(n)} =  V(\pi)^{(m_1)}\oplus _C \ V(\pi)_C$$
  $C$ parcourant les orbites de $g$ dans $E$. 

 Lorsque $g\in G$ est compact et $\ell$-r\'egulier l'endomorphisme de  $\pi(g)$ du $R$-espace vectoriel $V(\pi)^{(n)}$ de dimension finie est d'ordre fini premier \`a $\ell$ (voir III.2). On peut d\'efinir son caract\`ere ordinaire (la trace) et son caract\`ere de Brauer (voir III.1). 
Le caract\`ere (ordinaire ou de Brauer)  de  $\pi(g)$ sur 
$V(\pi)^{(n)}$ est la somme des caract\`eres 
(ordinaire ou de Brauer)  de  $\pi(g)$ sur 
$V(\pi)^{(m_1)}$ et sur les $V(\pi)_C$. 
Il est clair que le caract\`ere ordinaire de $\pi(g)$ sur 
$V(\pi)_C$ est nul pour chaque orbite $C$  de $g$ dans $E$.
Nous montrons que  le caract\`ere de Brauer de $\pi(g)$ sur 
$V(\pi)_C$ est nul.  

 On note  $t\geq 2$ le nombre d'\'el\'ements de $C$, et  
 $A$ l'action de $\pi(g)$ sur $ V(\pi)_C$. Soit $\s\in
C$. On note $W_j:=A^{j-1}V(\pi)_{\s}$ pour
pour tout $1\leq j \leq t$. On a $ V(\pi)_C=W_1  \oplus \ldots \oplus W_{t} $.
 L'isomorphisme  $A^t$ sur $W_1$ est diagonalisable, car $A^t$
est d'ordre fini premier \`a $\l$ comme $A$. 
On choisit une base $(e_i)_{1\leq i \leq s}$ de $W_1$ form\'ee de vecteurs propres pour
 $A^t$. L'espace $W'_i$ de dimension $t$ engendr\'e
par $(A^je_i)_{1\leq j \leq t}$ est stable par $A$, 
et  la d\'ecomposition
$$ V(\pi)_C=\oplus_{1\leq i \leq s} W'_i$$
est stable par $A$. 
Le caract\`ere de Brauer
de $A$ sur chaque $W'_i$ est nul (voir la partie III), donc le caract\`ere
de Brauer de $A$ sur $ V(\pi)_C$ est nul. \
$\diamond$

\bigskip Rappelons que $ K_n=K_{\gamma,n}= K_{g,n}$ pour tout $n\geq 0$ et  $g \in V(\gamma)$. 
Nous avons   montr\'e :

\bigskip {\bf   Proposition 3 \ } 
{\sl Pour  tout $\gamma \in
G^{ell}_{\l'}, g \in V(\gamma)$, 
le 
 caract\`ere $\X_n(g)$ et le caract\`ere 
de Brauer $\phi_n(g)$   de l'endomorphisme $\pi(g)$ sur $V(\pi)^{ K_{n}}$ et  
 ne d\'ependent pas de l'entier $n$ si $n \geq m_{1}$. }

\bigskip  On d\'efinit les fonctions sur $G^{ell}_{\l'}$
$$g \mapsto \X(g)= \lim_{n \to \infty} \ \X_n(g), \ \ 
g\mapsto \phi(g)=\lim_{n \to \infty} \ \phi_n(g).$$
 Par le th\'eor\`eme 1, $\X$ est le caract\`ere de
$\pi$. 

\bigskip
\bigskip {\bf II.3 \ }  Nous allons  montrer que $\phi$ est un caract\`ere de Brauer de $\pi$.

\medskip  1) {\sl $\phi$ est localement constante. } \ 
Soit $\gamma \in
G^{ell}_{\l'}$. Par la proposition de II.2,  $\phi(g) $ est le caract\`ere de Brauer  de l'endomorphisme $\pi(g)$ sur $V(\pi)^{ K_{m_{1}}}$ pour tout $g\in V(\gamma) $. Donc la fonction $\phi$ est constante sur le voisinage $V(\gamma)\cap\gamma
K_{\gamma,m_1}$ de $\gamma$.  

\medskip  {\sl $\phi$ 
est invariante par $G$-conjugaison. } \ Ceci provient de la propri\'et\'e  (II.1.b) $K_{xgx^{-1},n}=xK_{g,n}x^{-1}$ pour tout $g\in G^{ell}_{\l'}, \  x\in G$, en remarquant que
le caract\`ere de Brauer de $\pi(xgx^{-1})$ sur $V(\pi)^{xKx^{-1}}$
est le m\^eme que celui de $\pi(g)$ sur $V(\pi)^K$ pour tout sous-groupe 
ouvert compact $K$ de $G$. 

\medskip2) {\sl La r\'eduction de $\phi$ modulo $\Lambda$ est \'egale \`a
$\X$ \ } car pour tout $g\in G^{ell}_{\l'}$, si $n$ est assez grand, 
la r\'eduction de $\phi_n(g)$ modulo $\Lambda$ est \'egale \`a
$\X_n(g)$. 

\medskip 3) {\sl $\phi$
ne d\'epend que de la semi-simplifi\'ee
de $\pi$ \ } car le caract\`ere de Brauer de l'endomorphisme $\pi(g)$ sur  $V(\pi)^{ K_{g,n}}$ ne d\'epend que de 
 la semi-simplifi\'ee
de $\pi$ pour tout entier $n\geq 0$.

\medskip 4) {\sl Si $\pi$ est la r\'eduction d'un $\Zbar$-structure $L$
d'une  $\Qbar$-repr\'esentation irr\'eductible $\Pi$, on a $\X_{\Pi}=\phi_{\pi}$. \ }
En effet, pour tout $g \in G_{\l'}^{ell}$, et tout entier $n$ assez grand, 
la r\'eduction modulo $\Lambda$ de $L^{K_{g,n}}$ est $V(\pi)^{K_{g,n}}$ et
$L^{K_{g,n}}\otimes_{\Zbar} \Qbar \simeq V(\Pi)^{K_{g,n}}$. 
Les valeurs
 propres de $\Pi(g)$ sur $V^{K_{g,n}}$ sont des racines de l'unit\'e 
dans $\mu_{\ell'}$ de somme est \'egale \`a $\phi_n(g)$. 
Par (I.2) et (II.2), on obtient $\X_{\Pi}(g)=\phi(g)$.

\bigskip {\bf II.4 } Nous montrons maintenant que le caract\`ere de Brauer  ci-dessus, 
peut \^etre d\'efini par une filtration quelconque.

\bigskip {\bf Th\'eor\`eme \ }{\sl Soit $g \in G_{\l'}^{ell}$. Pour toute filtration d\'ecroissante
s\'epar\'ee 
$(K_n)_{n \geq 0}$ en pro-p-sous-groupes ouverts compacts de $G$ normalis\'es 
par $g $, il existe un entier $n_g\geq 0$ tel que 
$\phi_{\pi}(g)$ est \'egal au caract\`ere de Brauer de $\pi(g)$ 
sur $V(\pi)^{K_n}$ quelque soit l'entier $n \geq n_g$.}

\bigskip {\sl Preuve \ } Soit $g \in G_{\l'}^{ell}$. On choisit une filtration  $(K_{g,n})$
et un entier $m_1$ comme en (II.1, II.2). On choisit un
entier
$n_g$ tel que 
$K_{n_g} \subset K_{g,m_1}$. Soit $n \geq n_g$, alors $K_{n} \subset K_{ n_{g}}$.
On choisit un entier  $r \geq 1$ tel que 
 $$
K_{g,m_{1} +r} \subset K_n \subset K_{g,m_{1}}.$$
On a alors les d\'ecompositions en somme directe stables par $\pi(g)$:
$$V(\pi)^{K_{g,m_{1}+r}} = V(\pi)^{K_{g,m_{1}}} \oplus W,$$
et 
$$V(\pi)^{K_n} = V(\pi)^{K_{g,m_{1}}} \oplus W^{K_n} $$
o\`u $$W\ =\ \oplus_{\s \in E} \  V(\pi)_{\s} \quad , \quad  W^{K_n}\ =\ \oplus_{\s \in E} \  V(\pi)_{\s}^{K_n} \ ,  $$  o\`u $E$ est l'ensemble fini  des repr\'esentations irr\'eductibles de $K_{n}$ contenues dans $\pi$ 
non triviales sur $K_{g,m_{1}}$ et  triviales sur $K_{g,m_{1} +r}$  (modulo isomorphisme).
 Les $V(\pi)_{\s}^{K_n}$ sont permut\'es  par $\pi(g)$ car $g$ normalise $K_{n}$, et aucun d'entre
 eux n'est stable par $\pi(g)$  par d\'efinition de $m_{1}$. On en d\'eduit comme en (II.2) que le caract\`ere
de Brauer de $\pi(g)$ sur $W^{K_n}$ est nul.  Le caract\`ere de Brauer de $\pi(g)$ sur $V(\pi)^{K_n}$ est donc \'egal au caract\`ere de Brauer de $\pi(g)$ sur $V(\pi)^{K_{g,m_{1}}}$.
 \ $\diamond$.

\bigskip\bigskip \centerline {{\bf III \  Rappels et d\'emonstration des propositions I.1 et II.1}} 
\bigskip Les d\'emonstrations des th\'eor\`emes 1 et 2  ont admis l'existence de certaines filtrations de $G$ et  des propri\'et\'es des caract\`eres de Brauer des endomorphismes finis que  nous v\'erifions  dans cette partie. Nous ne supposons pas   que le centre de $G$ est compact.
  On fixe un isomorphisme de groupes $$\zeta:\Fbar^*\to \mu_{\l'}$$ 
de $\Fbar^*$ dans le groupe
$\mu_{\l'}$ des racines de l'unit\'e d'ordre premier \`a $\l$ dans $\Zbar^*$, qui est 
une section de la r\'eduction modulo $\Lambda$.

\bigskip {\bf III.1 \ } {\sl Caract\`ere de Brauer d'un endomorphisme.}
 Soit $A$ un endomorphisme d'ordre
fini premier \`a $\l$ d'un $\Fbar$-espace vectoriel $W$
 de dimension finie $s$. Alors $A$ est diagonalisable; notons
 $(\alpha_i)_{0\leq i\leq s-1}$ 
 ses valeurs propres, et  
 $\zeta_i=\zeta(\alpha_i) \in \mu_{\l'}$ les rel\`evements de $\alpha_i$ pour $0\leq i\leq s-1$.
Le caract\`ere de Brauer de $A$ sur $W$ est la somme
$$\sum_{0\leq i\leq s-1}\zeta_i.$$

\bigskip  Soit $W'$ un $\Fbar$-sous-espace vectoriel de $W$ stable par 
$A$. Alors 
le caract\`ere de Brauer de $A$ sur $W$ est la somme du 
caract\`ere de Brauer de $A$ sur $W'$ et du caract\`ere de Brauer de 
$A$ sur $W/W'$.

\bigskip  Pour toute $\Fbar$-repr\'esentation $(\pi,W)$ de dimension 
finie d'un groupe fini $H$, le caract\`ere de Brauer des \'el\'ements
$\l$-r\'eguliers (d'ordre premier \`a $\l$) de  $H$ sur $W$ est appel\'e le
  caract\`ere de Brauer de $\pi$ [Serre].

\bigskip {\bf Lemme  \ }{\sl  Soit $s$ un entier $>1$. Supposons qu'il existe $w \in W$ tel que les \'el\'ements $w,A(w),
\ldots, A^{s-1}(w)$ forment une base de $W$, et tel que $A^s(w)$ soit un multiple 
de $w$. Alors le caract\`ere de Brauer de $A$ sur $W$ est nul.}

\bigskip {\sl Preuve \ } 
Il existe $x \in \Fbar^*$ tel que $A^sw = x^s w$ et
  une racine de l'unit\'e $x_s$ 
dans $\Fbar ^*$ d'ordre exactement $s$ car $(s,\l)=1$. 
Les valeurs propres  de $A$ sur $W$ sont 
 $\alpha_i= xx_s^i , \ \ 0\leq i \leq s-1.$ 
Notons $\eta, \eta_s \in \mu_{\l'}$ les rel\`evements de $x, x_s$.
Alors $\zeta_i =\eta \eta_s^i, (0\leq i \leq s-1).$ 
Le caract\`ere de Brauer de $A$ sur $W$
est $\eta \sum_{0\leq i \leq s-1} \eta_s^i=0$ car la somme des 
racines de l'\'equation $X^s-1=0$ dans $\Zbar$ est nulle. 
  \ $\diamond$

\bigskip {\bf III.2 \ } {\sl Elements compacts et $\ell$-r\'eguliers \ } 
 Un \'el\'ement $g$ de $G$ est dit {\sl compact} s'il est contenu 
dans un sous-groupe compact de $G$. Notons $\cl <g>$  la cl\^oture du groupe 
$<g>$ engendr\'e par $g\in G$.
Un \'el\'ement compact $g$ est le produit $$g=hk=kh$$ d'un $\l$-\'el\'ement $k\in \cl <g>$ et d'un 
\'el\'ement $\l$-r\'egulier $h\in \cl <g>$ qui commutent entre eux. 
  La d\'ecomposition 
est unique [SerreSLN5]. 
L'ordre de  $k$ est un entier $\l^a$,
 et le pro-ordre de $h$ est $rp^\infty$ pour un entier $r\geq 1$ premier \`a $\ell$. On dit que
$h$ est {\sl $\l$-r\'egulier} et qu'il est la partie 
{\sl $\l$-r\'eguli\`ere} de $g$.

Le calcul de la d\'ecomposition de $g$ se ram\`ene au cas d'un groupe 
fini; pour tout sous-groupe $K$ d'indice 
fini de $cl<g>$, l'image de $h$ dans le groupe fini $cl<g>/K$ 
 est la partie $\l$-r\'eguli\`ere $h_K$ 
de l'image $g_K$ de $g$ dans $cl<g>/K$. 
On calcule $h_K$ facilement dans le groupe fini $cl<g>/K$, 
en utilisant le th\'eor\`eme de B\'ezout ([Vig] III.2.1). 

Soit maintenant $K$ un sous-pro-$p$-groupe ouvert compact  de $G$ normalis\'e par un \'el\'ement compact ($\ell$-r\'egulier)  $g\in G$.  L'idempotent $$\epsilon_{gK}=1_{gK}{\rm vol} (K,dg)^{-1}$$ de $C^{\infty}_{c}(G;R)$  est d'ordre fini (premier \`a $\ell$) car 
l'ordre de $\epsilon_{gK}$ divise l'ordre de $g$ dans le groupe $cl<g>/(K\cap cl<g>)$ d'ordre fini (premier \`a $\ell$).
 
Si $(\pi,V)$ est une $R$-repr\'esentation lisse admissible de $G$ alors $V^{K}$ est stable par l'endomorphisme $\pi(g)$ et par  $\pi(\epsilon_{gK})$. De plus $\pi(g) = \pi(\epsilon_{gK})$ sur $V^{K}$.
On en d\'eduit:

\bigskip{\bf Lemme} {\sl Si $(\pi,V)$ est une $R$-repr\'esentation lisse admissible de $G$  et
si $K$ est  un sous-pro-$p$-groupe ouvert compact de $G$ normalis\'e par un \'el\'ement compact $\ell$-r\'egulier $g\in G$, alors  $V^{K}$ est un $R$-espace vectoriel de dimension finie et $\pi(g)$ restreint \`a 
$V^{K}$ est un endomorphisme de $V^{K}$ d'ordre fini premier \`a $\ell$.}
 
\bigskip {\bf III.3  \ }  {\sl La th\'eorie de Kirillov-Howe et le lemme de Harish-Chandra}

Le groupe $G$
 agit sur  son alg\`ebre de Lie $\cal G$   par la repr\'esentation adjointe $\rm{Ad}$, qui est  triviale sur le centre $Z$ de $G$.  On choisit un domaine de 
d\'efinition $(G,{\cal G})$-admissible  et  $\rm{Ad}(G)$-stable ${\cal G}_{0}$ d'une application exponentielle $\exp :{\cal G}_{0}\to  G$ ([dBS] \S 10, 17) .  
 
On note $O_F$ l'anneau des entiers de $F$ et 
 $p_F$ un g\'en\'erateur de l'id\'eal maximal de $O_F$.  On note $r$ l'entier \'egal  \`a $e'+1$   o\`u $e'$
   est \'egal \`a l'indice de ramification $e$ de $F/{\bf Q}_{p}$  si $p =3$ et au plus grand entier $\geq  3e/2(p-1)$ si $p\neq 3$.
   
Un $O_F$-r\'eseau $L$
de $\cal G$  sera appel\'e admissible s'il est contenu dans ${\cal G}_{0}$ et si  $[L,L] \subset p_{F}^{r}L$; lorsque  $p=2$ on demande aussi
que $L\subset 2{\cal G}_{0}$ et $[L,L] \subset 2^{3}L$.  On v\'erifie facilement les propri\'et\'es suivantes.

\bigskip {\bf Lemma} {\sl Si $L$ est un  $O_F$-r\'eseau de $\cal G$ alors  $p_{F}^{n}L$ est un $O_F$-r\'eseau admissible
de $\cal G$ pour  tout entier $n$ assez grand.

Soit $L$ un  $O_F$-r\'eseau admissible
de $\cal G$. Alors $p_{F}^{n}L , \ {\rm Ad}(g) L, \ L\cap L'$ sont des $O_F$-r\'eseaux admissibles
de $\cal G$  pour tout entier $n\geq 0$, pour tout \'el\'ement $ g\in G$, et  pour tout  $O_F$-r\'eseau $L'$
admissible de $\cal G$.}

\bigskip On choisit  un $O_F$-r\'eseau $L$
de $\cal G$  admissible. Alors $K_{L}:= \exp (L)$ et $K_{L/2}:=\exp (L/2)$ sont des sous-groupes ouverts compacts de $G$ et $K_{L}$ est un sous-groupe distingu\'e de $K _{L/2}$. 
  Le groupe $K _{L/2}$ agit
naturellement sur $\Irr_R K$. 

On choisit un caract\`ere non trivial $\psi:F\to R^*$ 
(l'image de $\psi$ est form\'ee par les racines de l'unit\'e d'ordre
$p^m$ pour tout entier $m>0$), et
$B$ une forme bilin\'eaire non d\'eg\'en\'er\'ee sur $\cal G$ 
qui est $G$-invariante (qui existe lorsque la caract\'eristique de $F$
est $0$ ou strictement sup\'erieure au nombre de Coxeter de $G$).
 Le dual  de $L$   par rapport \`a $\psi\circ B$ est d\'efini par :
$$L^*=\{Y\in {\cal G} \ | \ \psi(B(X,Y))=1 \ \forall X\in L\} \ . $$
 Il   est stable 
par $\rm{Ad}(L/2)$ donc $K_{L/2}$ agit sur  ${\cal G}/L^{*}$.  On  identifiera un sous-ensemble de ${\cal G}/L^{*}$ avec son image inverse dans ${\cal G}$.
Nous  utiliserons la partie suivante de la th\'eorie de Kirillov-Howe   ([dBS] th. 17.1, cor. 17.2, [Kir] th. 1.1):

\bigskip {\bf Th\'eor\`eme }  (Howe)  {\sl  Pour tout $O_F$-r\'eseau $L$ admissible de $\cal G$ 
il existe une bijection ${\cal O} \mapsto  \s_{\cal O}$
des  $K _{L/2}$-orbites dans ${\cal G}/L^*$
sur les $K _{L/2}$-orbites dans $\Irr_RK_{L}$ telle que

a)  ${\cal O}=L^{*}$ correspond \`a la repr\'esentation triviale de $K_{L} $.

b) Si $g\in G$ normalise $K _{L/2}$ et $K_{L}$ alors $\rm{Ad}(g) {\cal O}$ correspond \`a $g(\s_{\cal O})$.

c)  Deux $O_F$-r\'eseaux  $L, L'$ admissibles arbitraires de $\cal G$ v\'erifient  la propri\'et\'e suivante:

Soient $g\in G$ ,  $ {\cal O}$  une  $K_{L/2}$-orbite dans ${\cal G}/L^*$ , \ $ {\cal O}'$ une  $K_{L'/2}$-orbite dans ${\cal G}/L'^*$. Alors
  $${\cal O} \ \cap \ \rm{Ad}(g) {\cal O}' \ = \ \emptyset$$ 
  iff $g$ n'entrelace pas $ \s_{{\cal O}'}$ avec  $\s_{\cal O}$.}

\bigskip Lorsque $p=2$, seul cas o\`u $K _{L/2} \neq K$, on a not\'e $g(\s_{\cal O})$ l'ensemble des $g(\s)$
pour tout $\s\in \Irr_{R }K_{L}$ dans la $K _{L/2}$-orbite $\cal O$; 
on a convenu que $g\in G$ entrelace $ \s_{{\cal O}'}$ avec  $\s_{\cal O}$ s'il existe
$\s' \in \Irr_{R }K_{L'}$ contenu dans $ \s_{{\cal O}'}$ et $\s \in \Irr_{R }K_{L}$ contenu dans $ \s_{{\cal O}}$
tel que $g$ entrelace $\s'$ avec $\s$, i.e. tel que  les restrictions de $g(\s') \in \Irr_{R }gK_{L'}g^{-1}$ et de $\s\in  \Irr_{R }K_{L}$ \`a $gK_{L'}g^{-1}\cap K_{L}$ ont un composant irr\'eductible isomorphe.
 Le caract\`ere $\X_{\s_{\cal O}}$ de $\s_{\cal O}$ (la somme des caract\`eres des $\s \in \s_{\cal O}$ si $p=2$) est ``la transform\'ee de Fourier'' de $\cal O$:  
  $$\X_{\s_{\cal O}}(\exp (\lambda) )= d_{\cal O}^{-1} \sum _{X\in {\cal O}/L^{*} } \psi (B (X,\lambda)) \quad {\rm pour \ tout } \ \lambda \in L \ , $$  
  o\`u  $d_{\cal O}$ est l'indice dans $K_{L}$ du stabilisateur $\{ k\in K_{L} \  | \ {\rm Ad}(k) X\in X +L^{*}\} $ d'un \'el\'ement arbitraire $X\in \cal O$ et $\psi:F\to R^{*}$ est un caract\`ere additif non trivial fix\'e (\`a valeurs dans le groupe  des racines de l'unit\'e dans $R^{*}$ d'ordre une puissance de $p$).

  \bigskip Les propri\'et\'es a) et b) ne figurent pas explicitement dans la r\'ef\'erence mais se d\'eduisent facilement de la formule donnant le caract\`ere  $\X_{\s_{\cal O}}$. Si ${\cal O}=L^{*}$ le membre de droite \'egal \`a $1$  est le caract\`ere  de la repr\'esentation triviale.  Si $g\in G$ normalise $K _{L/2}$ et $K_{L}$ alors   $g(\sigma_{0})$ est une $K _{L/2}$-orbite  dans $\Irr_RK_{L}$ et 
  $$\X_{g(\s_{\cal O})}(\exp (\lambda) )= \X_{\s_{\cal O}}(g^{-1}\exp (\lambda)g )= \X_{\s_{\cal O}}(\exp ({\rm Ad}(g)^{-1}\lambda) = 
 d_{\cal O}^{-1} \sum _{X\in {\cal O}/L^{*} } \psi (B (X,{\rm Ad}(g)^{-1}\lambda))$$
$$= d_{\cal O}^{-1} \sum _{X\in {\cal O}/L^{*} } \psi (B ( {\rm Ad}(g)X, \lambda)) =  d_{ {\rm Ad} (g){\cal O}}^{-1} \sum _{X\in {{\rm Ad}(g)\cal O}/L^{*} } \psi (B (X, \lambda)) \ . $$

\bigskip Soit  $\cal N$ l'ensemble des \'el\'ements  nilpotents  de $\cal G$.
Les ${\rm Ad}(G)$-orbites ne s'\'eloignent pas  de $\cal N$.  Pour tout  compact $C$
de  $\cal G$, il  existe un $O_{F}$-r\'eseau $L_{C}$ de   $\cal G$ tel que   ([dBS] lemma 12.2, [Fou] lemma 1)
$${\rm Ad}(G) C \subset {\cal N}+ L_{C}  \ .  $$On l'applique \`a $C=L^{*}$ en notant que
 $L_{C} \subset p_{F}^{-n}L^{*}$ si $n$ est assez grand. On  d\'eduit  que  toute $K_{p_{F}^{n}L/2} $-orbite $\cal O$ dans ${\cal G}/ p_{F}^{-n}L^{*}$ qui rencontre ${\rm Ad}(G) L^{*} $ contient un \'el\'ement nilpotent.
 Ces orbites correspondent dans la th\'eorie de Kirillov-Howe (propri\'et\'e a)) aux $K_{p_{F}^{n}L/2} $-orbites dans $\Irr_RK_{p_{F}^{n}L}$ entrelac\'ees avec la repr\'esentation triviale de $K_{L}$ dans $G$.

\bigskip  Soit $\gamma\in G$  semi-simple.
On note $M=Z_G(\gamma)^o$ la composante connexe du centralisateur de 
$\gamma$ dans $G$. C'est un groupe r\'eductif comme $G$. 
 On note  ${\cal M}$ l'alg\`ebre de Lie de $M$. On a une d\'ecomposition 
 $${\cal G}={\cal M} \oplus ({\rm Ad}(\gamma) -1)   {\cal G}$$
 $B$-orthogonale ([dBS] \S 18). On note $p_{\cal M}$ la projection sur $\cal M$.  Soit $M_{\gamma}$ l'ensemble des $m\in M$ tels que $\rm{det } ({\rm Ad}(\gamma m)-1)|_{{\cal G}/{\cal M}} \neq 0$. Alors $M_{\gamma}$ est ouvert  ${\rm Ad}(M)$-stable et contient $1$; l'application 
 $$G\times M_{\gamma} \to G$$
 $$(x,m)\mapsto x\gamma m x^{-1}$$
 est submersive [HC162]. L'image d'un ouvert est ouverte.
   On choisit un   domaine de 
d\'efinition $(M,{\cal M})$-admissible  et  $\rm{Ad}(M)$-stable ${\cal M}_{0} \subset {\cal M} \cap {\cal G}_{0}$ d'une application exponentielle $\exp :{\cal M}_{0}\to  M$,    restriction d'une application exponentielle   $\exp :{\cal G}_{0}\to  G$. Ceci d\'etermine la notion de $O_F$-r\'eseau admissible   dans $\cal M$.

\bigskip  On fixe un $O_F$-r\'eseau de r\'ef\'erence dans $\cal G$, ce qui d\'efinit une valuation   
 sur $\cal G$ de boule unit\'e ${\cal G}(0)$ \'egale \`a ce r\'eseau.
 Le choix du r\'eseau ne compte pas vraiment
puisque deux r\'eseaux sont commensurables.   Pour tout $n\in \Z$, on note ${\cal G}(n)= p_{F}^{n}{\cal G}(0) $ l'ensemble des \'el\'ements de
valuation $\geq n$.  On note ${\cal N}_0$ l'ensemble des \'el\'ements nilpotents de valuation nulle, et pour
 tout entier $a>0$, on consid\`ere le voisinage de ${\cal N}_0$ 
 $${\cal V}_a:={\cal N}_0+ {\cal G}(a)\  $$
contenu dans la boule unit\'e ${\cal G}(0)$.   On note ${\cal M}_a$ le $F$-espace vectoriel engendr\'e par ${\cal V}_a \cap {\cal M}  $.

\bigskip 
Enoncons maintenant le lemme d'Harish-Chandra ([HC] Lemme 38) ([dBS] 21.3).

\bigskip {\bf Lemme  \ } (Harish-Chandra) {\sl   Soit $\gamma \in G$ semi-simple et $M=Z_{G}(\gamma)^{0} $. Il existe des $O_{F}$-r\'eseaux admissibles $L$ de $\cal G$ et $L'$ de $\cal M$ et un entier $c>0$  v\'erifiant la propri\'et\'e  (HC) : 

\medskip Si $a$ est un entier $>0$, pour tout entier $n\geq n_a$ assez grand,
et pour toute  $K_ {p_{F}^{n}L/2}$-orbite 
 ${\cal O}$ dans ${\cal G}/p_{F}^{-n}L^*$ on a 
 $$p_{\cal M}({\cal O}) \ \subset \  {\cal M}_{a}$$
 si ${\cal O} $ v\'erifie les trois conditions suivantes : 

1)  ${\cal O} $ contient un \'el\'ement nilpotent,

2) ${\cal O} $ contient un \'el\'ement de valuation $ \ < \ - \ c \ n$,

3)  ${\cal O} $ contient un \'el\'ement de ${\rm Ad}(\gamma K_{L'}) ({\cal O} ) $. 

Cette propri\'et\'e reste vraie pour tous les $O_{F}$-r\'eseaux admissibles $L_{1}$ de $\cal G$ et $L'_{1}$ de $\cal M$  plus petits que $L$ et $L'$.
 }

\bigskip Dans le dictionnaire de Kirillov-Howe,  la troisi\`eme condition dit que la $K_ {p_{F}^{n}L/2}$-orbite 
de $\Irr_{R}K_ {p_{F}^{n}L}$ correspondant \`a ${\cal O}$ est entrelac\'ee avec elle-m\^eme par un \'el\'ement de ${\rm Ad}(\gamma K_{L'}) $.  Lorsque $\gamma$ est semi-simple r\'egulier le centralisateur connexe  $M$ de $\gamma$ est un tore,
  ${\cal M}$ ne contient aucun nilpotent.   Les  parties compactes   ${\cal N}_0 $ et $ {\cal M} \cap {\cal G}(0)$  de $\cal G$ ayant une intersection vide,   ${\cal M}_a = \emptyset $
 pour $a$ assez grand. On obtient:

\bigskip {\bf Lemme \ }   {\sl Si $\gamma \in G^{reg}$  
la propri\'et\'e (HC) du lemme d'Harish-Chandra est \'equivalente \`a :     
  
    Pour tout $n $ assez grand, aucune  $K_ {p_{F}^{n}L/2}$-orbite 
 ${\cal O}$ dans ${\cal G}/p_{F}^{-n}L^*$ ne peut  v\'erifier les trois conditions 1), 2), 3).
  }

\bigskip {\bf III.4  \ }  {\sl    D\'emonstration de la proposition I.1  }

 Soit $L$ un  $O_{F}$-r\'eseau admissible  de $\cal G$. Nous allons montrer
que la suite $(K_{n})_{n\geq 0}$ d\'efinie par $K_{n}:= K_ {p_{F}^{n}L}$ v\'erifie la proposition I.1.
Il est clair que c'est une filtration s\'epar\'ee d\'ecroissante de pro-$p$-sous-groupes ouverts de $G$. 
La condition de finitude  fait intervenir un entier $n_{0}\geq 0$.  On se ram\`ene \`a $n_{0}=0$ en remplacant $L$ par $p_{F}^{n_{0}}L$ ce qui all\`ege les notations.

 Soit $\gamma \in G^{reg}$    et $M$ son centralisateur connexe dans $G$. On choisit un entier $n_{1} \geq 0$ et   un  $O_{F}$-r\'eseau admissible $L'$  de $\cal M$   tel que $p_{F}^{n_{1}}L$ et  $L'$ v\'erifient le lemme de Harish-Chandra, et  $\gamma K_{L'} \subset G^{reg}$.  On se ram\`ene \`a $n_{1}=0$ en remplacant $L$ par $p_{F}^{n_{1}}L$. On pose
$$V(\gamma):= \cup_{k\in K_{ L }} k\ \gamma  K_{L'} \ k^{-1} \  .$$
C'est un voisinage ouvert de $\gamma$ dans $G^{reg}$. 
On va montrer que  pour $n$  tr\`es grand,   il n'existe qu'un nombre fini de  $K_{p_{F}^{n}L/2} $-orbites dans $\Irr_RK_{p_{F}^{n}L}$ entrelac\'ees avec la repr\'esentation triviale $I_{L}$ de $K_{L}$  et avec elle-m\^emes par un \'el\'ement  de $V(\gamma)$. Ceci est une assertion un peu plus forte que la condition de finitude de la proposition 1.

 Supposons $n$ tr\`es grand. Par le dictionnaire de Kirillov-Howe les $K_{p_{F}^{n}L/2} $-orbites dans $\Irr_RK_{p_{F}^{n}L}$ correspondent \`a des  $K_{p_{F}^{n}L/2} $-orbites $\cal O$  dans ${\cal G}/ p_{F}^{-n}L^{*}$. Comme nous l'avons d\'eja remarqu\'e, la propri\'et\'e  ``entrelac\'ee avec   $I_{L}$  '' implique $\cal O$ contient un \'el\'ement nilpotent,
et la propri\'et\'e ``entrelac\'ee avec elle-m\^eme par un \'el\'ement  de $\gamma  K_{L'}$ ''
 est \'equivalente \`a  $\cal O$ contient un \'el\'ement de $ {\rm Ad}(\gamma K_{L'}) ({\cal O} ) $.
  Le lemme de Harish-Chandra pour $\gamma \in G^{reg}$  implique que les orbites $\cal O$ ayant ces deux propri\'et\'e ne contiennent pas un \'el\'ement de valuation $<- cn$ pour un $c>0$ fix\'e.  Elles sont donc en nombre fini.

Pour terminer la d\'emonstration on   remplace   $\gamma  K_{L'}$  par l'ouvert $V(\gamma)$ de $G$. Le groupe
$K_{L}$ normalise les groupes $K_{p_{F}^{n}L/2} $ et $K_{p_{F}^{n}L} $. Les  images  par  $ K_{L}$ des   $K_{p_{F}^{n}L/2} $-orbites dans $\Irr_RK_{p_{F}^{n}L}$   entrelac\'ees avec $I_{L}$ et 
 avec elle-m\^emes par un \'el\'ement  de $\gamma  K_{L'}$  sont \'egales aux $K_{p_{F}^{n}L/2} $-orbites dans $\Irr_RK_{p_{F}^{n}L}$  
  entrelac\'ees avec $I_{L}$ et avec elle-m\^emes par un \'el\'ement  de $V(\gamma)$.
   Elles  correspondent par les propri\'et\'es a) et   b)   du  dictionnaire de Kirillov-Howe aux images par ${\rm Ad}(K_{L})$  de certaines  $K_{p_{F}^{n}L/2} $-orbites $\cal O$ dans ${\cal G}/p_{F}^{-n}L^{*} $ contenant un \'el\'ement nilpotent et un \'el\'ement de   $ {\rm Ad}(\gamma K_{L'}) ({\cal O} ) $.  Nous avons vu que le nombre de ces orbites $\cal O$ est  fini, et  il est \'evident que l'image de $\cal O$ par ${\rm Ad}(K_{L}) $ est un ensemble fini d'orbites. Donc le nombre $K_{p_{F}^{n}L/2} $-orbites dans $\Irr_RK_{p_{F}^{n}L}$ entrelac\'ees avec   $I_{L}$  et avec elle-m\^emes par un \'el\'ement  de $V(\gamma)$ est fini.

\bigskip {\bf III.5  \ }  {\sl    D\'emonstration de la proposition II.1  }

  Un tore maximal $T$ de $G$ est appel\'e elliptique,  s'il est compact modulo le centre $Z$ de $G$. Posons $T^{reg}:= T \cap G^{reg}$. Alors  l'application
$$(x,t)\to xtx^{-1}:
\ G \times T^{reg}\to G$$ est submersive. Son image est un ouvert que l'on notera $G.T^{reg}$.
 L'ouvert $G^{ell}$ des \'el\'ements elliptiques de $G$, est une union disjointe  
d'ouverts   de la forme $G.T^{reg}$ pour  $T$ appartenant \`a l'ensemble fini  $\{T_{1}, \ldots, T_{r}\}$ de  tores maximaux  elliptiques    de $G$ modulo $G$-conjugation.

Soit $g\in G^{ell}$ et soit $T\in \{T_{1}, \ldots, T_{r}\}$ tel que $g\in G.T^{reg}$.
Le normalisateur $N_{G}(T)$  de $T$ dans $G$  est aussi compact modulo $Z$.
Il  existe une classe unique $xN_G(T)$  telle que $g\in xT^{reg}x^{-1}$.

  Soit $L$ un  $O_{F}$-r\'eseau admissible  de $\cal G$ qui est stable par $\rm{Ad}(N_{G}(T))$. Il suffit de partir 
d'un $O_F$-r\'eseau admissible quelconque   $L'$  de $\cal G$ et de prendre
  $$L \ := \  \cap_{n\in N_{G}(F)}\ \rm{Ad}(n) L' \ .$$
C'est   un $O_F$-r\'eseau admissible de   $\cal G$  car   le nombre des  $O_F$-r\'eseaux $\rm{Ad}(n) L'  $ pour $n \in N_{G}(F)$ est fini.   Il est clairement $N_{G}(F)$-stable.  
  La suite $(K_{g,n})_{n\geq 0}$ 
 d\'efinie par $$K_{g,n}:= K_{p_{F}^{n}{\rm Ad}(x)L}=xK_{p_{F}^{n}L}x^{-1}$$ est une
 filtration s\'epar\'ee d\'ecroissante de pro-$p$-sous-groupes ouverts de $G$ normalis\'es par $g$ qui ne d\'epend pas du choix de $x$.
Nous   montrons maintenant que ces filtrations  $(K_{g,n})_{n\geq 0}$ pour  tout $g\in G^{ell}$ v\'erifient la proposition II.1.

a) L'application $g'\mapsto K_{g',n}$ est constante sur le voisinage   $V(g):=\cup_{y\in xK_{L}} yT^{reg}y^{-1}$  de $g$ dans $G$ car   $K_{L}$ normalise $K_{p_{F}^{n}L}$ donc
 $K_{',n}= yK_{p_{F}^{n}L}y^{-1}=K_{g,n}$ pour tout $g'\in yT^{reg}y^{-1}\ , \ y\in xK_{L}$.  
 
b) On a $y K_{g,n}  y^{-1} = yxK_{p_{F}^{n}L}(xy)^{-1} = K_{ygy^{-1},n}$ pour tout $y\in G$ par d\'efinition 
de $K_{g,n}$.

c) Le r\'eseau ${\rm Ad}(x)L$ est admissible. Par le preuve de la proposition I.1 la filtration $K_{p^{n}{\rm Ad}(x)L}=K_{g,n}$ v\'erfie la condition de finitude c).

\bigskip\bigskip {\bf Bibliographie }

[HC1]  Harish-Chandra. Admissible invariant distributions on reductive p-adic groups, dans : Collected
Papers 
II, Springer 1984, 371-438. 

[dBS]  Harish-Chandra. Admissible invariant distributions on reductive p-adic groups, dans Notes by DeBacker and Sally 1997. University Lecture series  16 (1999).

[HC162]  Harish-Chandra. Notes by G. van Dijk. Harmonic Analysis on Reductive p-adic Groups. Lecture
Notes in Mathematics 162 Springer-Verlag 1970.

[Fou] Howe R. The Fourier transform and germs of characters. Math. Ann. 208 (1974) 305-322.

[Kir] Howe R. Kirillov Theory for compact p-adic groups. Pacific J. of Math. Vo. 73, No. 2. 1997.

[Serre] Serre J.-P. Repr\'esentations lin\'eaires des groupes finis. Hermann, deuxi\`eme \'edition 1971.

[SerreSLN5] Serre J.-P.  Cohomologie galoisienne. Lecture
Notes in Mathematics 5. Springer-Verlag, cinqui\`eme \'edition 1994.

[Vig] Vign\'eras M.-F. Repr\'esentations $\l$-modulaires d'un groupe r\'eductif p-adique
avec $\l \neq p$. Progress in Math 137. Birkhauser 1996.

\end{document}